\documentclass[final,leqno]{amsart}
\usepackage{mathbbol}
\usepackage{mathrsfs}
\usepackage{amsfonts}

\usepackage{amsthm,mathrsfs}

\usepackage{amsfonts}
\usepackage{stmaryrd}
\usepackage{amsmath}
\usepackage{amssymb}
\usepackage{multirow}
\usepackage{cite}
\usepackage{caption}
\usepackage{tabularx}
\usepackage{stmaryrd}
\usepackage{booktabs}
\usepackage{cite}
\usepackage{verbatim}
\usepackage{color}
\usepackage{graphicx}
\usepackage{subfigure}
\usepackage{hyperref}
\usepackage{bm}
\input psfig.sty

\theoremstyle{plain}
\newtheorem{theorem}{Theorem} [section]
\newtheorem{lemma}[theorem]{Lemma}

\newtheorem{proposition}[theorem]{Proposition}

\theoremstyle{remark}
\newtheorem{remark}[theorem]{Remark}

\theoremstyle{definition}

\numberwithin{equation}{section}

\allowdisplaybreaks

\newcommand{\p}{\partial}
\newcommand{\og}{\omega}

\newcommand{\Og}{\Omega}
\newcommand{\ep}{\varepsilon}
\newcommand{\fl}[2]{\frac{#1}{#2}}

\newcommand{\nn}{\nonumber}

\newcommand{\be}{\begin{equation}}
\newcommand{\ee}{\end{equation}}
\newcommand{\ba}{\begin{array}}
\newcommand{\ea}{\end{array}}
\newcommand{\bea}{\begin{eqnarray}}
\newcommand{\eea}{\end{eqnarray}}
\newcommand{\beas}{\begin{eqnarray*}}
\newcommand{\eeas}{\end{eqnarray*}}
\newcommand{\bx}{{\bf x} }

\def\R{{\mathbb R}}% real numbers
\def\({\left(}
\def\){\right)}
\def\<{\left\langle}
  \def\>{\right\rangle}

\begin{document}

\title[Energy regularization for logNLS]{Error estimates of local energy regularization for the logarithmic Schr\"odinger equation}
  \author[W. Bao]{Weizhu Bao}
 \address{Department of Mathematics,
National University of Singapore, Singapore 119076}
\email{matbaowz@nus.edu.sg}

\author[R. Carles]{R\'{e}mi Carles}
\address{Univ Rennes, CNRS, IRMAR - UMR 6625, F-35000 Rennes, France}
\email{ Remi.Carles@math.cnrs.fr}

\author[C. Su]{Chunmei Su}
\address{Yau Mathematical Sciences Center, Tsinghua University, Beijing 100084, China}
\email{sucm@tsinghua.edu.cn}

\author[Q. Tang]{Qinglin Tang}
\address{School of Mathematics, State Key Laboratory of Hydraulics and
  Mountain River Engineering, Sichuan University, Chengdu 610064,
  People's Republic of China}
\email{qinglin\_tang@scu.edu.cn}

\begin{abstract}
The logarithmic nonlinearity has been used in many partial differential equations (PDEs) for modeling problems in various applications.
Due to the singularity of the logarithmic function, it introduces
tremendous difficulties in establishing mathematical theories, as well as
in designing and analyzing numerical methods for PDEs with such nonlinearity. Here we take the logarithmic Schr\"odinger equation (LogSE)
as a prototype model. Instead of regularizing $f(\rho)=\ln \rho$ in the
LogSE directly and globally as being done in the literature, we propose a local energy regularization (LER) for the LogSE by
first regularizing $F(\rho)=\rho\ln \rho -\rho$ locally near $\rho=0^+$ with a polynomial approximation in the energy functional of the LogSE and then obtaining an energy regularized logarithmic Schr\"odinger equation (ERLogSE) via energy variation. Linear convergence is established between the solutions of ERLogSE and LogSE in terms of a small regularization parameter $0<\ep\ll1$. Moreover, the conserved energy of the ERLogSE converges to that of LogSE quadratically, which significantly improves
the linear convergence rate of the regularization method in the literature. Error estimates are also
presented for solving the ERLogSE by using Lie-Trotter splitting
integrators. Numerical results are reported to confirm our error
estimates of the LER and of the time-splitting integrators for the
ERLogSE. Finally our results suggest that the LER performs better than regularizing the logarithmic nonlinearity in the LogSE directly.
\end{abstract}



\thanks{This work was partially supported by the Ministry
of Education of Singapore grant R-146-000-296-112
(MOE2019-T2-1-063) (W. Bao), Rennes M\'etropole through its AIS
program (R. Carles), the Alexander von Humboldt 
Foundation (C. Su), the Institutional Research Fund from Sichuan
University (No. 2020SCUNL110) and the National Natural Science
Foundation of China (No. 11971335) (Q. Tang).} 

\maketitle

\section{Introduction}
The logarithmic nonlinearity appears in physical models from many
fields. For example, the logarithmic nonlinearity is introduced in quantum mechanics or quantum optics, where a
logarithmic Schr\"odinger equation (LogSE) is considered
(e.g. \cite{BiMy76, BiMy79, buljan, KEB00}),
\[
  i\p_t u=-\Delta u+\lambda\, u\ln |u|^2,\quad \lambda \in \R;
  \]
in oceanography and in fluid
dynamics, with a logarithmic Korteweg-de Vries (KdV) equation or a
logarithmic Kadomtsev-Petviashvili (KP) equation (e.g. \cite{wazwaz2014,
  wazwaz2016, james2014}); in quantum field theory and in inflation
cosmology, via a logarithmic Klein-Gordon equation (e.g. \cite{rosen1969,
  bartkowski2008, gorka2009}); or in material sciences, by the
introduction of a Cahn-Hilliard (CH) equation with
logarithmic potentials (e.g. \cite{cherfils2011, gilardi2009, elliott1996}). Recently, the heat equation with a logarithmic nonlinearity has been investigated mathematically \cite{chen2015, alfaro2017}.

In the context of quantum mechanics, the logarithmic nonlinearity was
selected by assuming the separability of noninteracting subsystems
property (cf. \cite{BiMy76}). This means that a solution of the nonlinear equation for the whole system can be constructed, as in the linear theory, by taking the product of two arbitrary solutions of the nonlinear equations for the subsystems. In other words, no correlations are introduced for noninteracting subsystems. As for the physical reality, robust physical grounds have been found for the application of equations with logarithmic nonlinearity. For instance, it was found in the stochastic formulation of quantum mechanics \cite{lemos1983, nassar1985} that the logarithmic nonlinear term originates naturally from an internal stochastic force due to quantum fluctuations. Such kind of nonlinearity also
appears naturally in inflation cosmology and in supersymmetric field theories \cite{barrow1995, enqvist1998q}.

Remarkably enough for a nonlinear PDE, many explicit solutions are
available for the logarithmic
mechanics (see e.g.  \cite{BiMy76,koutvitsky2006}). For example, the logarithmic
KdV equation, the logarithmic KP equation, the logarithmic
Klein-Gordon equation give Gaussons: solitary wave solutions with
Gaussian shapes \cite{wazwaz2014, wazwaz2016}.
In the case of LogSE (see \cite{CaGa18,ferriere-p1}), or the heat
equation \cite{alfaro2017}, every initial Gaussian function evolves as
a Gaussian: solving the corresponding nonlinear PDE is
equivalent to solving ordinary differential equations (involving the
purely time dependent parameters of the Gaussian).  However we
emphasize that this is not so in
the case of, e.g.,  the logarithmic
KdV equation, the logarithmic KP equation, or the logarithmic
Klein-Gordon equation. This can be directly seen by trying to plug time
dependent Gaussian functions into these equations. Note that this
distinction between various PDEs regarding the propagation of
Gaussian functions is the same as at  the linear level.

The well-posedness of the Cauchy problem for logarithmic equations is
not trivial since the logarithmic nonlinearity is not locally
Lipschitz continuous, due to the singularity of the logarithm at the
origin. Existence was proved by compactness argument  based on
regularization of the nonlinearity, for the CH equation with a logarithmic
potential \cite{elliott1991} and the LogSE
\cite{cazenave1983}. Uniqueness is also a challenging question,
settled in the case of LogSE thanks to a surprising inequality
discovered in \cite{CaHa80}, recalled in Lemma~\ref{pre} below.

The singularity of the logarithmic
nonlinearity also makes it  very challenging to design and analyze numerical schemes. There have been extensive numerical works for the CH
equation with a logarithmic Flory Huggins energy potential
\cite{copetti1992, gokieli2003, jeong2016, jeong2017, yang2019,
  chen2019}. Specifically, a regularized energy functional was adopted
for the CH equation with a logarithmic free energy \cite{copetti1992,
  yang2019}. A regularization of the logarithmic nonlinearity was introduced and
analyzed in \cite{bao2018, bao2019error} in the case LogSE,  see also
\cite{li2019}.

In this paper, we  introduce and analyze numerical methods for
logarithmic equations via a local energy regularization. We consider the
LogSE as an example; the regularization can be extended to other
logarithmic equations. The LogSE which arises in a model of nonlinear
wave mechanics reads (cf. \cite{BiMy76}),
\be\label{LSE}
\left\{
\begin{aligned}
&i\p_t u(\bx,t)=-\Delta u(\bx,t)+\lambda\, u(\bx,t)\,f(|u(\bx,t)|^2),\quad \bx\in \Omega, \quad t>0,\\
& u(\bx,0)=u_0(\bx),\quad \bx\in \overline{\Omega},
\end{aligned}
\right.
\ee
where $t$ and $\bx\in \mathbb{R}^d$ ($d=1,2,3$) represent the temporal and spatial coordinates, respectively, $\lambda\in \mathbb{R}\backslash\{0\}$ measures the force of the nonlinear interaction,
$u:=u(\bx,t)\in\mathbb{C}$ is the dimensionless wave function, and
\be\label{frhoSE}
f(\rho)=\ln \rho, \qquad \rho>0, \qquad \hbox{with}\quad \rho=|u|^2.
\ee
The spatial domain is either $\Omega=\mathbb{R}^d$, or $\Omega\subset\mathbb{R}^d$ bounded with  Lipschitz continuous boundary; in
the latter case the equation is subject to homogeneous Dirichlet or periodic boundary
conditions.
This model has been widely applied in quantum mechanics, nuclear physics,
geophysics, open quantum systems and Bose-Einstein condensation, see e.g.
\cite{Hef85, yasue, HeRe80, de2003, BEC}.  We choose to consider positive time
only merely to simplify the presentation, since \eqref{LSE} is time reversible. Formally, the flow of
\eqref{LSE} enjoys two important conservations. The {\sl mass}, defined as
\be\label{massSE}
N(t):=N(u(\cdot,t))=\|u\|^2=\int_\Omega |u(\bx,t)|^2d\bx
\equiv N(u_0), \qquad t\ge0,
\ee
and the {\sl energy}, defined as
\be\label{conserv}
\begin{split}
E(t):&=E(u(\cdot,t))=\int_\Omega\left[|\nabla u(\bx,t)|^2d\bx+\lambda F(|u(\bx,t)|^2)\right]d\bx\\
&\equiv\int_\Omega\left[|\nabla u_0(\bx)|^2+\lambda F(|u_0(\bx)|^2)\right]d\bx=E(u_0), \qquad t\ge0,
\end{split}
\ee
where
\be \label{Frho345}
F(\rho)=\int_0^\rho f(s)ds=\int_0^\rho \ln s\, ds=\rho\,\ln \rho-\rho, \qquad \rho\ge0.
\ee
The total angular momentum is also conserved, an identity that we do
not use in the present paper.
For the Cauchy problem \eqref{LSE} in a suitable
functional framework, we refer to \cite{CaHa80, CaGa18,GLN10}. For stability properties of standing waves for \eqref{LSE}, we refer to
\cite{cazenave1982, cazenave1983, Ar16}. For the analysis of breathers
and the existence of multisolitons, see
\cite{ferriere-p1,ferriere-p2}.

In order to avoid numerical blow-up of the logarithmic nonlinearity at
the origin, two models of regularized logarithmic Schr\"odinger equation (RLogSE)
were proposed in \cite{bao2019error}, involving a direct
regularization of $f$ in \eqref{frhoSE}, relying on a small regularized parameter $0<\ep\ll1$,
\be\label{RLSE0}
\left\{
\begin{aligned}
&i\p_t u^\ep(\bx,t)=-\Delta u^\ep(\bx,t)+\lambda \, u^\ep(\bx,t)\,\widetilde{f}^\ep(|u^\ep(\bx,t)|)^2),\quad \bx\in \Omega, \quad t>0,\\
&u^\ep(\bx,0)=u_0(\bx),\quad \bx\in \overline{\Omega},
\end{aligned}
\right.
\ee
and
\be\label{RLSE1}
\left\{
\begin{aligned}
&i\p_t u^\ep(\bx,t)=-\Delta u^\ep(\bx,t)+\lambda \, u^\ep(\bx,t)\,\widehat{f}^\ep(|u^\ep(\bx,t)|^2)),\quad \bx\in \Omega, \quad t>0,\\
&u^\ep(\bx,0)=u_0(\bx),\quad \bx\in \overline{\Omega}.
\end{aligned}
\right.
\ee
Here, $\widetilde{f}^\ep(\rho)$ and $\widehat{f}^\ep(\rho)$
are two types of regularization for $f(\rho)$,  given by
\be \label{1str2}
\widetilde{f}^\ep(\rho)=2\ln (\ep+\sqrt{\rho}),\quad
\widehat{f}^\ep(\rho)=\ln(\ep^2+\rho),\quad \rho\ge0, \qquad
\hbox{with}\quad \rho=|u^\ep|^2.
\ee
Again, the RLogSEs \eqref{RLSE0} and \eqref{RLSE1} conserve
the mass \eqref{massSE} with $u=u^\ep$, as well as the {\sl energies}
\be\label{conserv1}
\widetilde{E}^\ep(t):=\widetilde{E}^\ep
(u^\ep(\cdot,t))=\int_\Omega\left[|\nabla u^\ep(\bx,t)|^2d\bx+\lambda
\widetilde{F}^\ep(|u^\ep(\bx,t)|^2)\right]d\bx
\equiv \widetilde{E}^\ep(u_0),
\ee
and
\be\label{conserv2}
\widehat{E}^\ep(t):=\widehat{E}^\ep
(u^\ep(\cdot,t))=\int_\Omega\left[|\nabla u^\ep(\bx,t)|^2d\bx+\lambda
\widehat{F}^\ep(|u^\ep(\bx,t)|^2)\right]d\bx
\equiv \widehat{E}^\ep(u_0),
\ee
respectively, with, for $\rho\ge 0$,
\be\label{1str}
\begin{split}
\widetilde{F}^\ep(\rho)&=\int_0^\rho \widetilde{f}^\ep(s)ds
=2\rho\ln(\ep+\sqrt{\rho})+2\ep\sqrt{\rho}-\rho-2
\ep^2\ln(1+\sqrt{\rho}/\ep),\\
\widehat{F}^\ep(\rho)&=\int_0^\rho \widehat{f}^\ep(s)ds=(\ep^2+\rho)\ln(\ep^2+\rho)-\rho-2\ep^2\ln \ep.
\end{split}
\ee
The idea of this regularization is that the function $\rho\mapsto \ln \rho$ causes no (analytical or numerical) problem for large values of $\rho$,  but is singular at $\rho=0$.
A linear convergence was established between the solutions of the LogSE \eqref{LSE} and the regularized model \eqref{RLSE0} or \eqref{RLSE1}  for bounded $\Omega$ in terms of the small regularization parameter $0<\ep\ll1$, i.e.,
\[\sup_{t\in [0,T]}\|u^\ep(t) -u(t)\|_{L^2(\Omega)}=O(\ep),\quad \forall\
  T>0.\]
Applying this regularized model, a semi-implicit finite difference method (FDM) and a time-splitting method were proposed and analyzed
for the LogSE \eqref{RLSE0}  in \cite{bao2019error}  and \cite{bao2018}
 respectively. The above regularization saturates
the nonlinearity in the region $\{\rho<\ep^2\}$ (where
$\rho=|u^\ep|^2$), but of course has also some (smaller) effect in the
other region $\{\rho>\ep^2\}$, i.e., it regularizes $f(\rho)=\ln \rho$ globally.

Energy regularization is a method which has been adapted in different
fields
for dealing with singularity and/or roughness: in materials
science, for establishing the well-posedness of the Cauchy problem for the CH equation with a logarithmic potential \cite{elliott1991}, and
for treating strongly anisotropic surface energy \cite{Jiang,BaoJ}; in mathematical physics, for the well-posedness of
the LogSE \cite{cazenave1983}; in scientific computing, for designing
regularized numerical methods in the presence of singularities \cite{copetti1992, yang2019,BaoR}. The main goal of this paper is to present a local energy regularization ({\sl LER})
for the LogSE \eqref{LSE}. We
regularize the interaction energy density $F(\rho)$ only locally in the region
$\{\rho<\ep^2\}$  by a sequence of polynomials, and keep it unchanged in
$\{\rho>\ep^2\}$. The choice of the regularized interaction energy density $F_n^\ep$ is
prescribed by the regularity $n$ imposed at this step, involving the matching
conditions at $\{\rho=\ep^2\}$. We then
obtain a sequence of energy regularized logarithmic Schr\"odinger
equations (ERLogSEs), from the regularized energy functional density
$F_n^\ep$, via energy variation.
Unlike in \cite{copetti1992,yang2019}, where the interaction energy density $F(\rho)$ is approximated by a second order
polynomial near the origin, here we present a systematic way to regularize
the interaction energy density near the origin, i.e. locally,  by a sequence of polynomials such
that the order of regularity $n$  of the overall regularized interaction energy density is arbitrary.
We establish convergence rates between the solutions of ERLogSEs and LogSE in terms of the small regularized parameter $0<\ep\ll1$.
In addition, we also prove  error estimates of numerical approximations
of ERLogSEs by using time-splitting integrators.

The rest of this paper is organized as follows. In Section~\ref{sec:regul}, we
introduce a sequence of regularization $F_n^\ep$ for the logarithmic
potential. A regularized model is derived and analyzed in Section~\ref{sec:regulLSE}
via the LER of the LogSE. Some numerical methods are
proposed and analyzed in Section~\ref{sec:lie}. In
Section~\ref{sec:num}, we present numerical experiments. Throughout
the paper, we adopt the
standard $L^2$-based Sobolev spaces as well as the corresponding
norms, and denote by $C$ a generic positive constant independent of $\ep$, the
time step $\tau$ and the
function $u$, and by $C(c)$ a generic positive constant depending on $c$.

\section{Local regularization for $F(\rho)=\rho\,\ln \rho-\rho$}
\label{sec:regul}
We consider a local regularization starting from an approximation to the interaction energy density $F(\rho)$ in \eqref{Frho345} (and thus in \eqref{conserv}).
\subsection{A sequence of local regularization}
In order to make a comparison with the former global regularization
\eqref{RLSE0}, we again distinguish the regions
$\{\rho>\ep^2\}$ and $\{\rho<\ep^2\}$.  Instead of saturating the
nonlinearity in the second region, we regularize it locally as follows. For an
arbitrary integer $n\ge2$, we  approximate $F(\rho)$ by a piecewise
smooth function which is polynomial near the origin,
\be\label{Fn}
F^\ep_n(\rho)=F(\rho)\chi_{\{\rho\ge \ep^2\}}+P^\ep_{n+1}(\rho)\chi_{\{\rho<\ep^2\}}, \quad n\ge 2,
\ee
where $0<\ep\ll1$ is a small regularization parameter, $\chi_{_A}$ is the characteristic
function of the set $A$, and $P^\ep_{n+1}$ is a polynomial of degree
$n+1$. We demand $F^\ep_n \in C^n([0,+\infty))$ and
$F^\ep_n(0)=F(0)=0$ (this allows the regularized energy to be
well-defined on the whole space). The above conditions determine
$P_{n+1}^\ep$, as we now check. Since $P^\ep_{n+1}(0)=0$, write
\be\label{PQ}
P^\ep_{n+1}(\rho)=\rho\, Q_n^\ep(\rho),
\ee
with $Q^\ep_n$ a polynomial of degree $n$. Correspondingly, denote
$F(\rho)=\rho\, Q(\rho)$ with $Q(\rho)=\ln \rho-1$. The continuity
conditions read
\[P^\ep_{n+1}(\ep^2)=F(\ep^2), \quad (P^\ep_{n+1})'(\ep^2)=F'(\ep^2), \quad\ldots, \quad (P^\ep_{n+1})^{(n)}(\ep^2)=F^{(n)}(\ep^2),\]
which in turn yield
\[Q^\ep_{n}(\ep^2)=Q(\ep^2), \quad (Q^\ep_{n})'(\ep^2)=Q'(\ep^2),
  \quad\ldots, \quad (Q^\ep_{n})^{(n)}(\ep^2)=Q^{(n)}(\ep^2).\]
Thus $Q^\ep_n$ is nothing else but Taylor polynomial of $Q$ of degree
$n$ at $\rho=\ep^2$, i.e.,
\be\label{Qd}
Q^\ep_n(\rho)=Q(\ep^2)+\sum\limits_{k=1}^n  \fl{Q^{(k)}(\ep^2)}{k!}(\rho-\ep^2)^k=
\ln \ep^2-1 -\sum\limits_{k=1}^n \fl{1}{k}\left(1-\fl{\rho}{\ep^2}\right)^k.
\ee
In particular, Taylor's formula yields
\begin{equation}\label{eq:Taylor}
 Q(\rho)-  Q^\ep_n(\rho)=\int_{\ep^2}^\rho
 Q^{(n+1)}(s)\frac{(\rho-s)^n}{n!}ds = \int_{\ep^2}^\rho
   \frac{(s-\rho)^n}{s^{n+1}}ds  .
\end{equation}
Plugging \eqref{Qd} into \eqref{PQ},  we get the explicit formula of
$P_{n+1}^\ep(\rho)$. We emphasize a formula which will be convenient
for convergence results:
\begin{equation}
  \label{eq:Q'}
  \(Q^\ep_n\)'(\rho)= \frac{1}{\ep^2}\sum\limits_{k=1}^n
  \left(1-\fl{\rho}{\ep^2}\right)^{k-1}= \frac{1}{\rho}\(
  1-\(1-\fl{\rho}{\ep^2}\)^n\), \qquad 0\le \rho \le \ep^2.
\end{equation}

\subsection{Properties of the local regularization functions}
Differentiating \eqref{Fn} with respect to $\rho$ and noting
\eqref{PQ}, \eqref{Qd} and \eqref{eq:Q'}, we get
\be\label{fep}
f_n^\ep(\rho)=(F_n^\ep)'(\rho)=\ln \rho \,\chi_{\{\rho\ge \ep^2\}}+q^\ep_n(\rho)\chi_{\{\rho<\ep^2\}}, \qquad \rho\ge0,
\ee
where
\begin{align*}
q^\ep_n(\rho)&=(P_{n+1}^\ep)'(\rho)=Q_n^\ep(\rho)+\rho\, (Q_n^\ep)'(\rho)\\
&=\ln (\ep^2)-\frac{n+1}{n}\left(1-\fl{\rho}{\ep^2}\right)^n-\sum\limits_{k=1}^{n-1} \fl{1}{k}\left(1-\fl{\rho}{\ep^2}\right)^k.
\end{align*}
Noticing that $q^\ep_n$ is increasing in $[0, \ep^2]$, $\widetilde{f}^\ep$ and $\widehat{f}^\ep$ are increasing on $[0, \infty)$, thus all three types of regularization \eqref{Fn} and \eqref{1str} preserve the convexity of $F$. Moreover, as a sequence of local regularization (or approximation) for the semi-smooth function $F(\rho)\in C^0([0,\infty))\cap C^\infty((0,\infty))$, we have $F^\ep_n \in C^n([0,+\infty))$ for $n\ge2$,
while $\widetilde{F}^\ep \in C^1([0,\infty))\cap C^\infty((0,\infty))$ and $\widehat{F}^\ep \in C^\infty([0,\infty))$. Similarly, as a sequence of local regularization (or approximation) for the logarithmic function $f(\rho)=\ln \rho\in C^\infty((0,\infty))$, we observe that $f_n^\ep\in C^{n-1}([0, \infty))$ for $n\ge 2$, while $\widehat{f}^\ep\in C^\infty([0, \infty))$ and $\widetilde{f}^\ep \in C^0([0,\infty))\cap C^\infty((0,\infty))$.

Recall the following lemma, established initially in \cite[Lemma~1.1.1]{CaHa80}.
\begin{lemma}\label{pre}
For $z_1, z_2\in\mathbb{C}$, we have
\[\left|\mathrm{Im}\left(\(z_1\ln (|z_1|^2)-z_2\ln (|z_2|^2)\)
(\overline{z_1}-\overline{z_2})\right)\right|\le
  2|z_1-z_2|^2, \]
where $\mathrm{Im}(z)$ and $\overline{z}$ denote the imaginary part
and the complex conjugate of $z$, respectively.
\end{lemma}
Next we highlight some properties of $f_n^\ep$.
\begin{lemma}
Let $n\ge 2$ and $\ep>0$. For $z_1$, $z_2\in\mathbb{C}$, we have
\begin{align}
&|f_n^\ep(|z_1|^2)-f_n^\ep(|z_2|^2)|\le\fl{4n|z_1-z_2|}{\max\{\ep,\min\{|z_1|, |z_2|\}\}},\label{fl}\\
&\left|\mathrm{Im}\left[\(z_1f_n^\ep(|z_1|^2)-z_2f_n^\ep(|z_2|^2)\)
(\overline{z_1}-\overline{z_2})\right]\right|\le
  4n|z_1-z_2|^2,\label{gl}\\
&|\rho (f_n^\ep)'(\rho)|\le 3,\quad |\sqrt{\rho} (f_n^\ep)'(\rho)|\le \fl{2n}{\ep},\quad |\rho^{3/2} (f_n^\ep)''(\rho)|\le \fl{3n^2}{2\ep},\quad \rho\ge0,\label{fd}\\
&|f_n^\ep(\rho)|\le \max\{|\ln A|, 2+\ln(n\ep^{-2})\}, \quad \rho\in [0, A].\label{fb}
%&\left|z_1^2(f_n^\ep)'(|z_1|^2)-z_2^2(f_n^\ep)'(|z_2|^2)\right|\le \fl{C(n)}{\ep}|z_1-z_2|,\label{fd1}\\
%&\left||z_1|^2(f_n^\ep)'(|z_1|^2)-|z_2|^2(f_n^\ep)'(|z_2|^2)\right|\le \fl{C(n)}{\ep}|z_1-z_2|,\label{fd2}
  \end{align}
\end{lemma}
%where $C(n)>0$ is a general constant depending on $n$, but independent
%of $\ep\in (0,1]$. %already in the convention at the end of the introduction

\begin{proof}
When $|z_1|, |z_2|\ge\ep$, we have
\[
\left|f_n^\ep(|z_1|^2)-f_n^\ep(|z_2|^2)\right|=2\ln\Big(1+\fl{\left||z_1|-|z_2|\right|
}{\min\{|z_1|,|z_2|\}}\Big)\le \fl{2|z_1-z_2|}{\min\{|z_1|,|z_2|\}}.
\]
A direct calculation gives
\be\label{fp1}
(f_n^\ep)'(\rho)=\fl{1}{\rho}\chi_{\{\rho\ge \ep^2\}}+\left(\fl{n}{\ep^2}
\big(1-\fl{\rho}{\ep^2}\big)^{n-1}+
\fl{1}{\ep^2}\sum\limits_{k=0}^{n-1}\big(1-\fl{\rho}{\ep^2}\big)^{k}\right)\chi_{\{\rho< \ep^2\}}.
\ee
Thus when $|z_1|<|z_2|\le\ep$, we have
\begin{align*}
\left|f_n^\ep(|z_1|^2)-f_n^\ep(|z_2|^2)\right|&=
\int_{|z_1|^2}^{|z_2|^2} (f_n^\ep)'(\rho)d\rho\\
&=\frac{n}{\ep^2}\int_{|z_1|^2}^{|z_2|^2}
\big(1-\fl{\rho}{\ep^2}\big)^{n-1}d\rho+\frac{1}{\ep^2}\sum\limits_{k=0}^{n-1}
\int_{|z_1|^2}^{|z_2|^2}\big(1-\fl{\rho}{\ep^2}\big)^{k}d\rho\\
&\le \frac{n}{\ep^2}(|z_2|^2-|z_1|^2)+\frac{1}{\ep^2}\sum\limits_{k=0}^{n-1}(|z_2|^2-|z_1|^2)\\
&=\frac{2n}{\ep^2}(|z_2|^2-|z_1|^2)\le \frac{4n}{\ep}|z_1-z_2|.
\end{align*}
Another case when $|z_2|<|z_1|\le \ep$ can be established similarly.
Supposing, for example, $|z_2|<\ep< |z_1|$, denote by $z_3$ the intersection point of the circle $\{z\in \mathbb{C}: |z|=\ep\}$ and the line segment connecting $z_1$ and $z_2$. Combining the inequalities above, we have
\begin{align*}
\left|f_n^\ep(|z_1|^2)-f_n^\ep(|z_2|^2)\right|&\le |f_n^\ep(|z_2|^2)-f_n^\ep(|z_3|^2)|+
|\ln(|z_1|^2)-\ln(|z_3|^2)|\\
&\le \fl{4n}{\ep}|z_2-z_3|+
\fl{2}{\ep}|z_1-z_3|\\
&\le \fl{4n}{\ep}\left(|z_2-z_3|+|z_1-z_3|\right)=\fl{4n}{\ep}|z_1-z_2|,
\end{align*}
which completes the proof for \eqref{fl}.

Noticing that
\begin{align*}
&\mathrm{Im}\left[\(z_1f_n^\ep(|z_1|^2)-z_2f_n^\ep(|z_2|^2)\)
(\overline{z_1}-\overline{z_2})\right]\\
&\quad=-\mathrm{Im}(\overline{z_1}
z_2) f_n^\ep(|z_2|^2)-\mathrm{Im}(z_1\overline{z_2})f_n^\ep(|z_1|^2)\\
&\quad=\mathrm{Im}(\overline{z_1}
z_2)\left[f_n^\ep(|z_1|^2)- f_n^\ep(|z_2|^2)\right]\\
&\quad=\fl{1}{2i}(\overline{z_1}z_2-z_1\overline{z_2})
\left[f_n^\ep(|z_1|^2)-f_n^\ep(|z_2|^2)\right],
\end{align*}
and
\begin{align*}
&\left|\overline{z_1}z_2-z_1\overline{z_2}\right|=
\left|z_2(\overline{z_1}-\overline{z_2})+\overline{z_2}(z_2-z_1)\right|\le
 2|z_2|\,|z_1-z_2|,\\
 &\left|\overline{z_1}z_2-z_1\overline{z_2}\right|=
\left|\overline{z_1}(z_2-z_1)+z_1(\overline{z_1}-\overline{z_2})\right|\le
 2|z_1|\,|z_1-z_2|,
\end{align*}
which implies
\[\left|\overline{z_1}z_2-z_1\overline{z_2}\right|\le 2\min\{|z_1|, |z_2|\}\,|z_1-z_2|,\]
one can conclude \eqref{gl} by applying \eqref{fl}.

It follows from \eqref{fp1} that
\begin{align*}
g(\rho)&=\rho(f_n^\ep)'(\rho)=\chi_{\{\rho\ge \ep^2\}}+\left(\fl{n\rho}{\ep^2}
\big(1-\fl{\rho}{\ep^2}\big)^{n-1}+
\fl{\rho}{\ep^2}\sum\limits_{k=0}^{n-1}\big(1-\fl{\rho}{\ep^2}\big)^{k}\right)\chi_{\{\rho< \ep^2\}}\\
&=\chi_{\{\rho\ge \ep^2\}}+\left(\fl{n\rho}{\ep^2}
\big(1-\fl{\rho}{\ep^2}\big)^{n-1}+1-\big(1-\fl{\rho}{\ep^2}\big)^{n}
\right)\chi_{\{\rho< \ep^2\}},
\end{align*}
which gives that
\[g'(\rho)\chi_{\{\rho< \ep^2\}}=\frac{n}{\ep^2}\big(1-\fl{\rho}{\ep^2}\big)^{n-2}
\left[2-\frac{(n+1)\rho}{\ep^2}\right].\]
This leads to
\[\left|\rho(f_n^\ep)'(\rho)\right|=g(\rho)\le \max\{1, g\left(\frac{2\ep^2}{n+1}\right)\}\le 1+\frac{2n}{n+1}\le 3,\]
which completes the proof for the first inequality in \eqref{fd}. Finally it
follows from \eqref{fp1} that
\begin{align*}
&\sqrt{\rho}(f_n^\ep)'(\rho)=\fl{1}{\sqrt{\rho}}\chi_{\{\rho\ge \ep^2\}}+\fl{\sqrt{\rho}}{\ep^2}\left(n
\big(1-\fl{\rho}{\ep^2}\big)^{n-1}+
\sum\limits_{k=0}^{n-1}\big(1-\fl{\rho}{\ep^2}\big)^{k}\right)\chi_{\{\rho< \ep^2\}},\\
&(f_n^\ep)''(\rho)=-\fl{1}{\rho^2}\chi_{\{\rho\ge \ep^2\}}-\left(\fl{n^2-1}{\ep^4}\big(1-\fl{\rho}{\ep^2}\big)^{n-2}+\fl{1}{\ep^4}
\sum\limits_{k=0}^{n-3}(k+1)\big(1-\fl{\rho}{\ep^2}\big)^{k}\right)\chi_{\{\rho< \ep^2\}},
\end{align*}
which immediately yields that
\begin{align*}
&|\sqrt{\rho}(f_n^\ep)'(\rho)|\le \frac{2n}{\ep},\\
&|\rho^{3/2}(f_n^\ep)''(\rho)|\le \frac{1}{\ep}\left(n^2-1+
\sum\limits_{k=0}^{n-3}(k+1)\right)=\frac{3n(n-1)}{2\ep}<\frac{3n^2}{2\ep}.
\end{align*}

For $\rho\in [0, \ep^2]$, in view of $\ep\in (0, 1]$, one deduces
\begin{align*}
|f_n^\ep(\rho)|&\le \ln (\ep^{-2})+\frac{n+1}{n}\left(1-\fl{\rho}{\ep^2}\right)^n+\sum\limits_{k=1}^{n-1} \fl{1}{k}\left(1-\fl{\rho}{\ep^2}\right)^k\\
&\le \ln (\ep^{-2})+\frac{n+1}{n}+\sum\limits_{k=1}^{n-1} \fl{1}{k}\\
&\le \ln (\ep^{-2})+2+\sum\limits_{k=2}^{n} \fl{1}{k}\\
&\le 2+\ln (n\ep^{-2}),
\end{align*}
which together with $|f_n^\ep(\rho)|\le \max\{\ln(\ep^{-2}), |\ln(A)|\}$ when $\rho\in [\ep^2, A]$ concludes \eqref{fb}.
\end{proof}

\subsection{Comparison between different regularizations}
To compare different regularizations for $F(\rho)$ (and thus for $f(\rho)$),
Fig. ~\ref{Fcomp} shows $F_n^\ep$ ($n=2,4,100,500$), $\widetilde{F}^\ep$ and $\widehat{F}^\ep$ for different $\ep$,
from which we can see that the newly proposed local regularization $F_n^\ep$ approximates $F$ more accurately.

\begin{figure}[h!]
\begin{center}
\includegraphics[width=6cm,height=4cm]{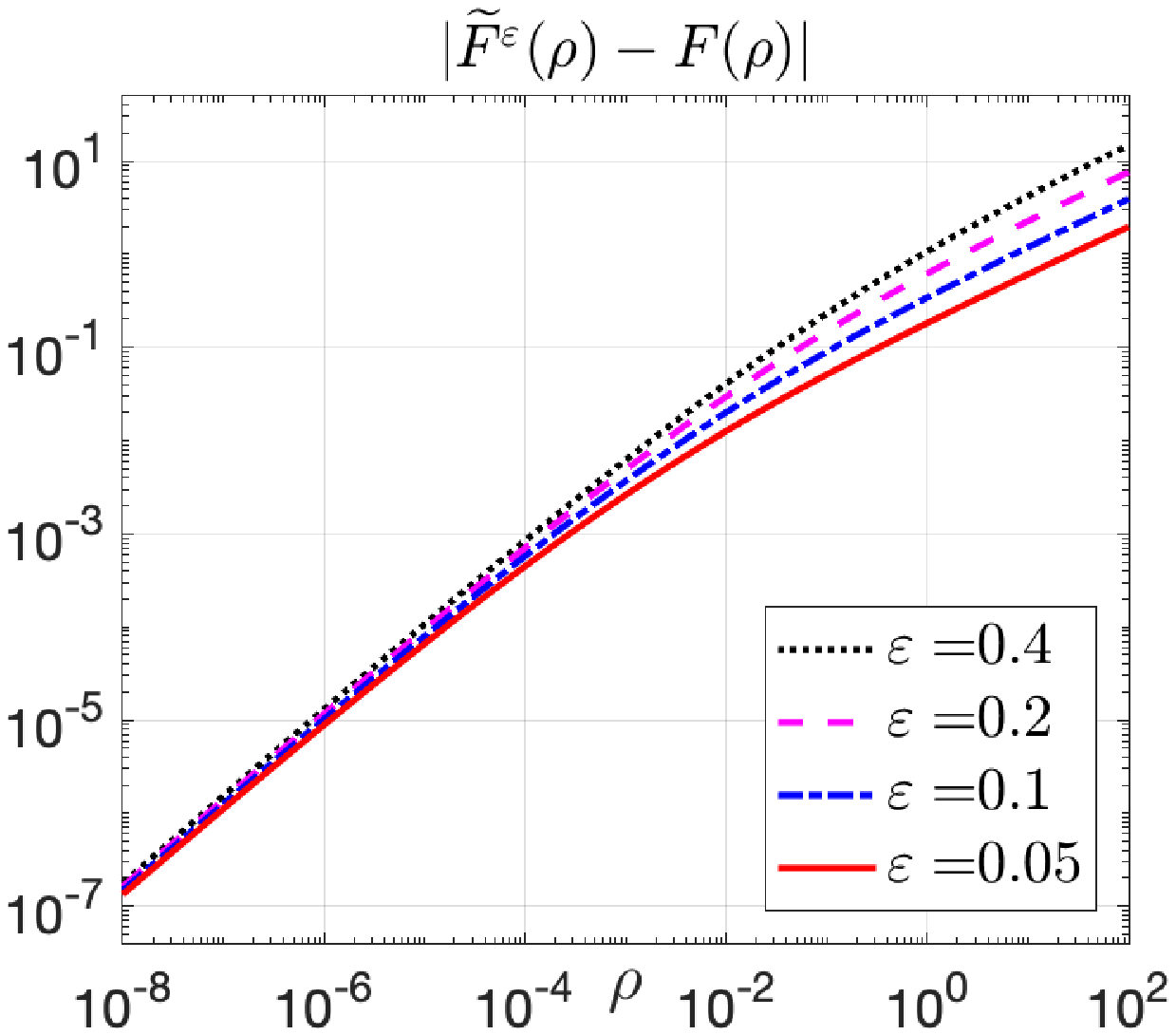}
\quad
\includegraphics[width=6cm,height=4cm]{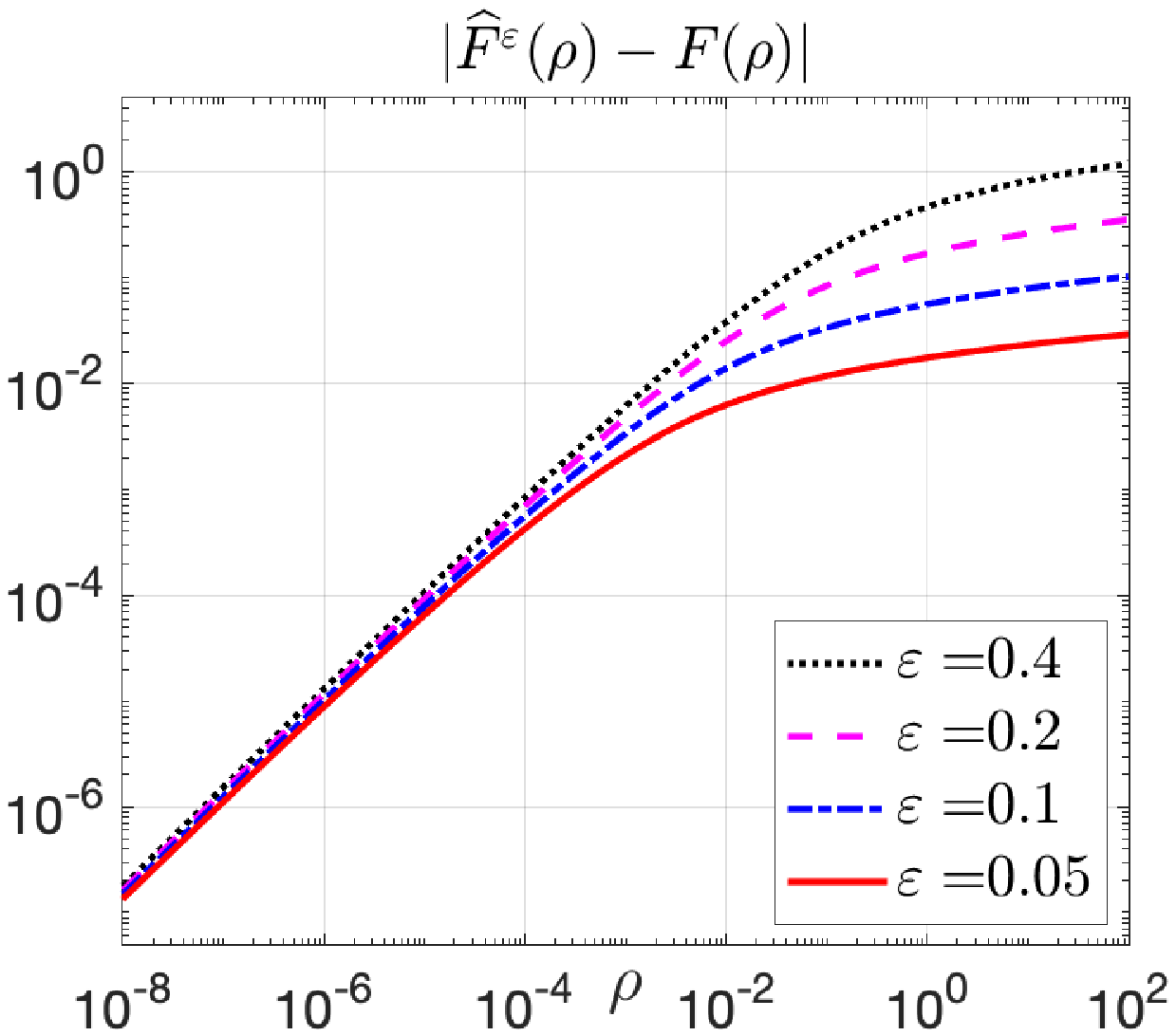}\\[1em]
\includegraphics[width=6cm,height=4cm]{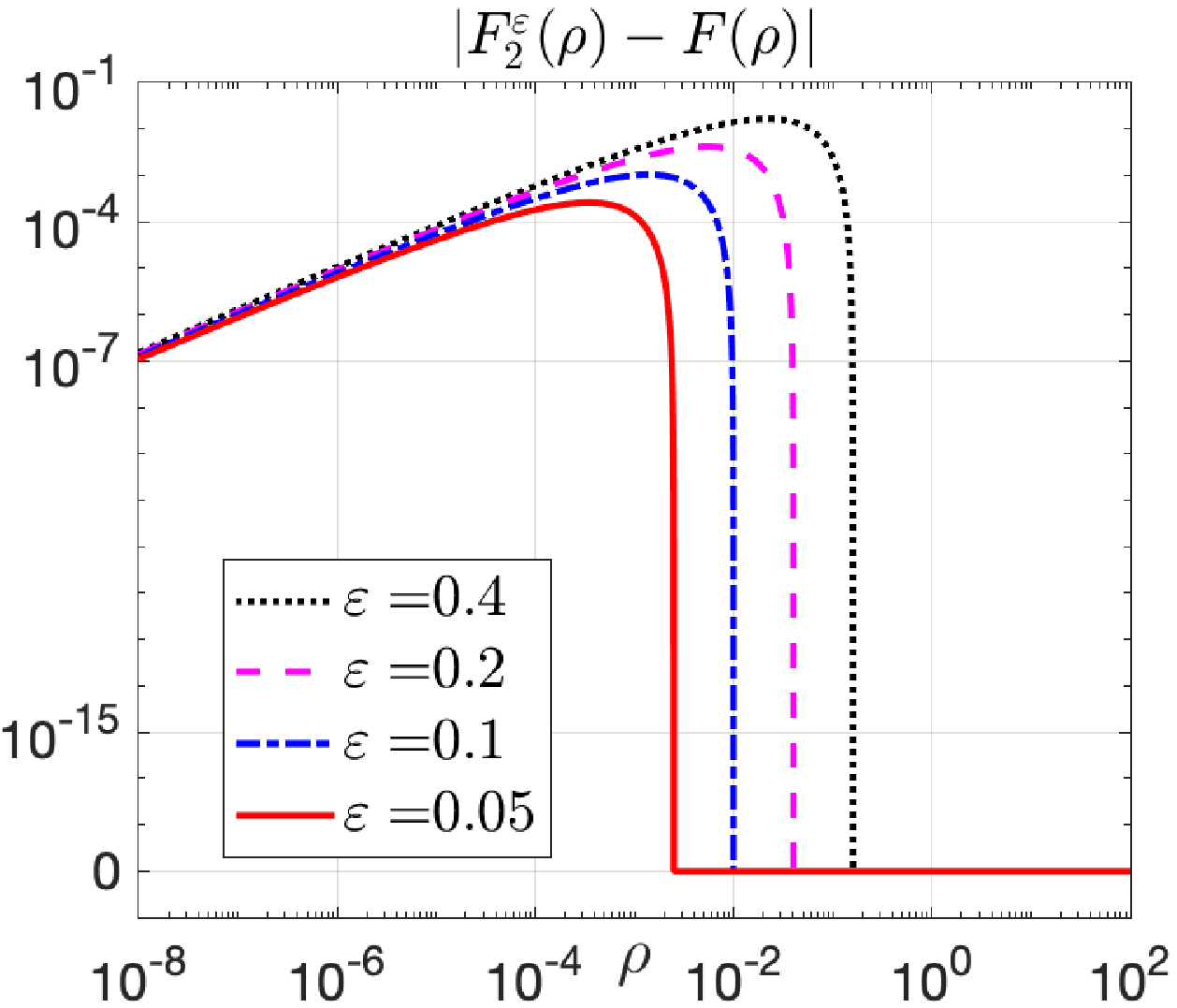}
\quad
\includegraphics[width=6cm,height=4cm]{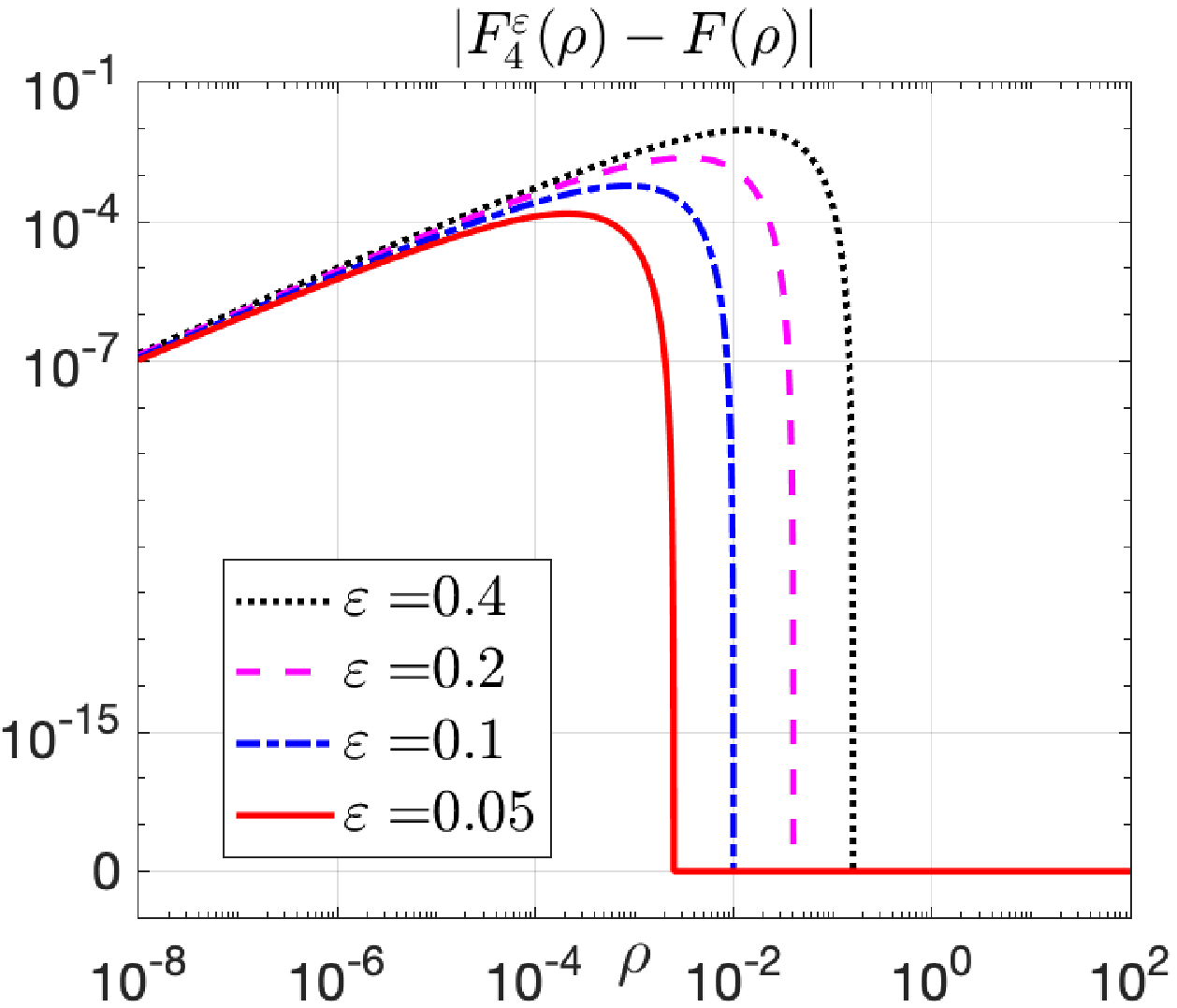}\\[1em]
\includegraphics[width=6cm,height=4cm]{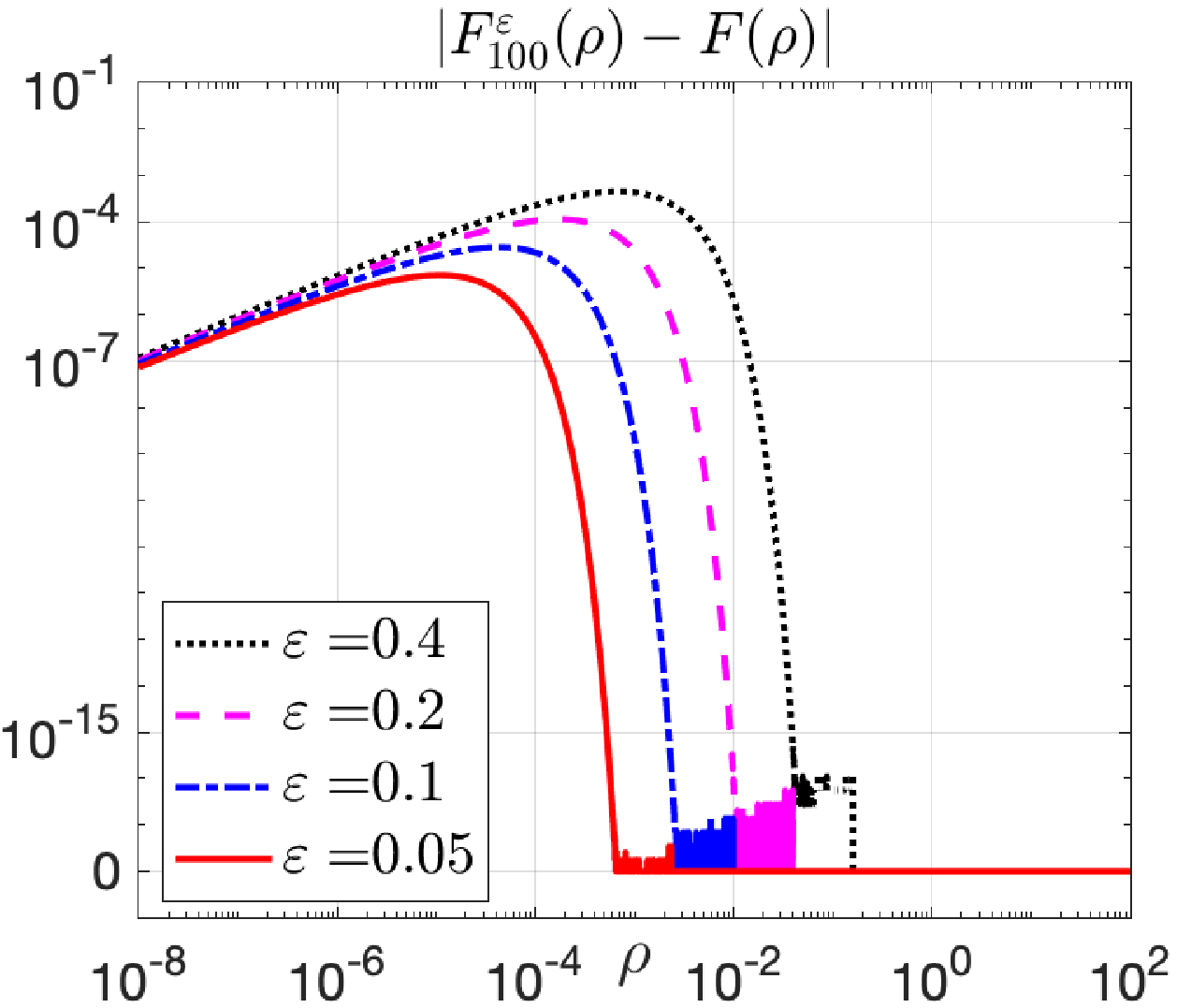}
\quad
\includegraphics[width=6cm,height=4cm]{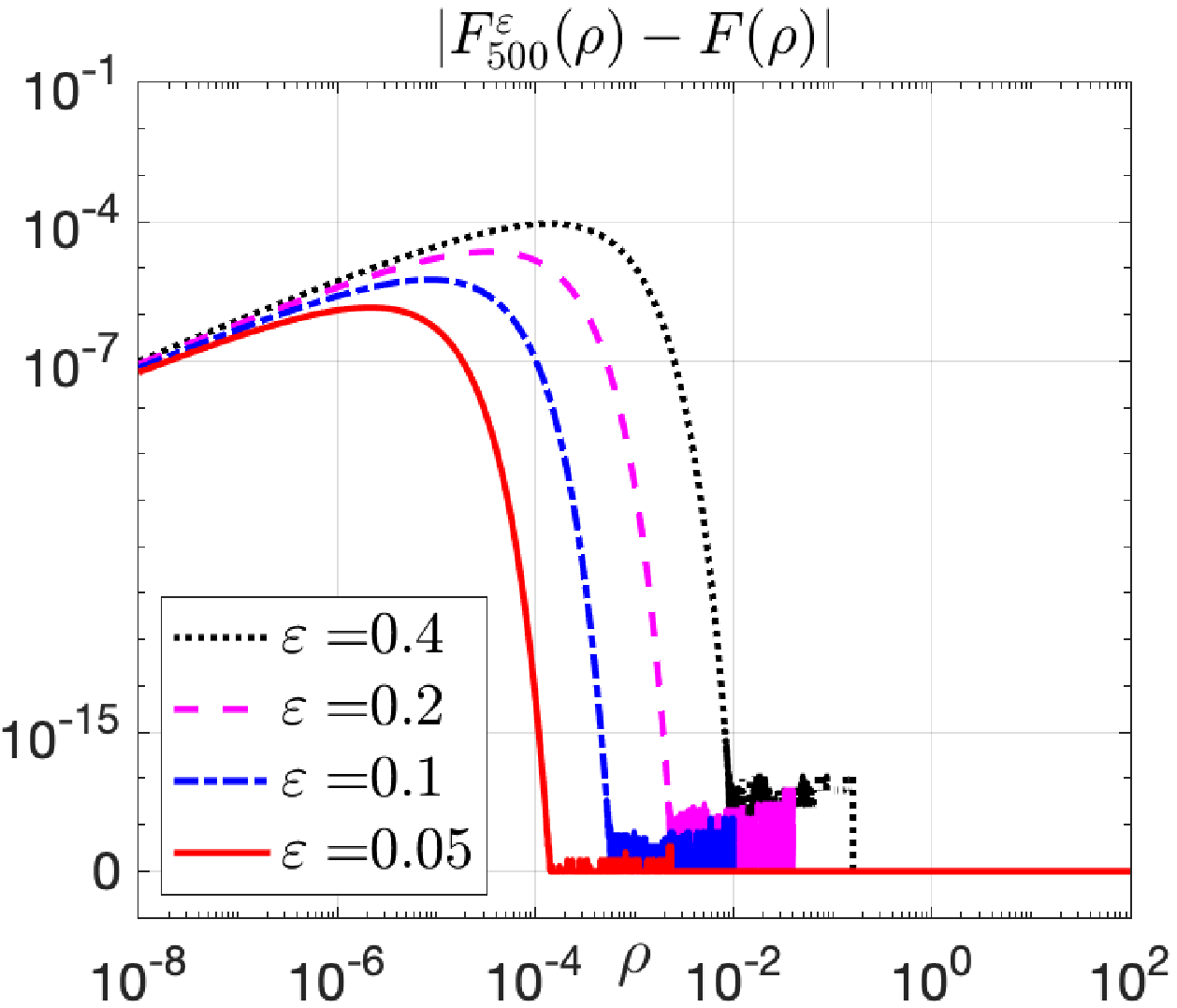}
\end{center}
 \caption{Comparison of different regularizations for $F(\rho)=\rho\ln\rho-\rho$. }
\label{Fcomp}
\end{figure}

Fig. \ref{fcomp} shows various regularizations $f_n^\ep$ ($n=2,4,100,500$), $\widetilde{f}^\ep$ and $\widehat{f}^\ep$
for various $\ep$, while Figs. \ref{fprime} \& \ref{fpp} show  their first- and second-order derivatives.
From these figures, we can see that the newly proposed local regularization $f_n^\ep$ (and its derivatives with larger $n$)
approximates the nonlinearity $f$ (and its derivatives) more accurately. In addition, Fig. \ref{fig:Revision_EnergyReguL_Conv_WRT_Order_N} depicts  $F_n^\varepsilon(\rho)$ (with $\varepsilon=0.1$) and its   derivatives for different $n$, from which we can clearly see the convergence of $F_n^\varepsilon(\rho)$ (and its derivatives)
to $F(\rho)$ (and its derivatives) W.R.T. order $n$.

\begin{figure}[htbp!]
\begin{center}
\includegraphics[width=6cm,height=4cm]{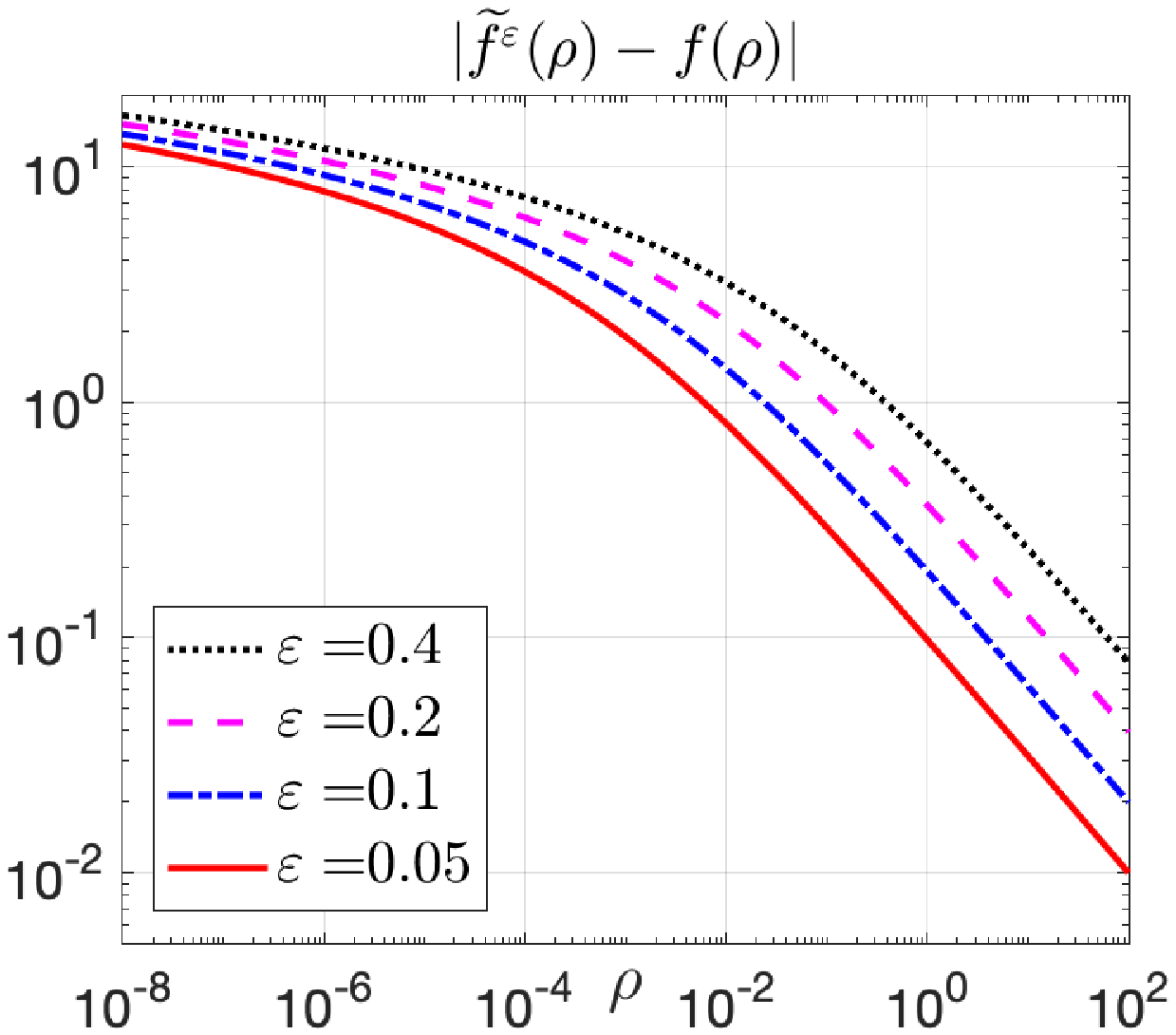}
\quad
\includegraphics[width=6cm,height=4cm]{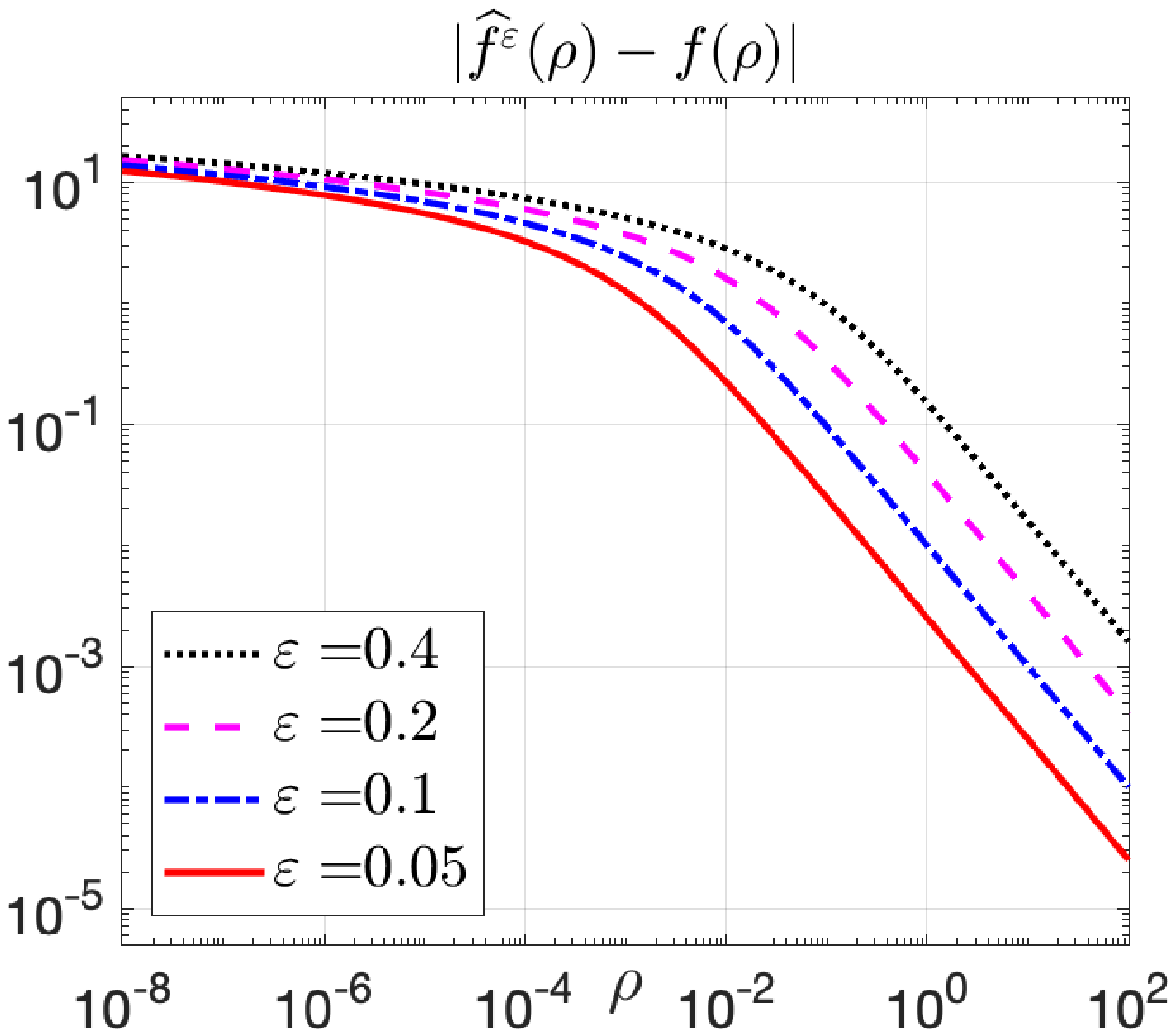}\\[1em]
\includegraphics[width=6cm,height=4cm]{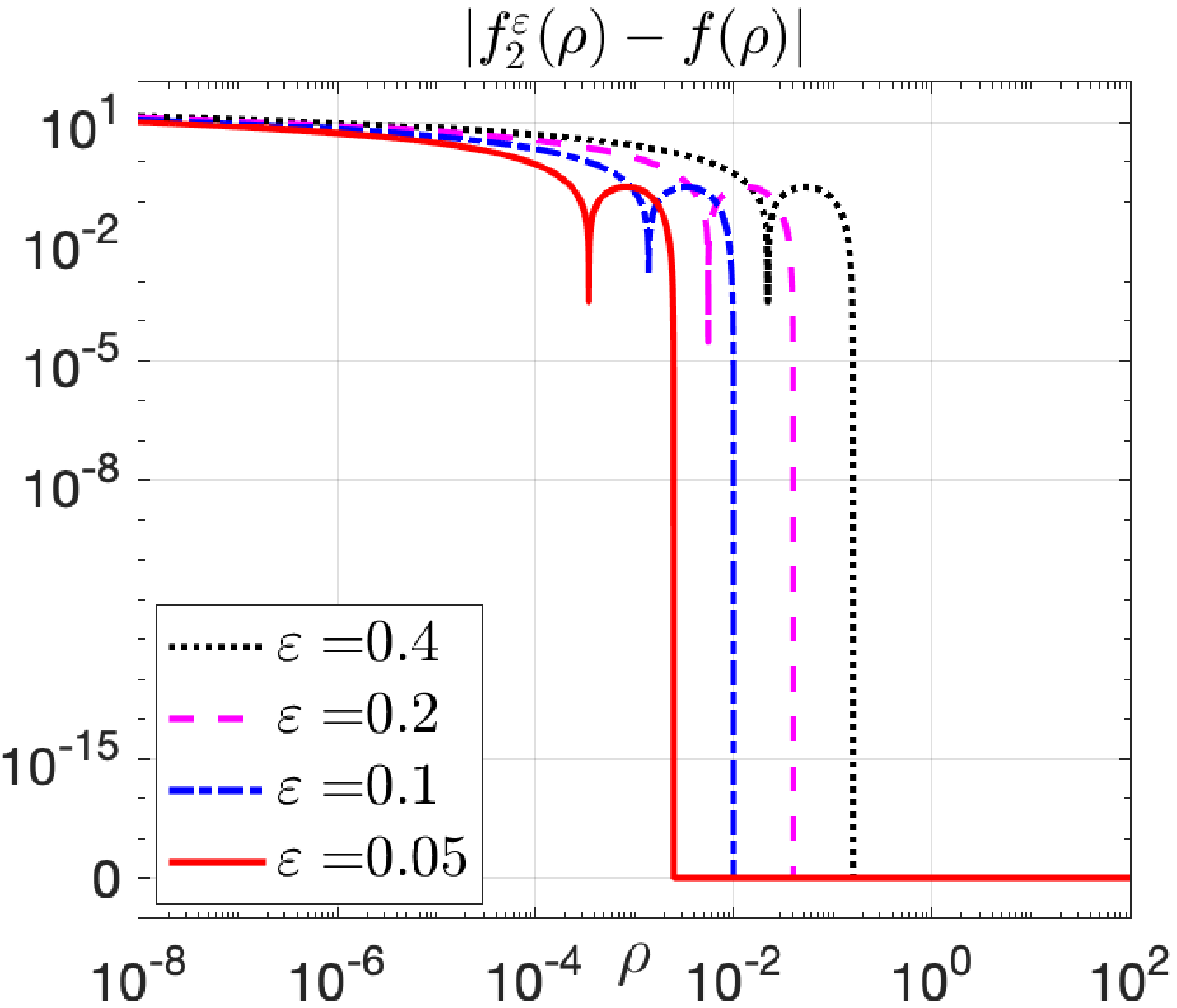}
\quad
\includegraphics[width=6cm,height=4cm]{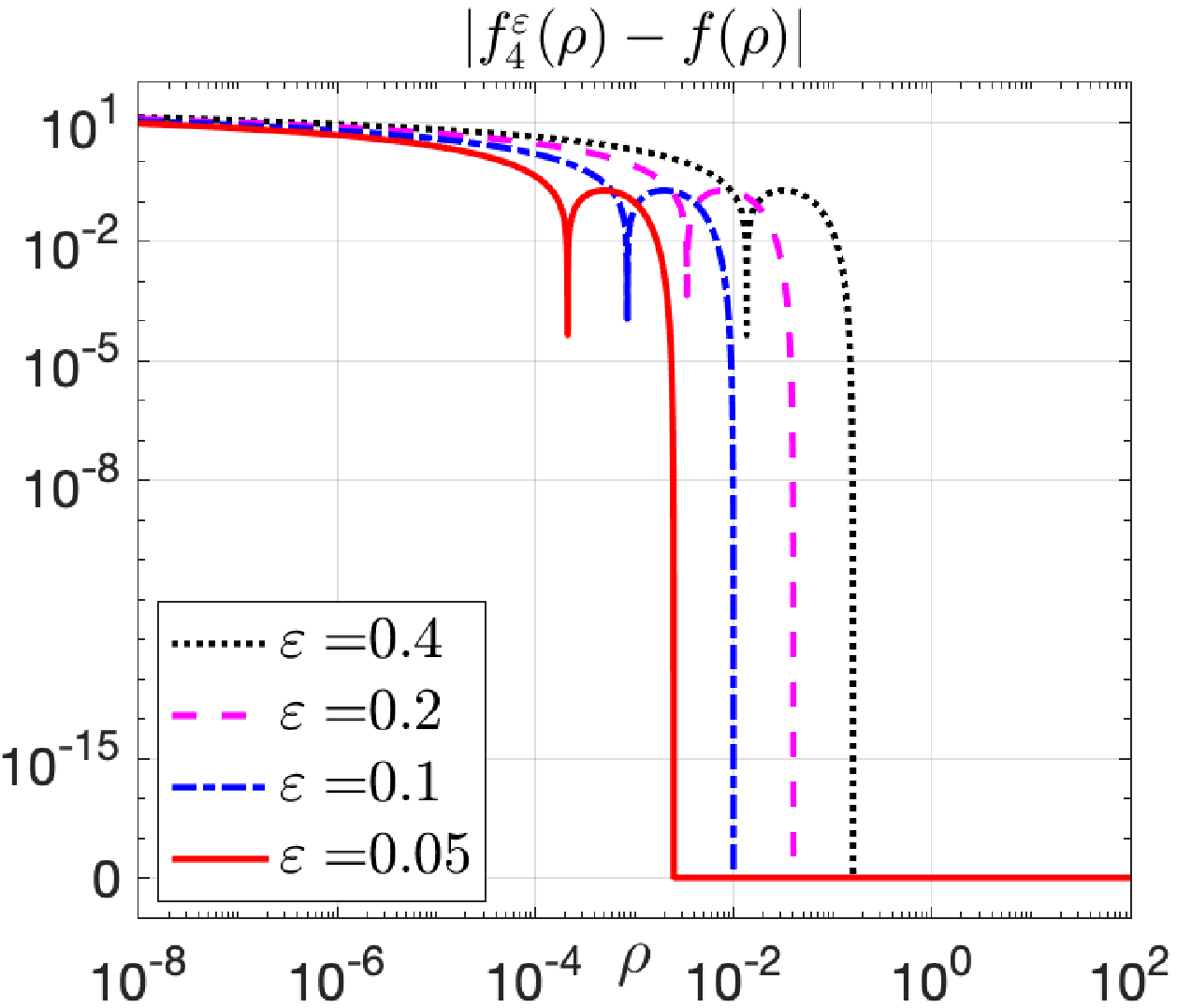}\\[1em]
\includegraphics[width=6cm,height=4cm]{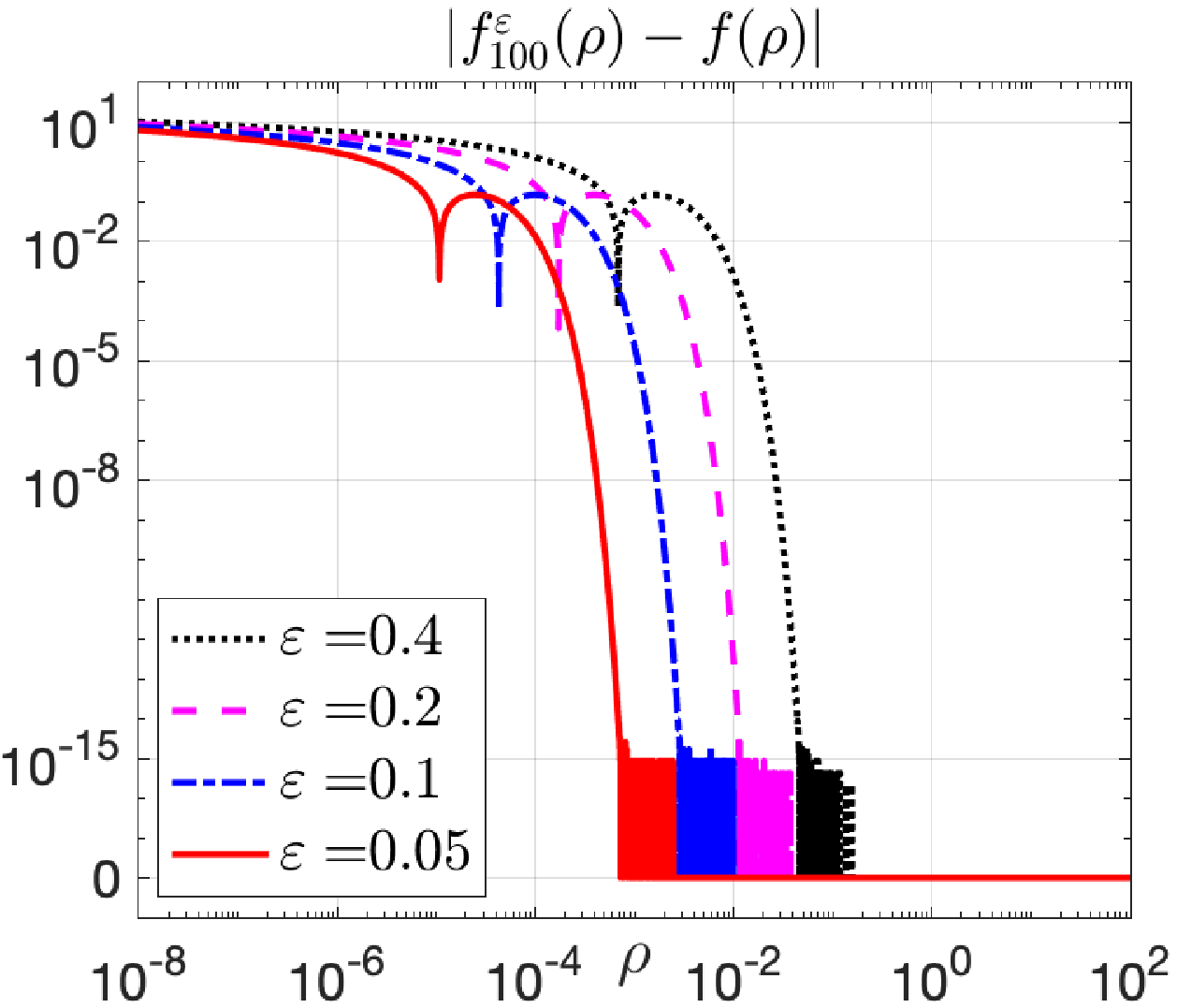}
\quad
\includegraphics[width=6cm,height=4cm]{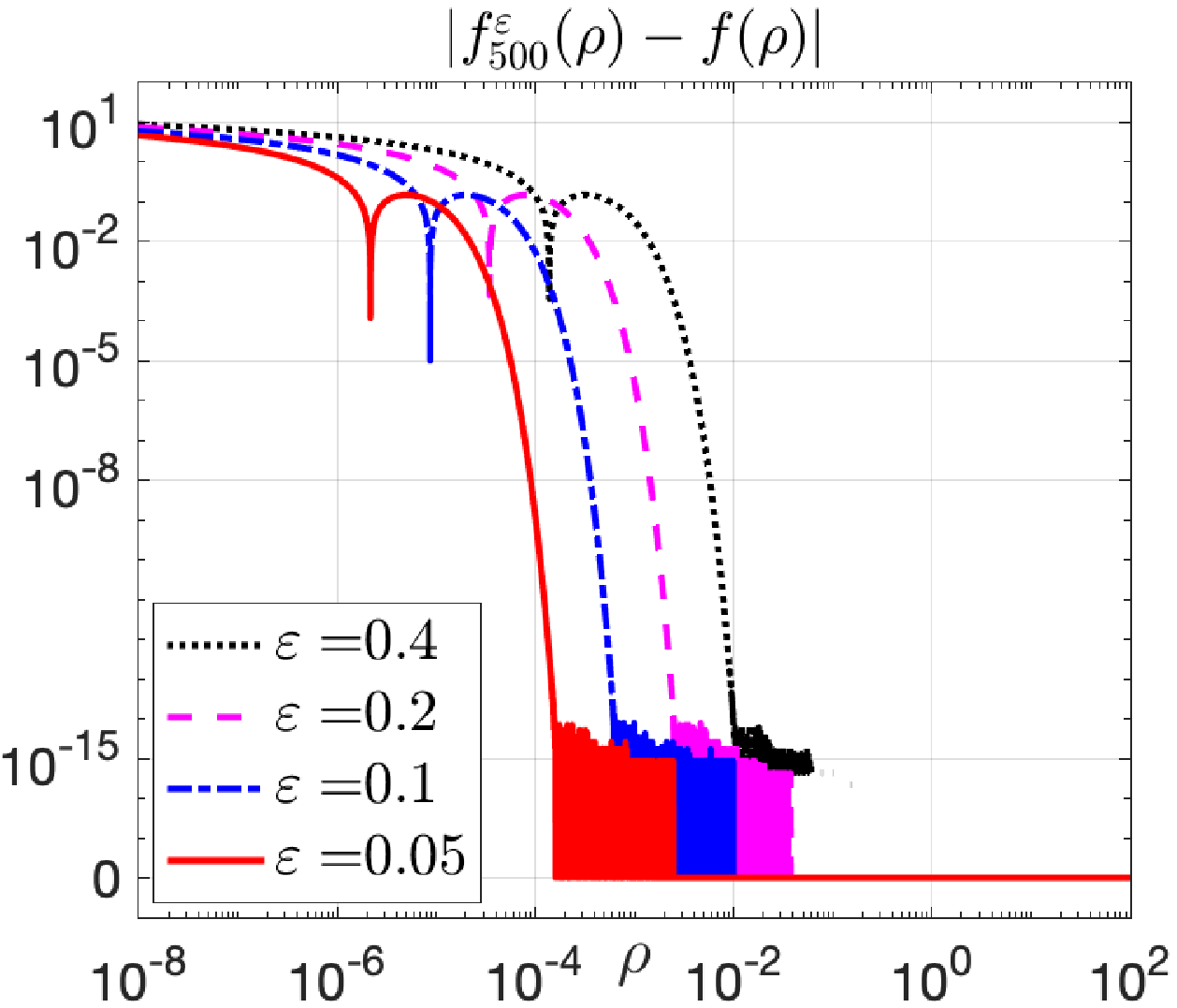}
\end{center}
 \caption{Comparison of different regularizations for the nonlinearity $f(\rho)=\ln \rho$.}
\label{fcomp}
\end{figure}

\begin{figure}[htbp!]
\begin{center}
\includegraphics[width=6cm,height=4cm]{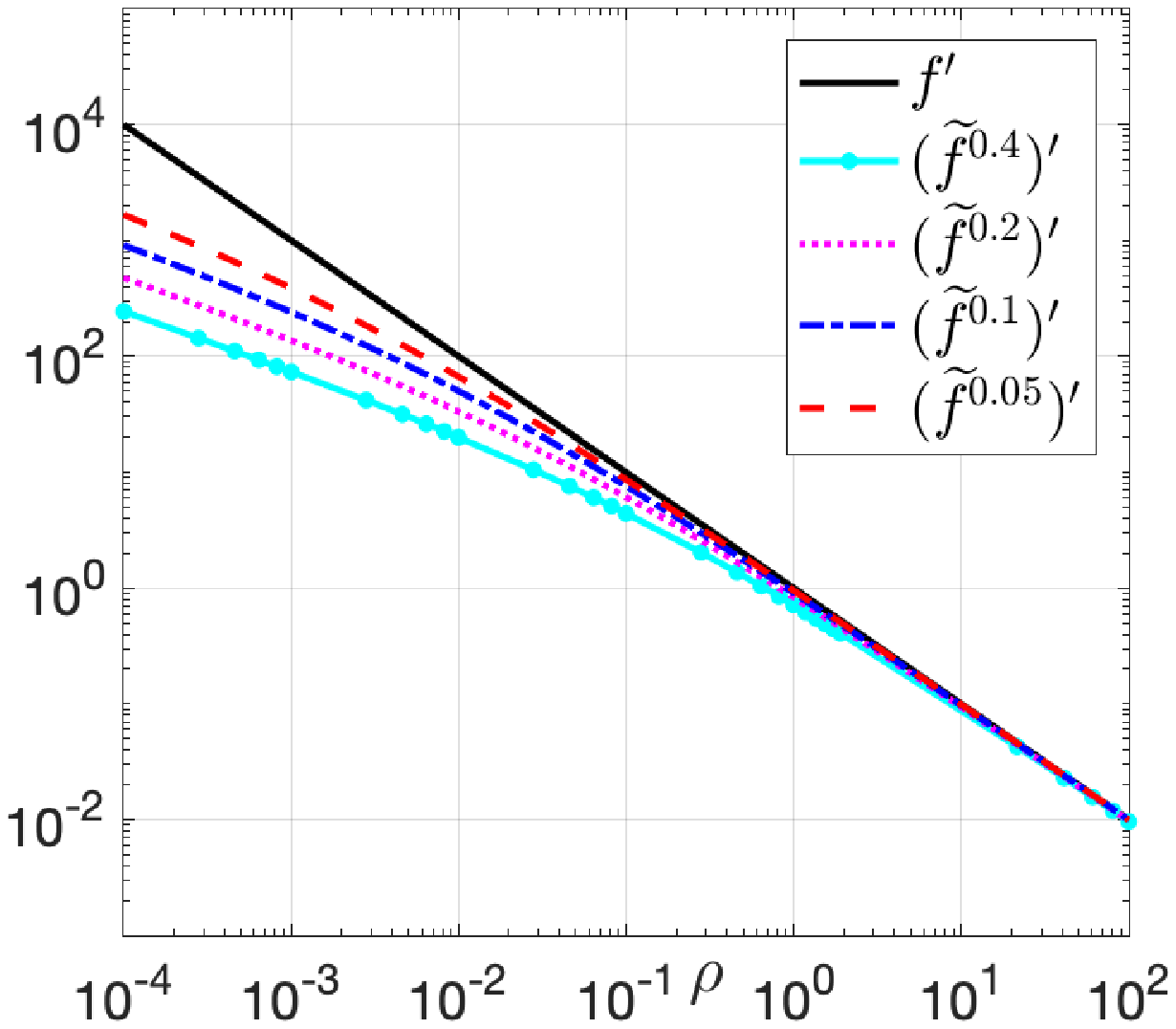}
\quad
\includegraphics[width=6cm,height=4cm]{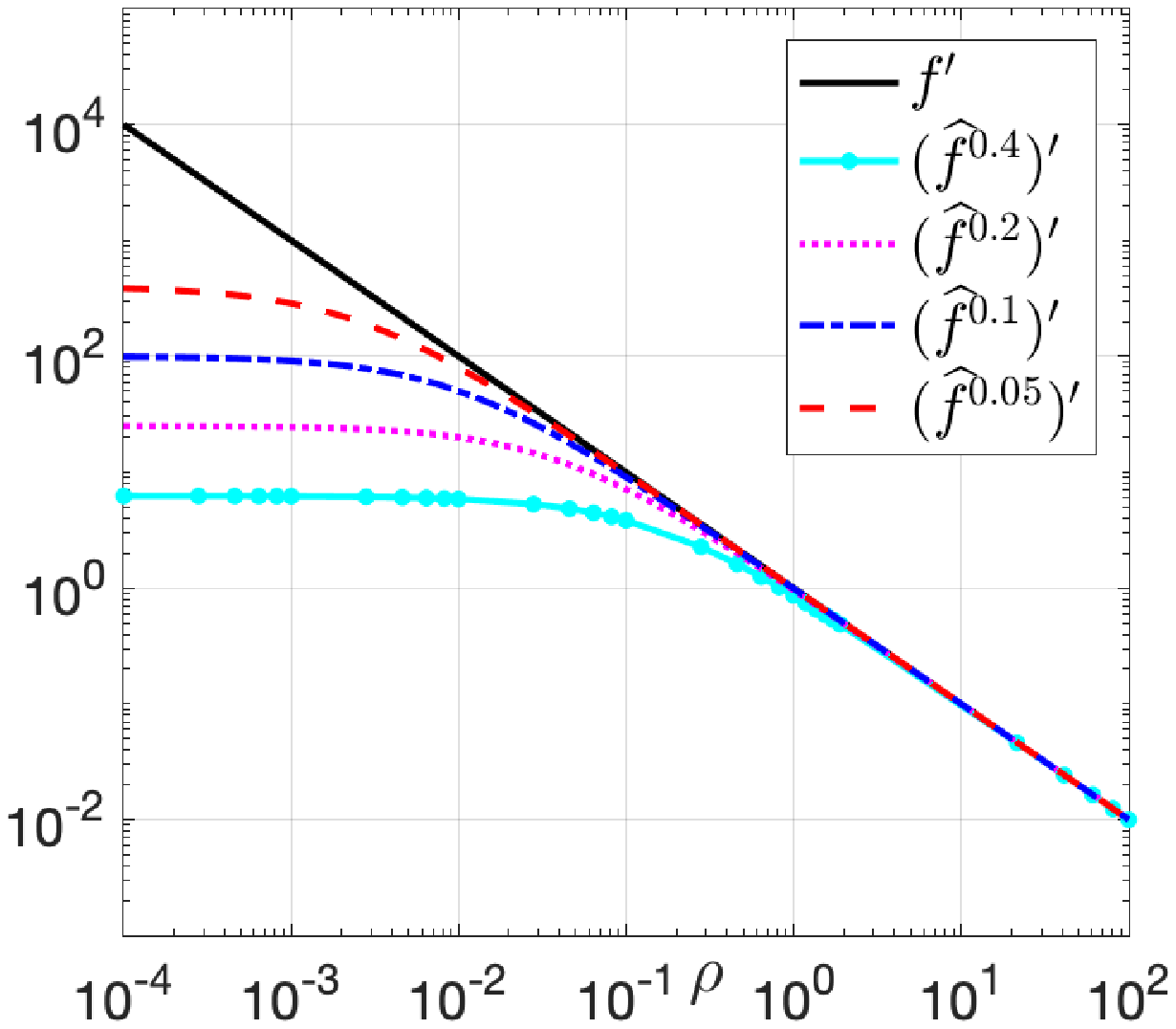}\\[1em]
\includegraphics[width=6cm,height=4cm]{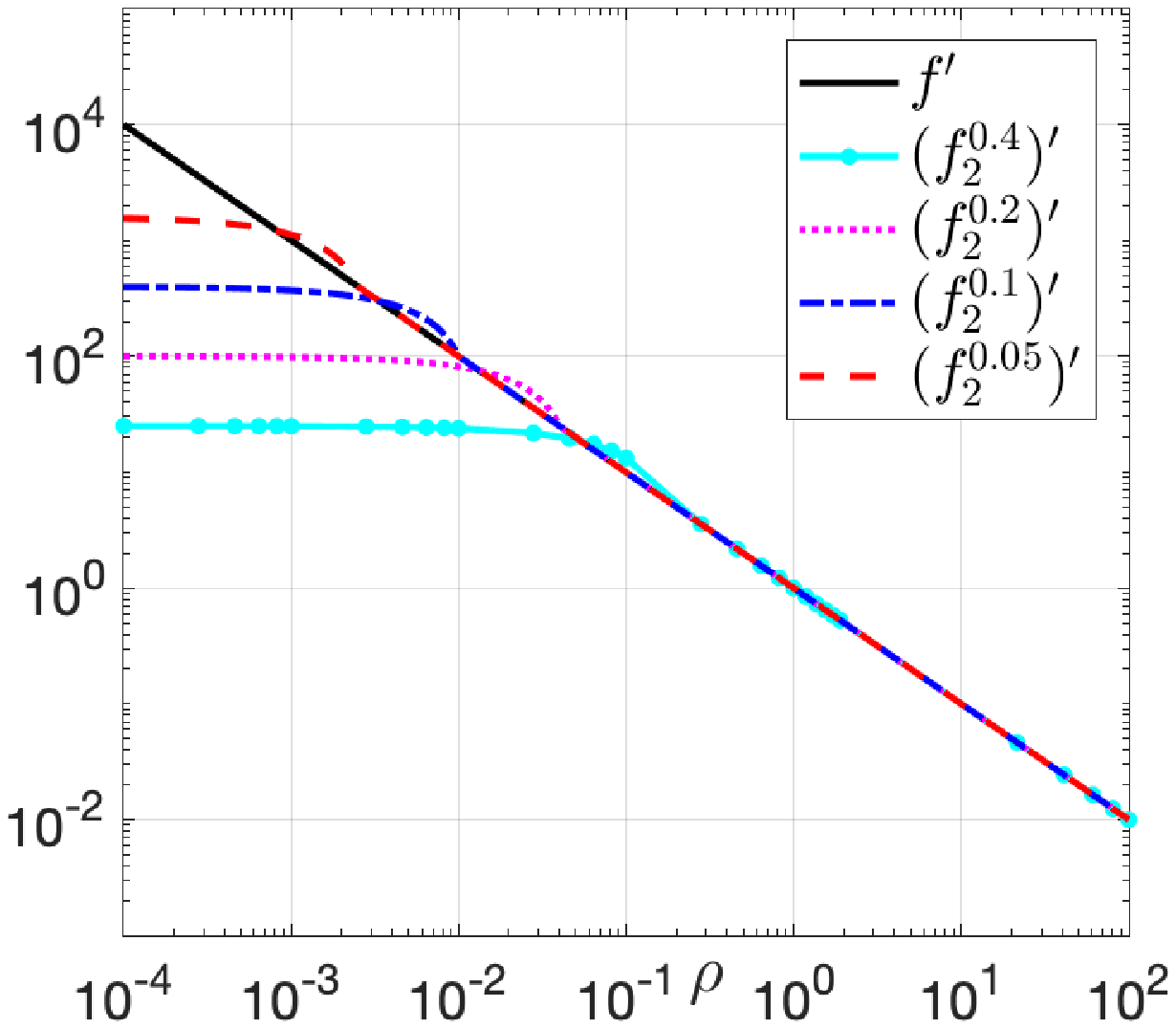}
\quad
\includegraphics[width=6cm,height=4cm]{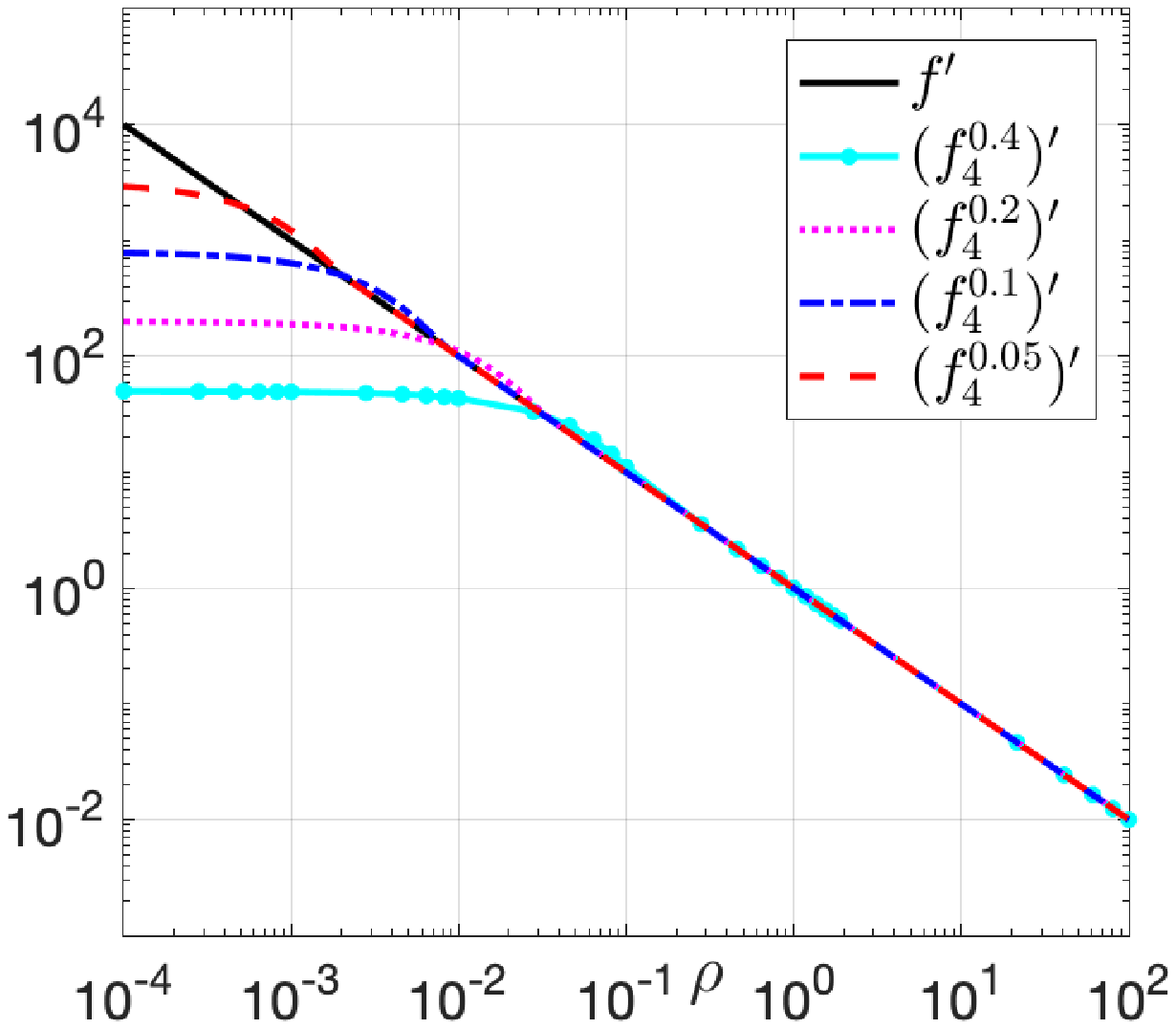}\\[1em]
\includegraphics[width=6cm,height=4cm]{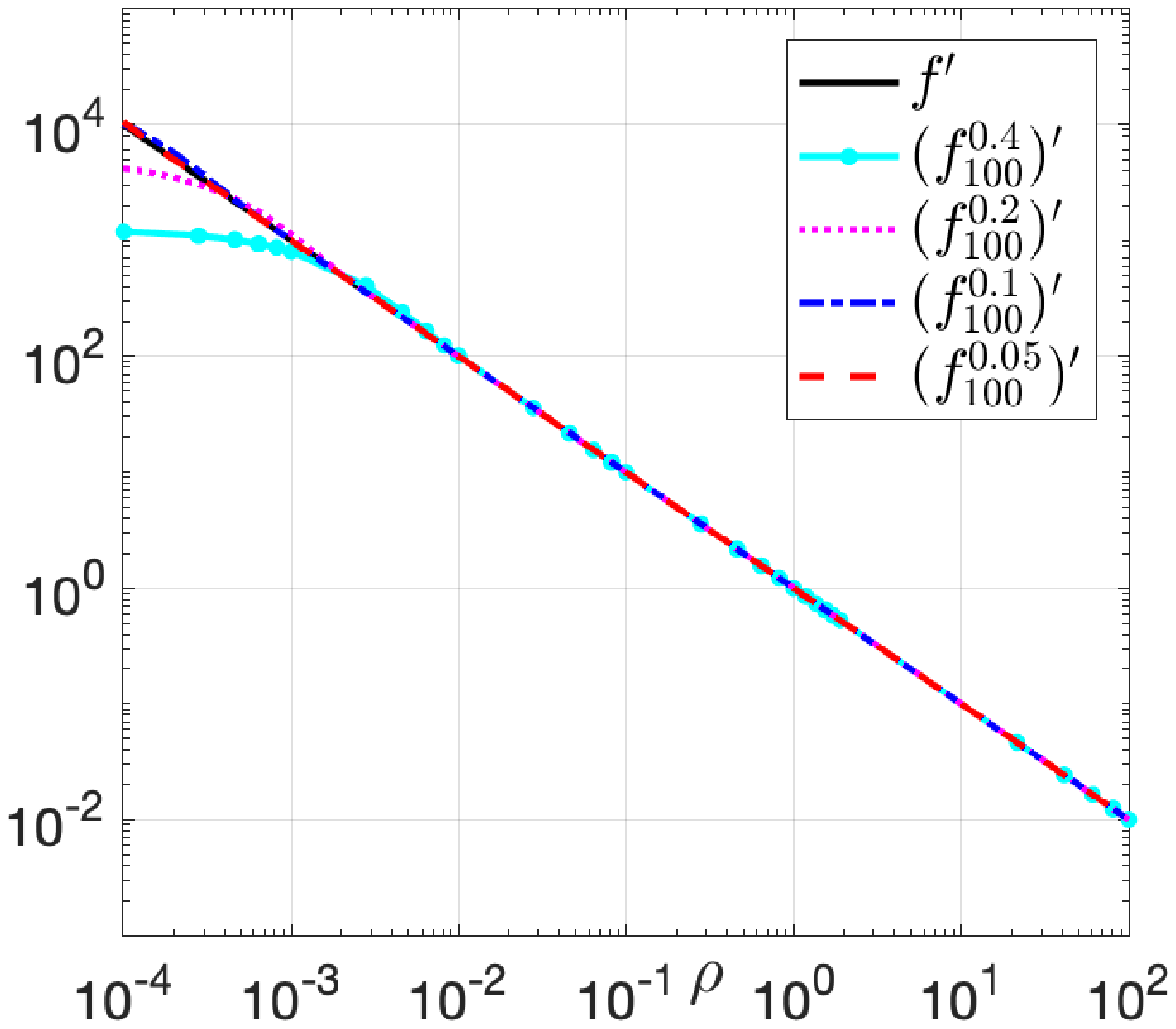}
\quad
\includegraphics[width=6cm,height=4cm]{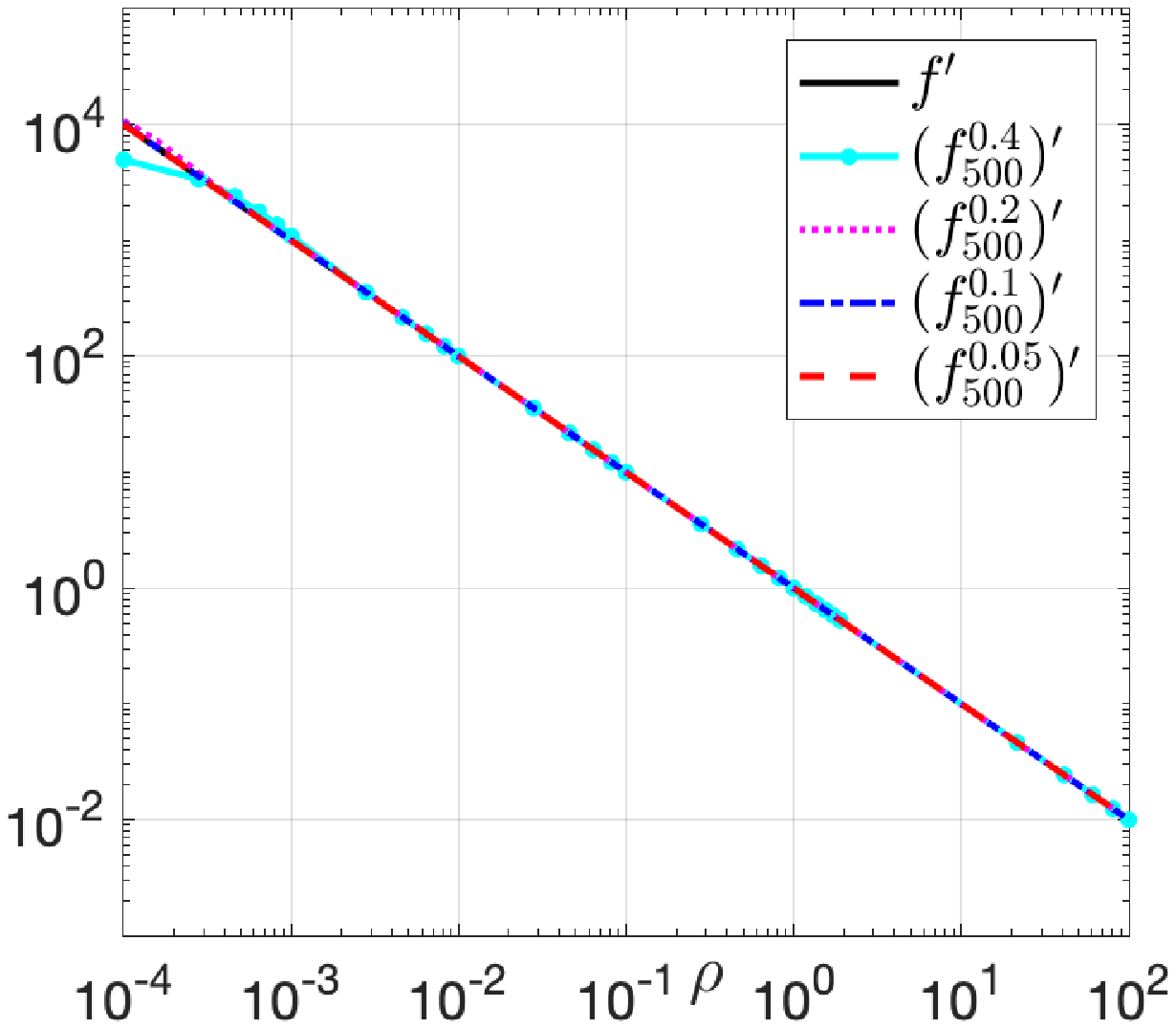}
\end{center}
 \caption{Comparison of different regularizations for  $f'(\rho)=1/\rho$.}
\label{fprime}
\end{figure}

\begin{figure}[htbp!]
\begin{center}
\includegraphics[width=6cm,height=4cm]{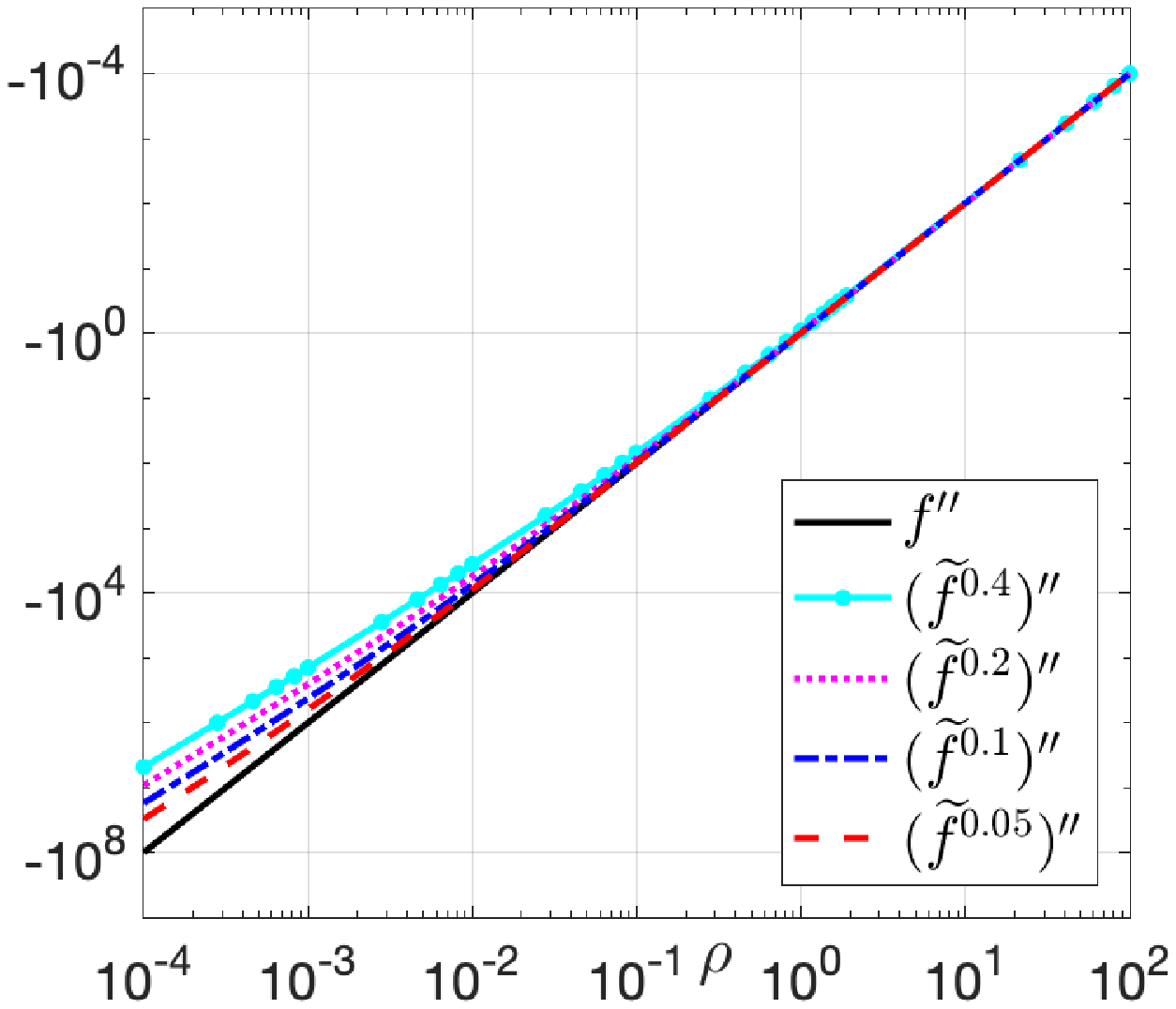}
\quad
\includegraphics[width=6cm,height=4cm]{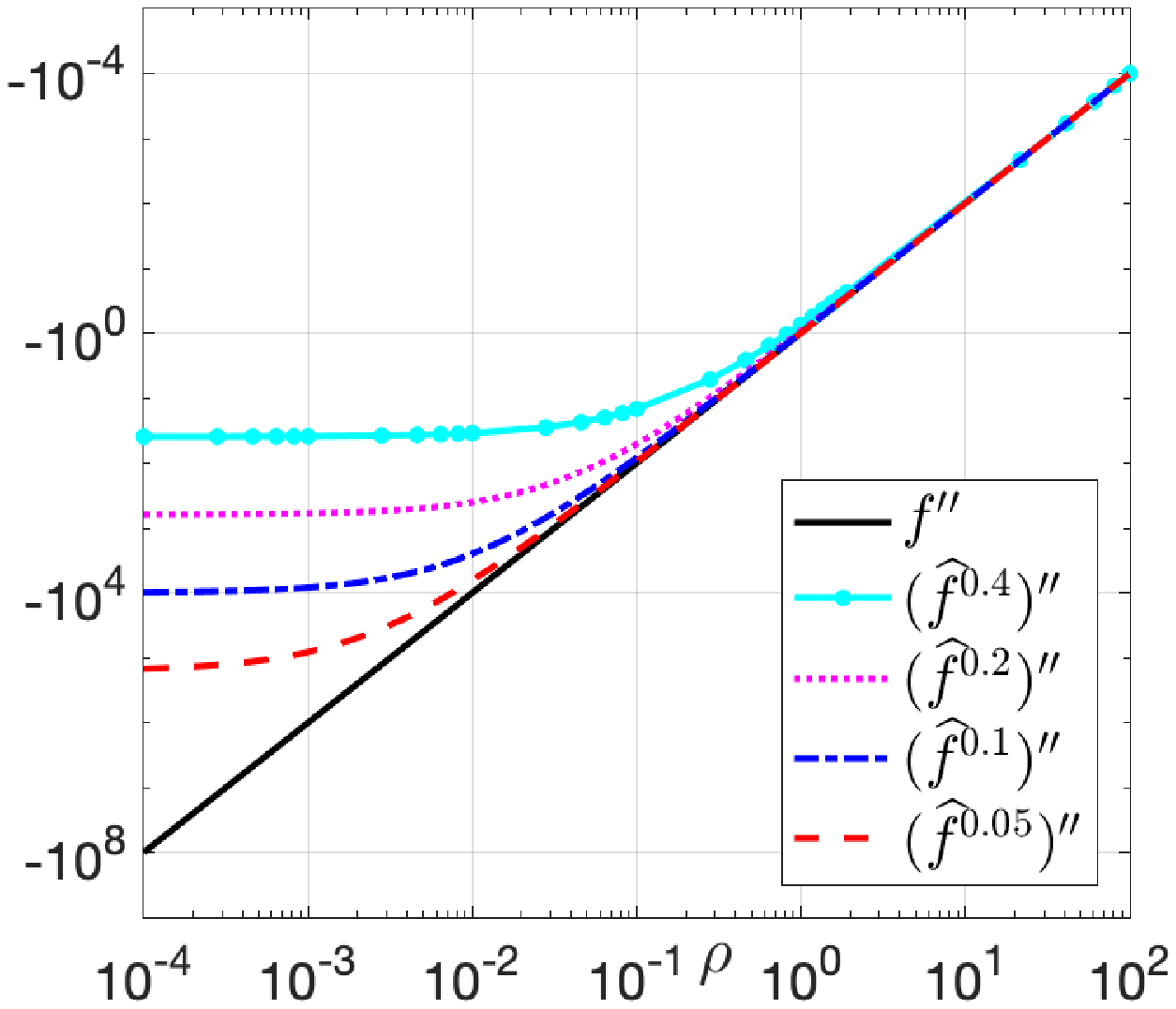}\\[1em]
\includegraphics[width=6cm,height=4cm]{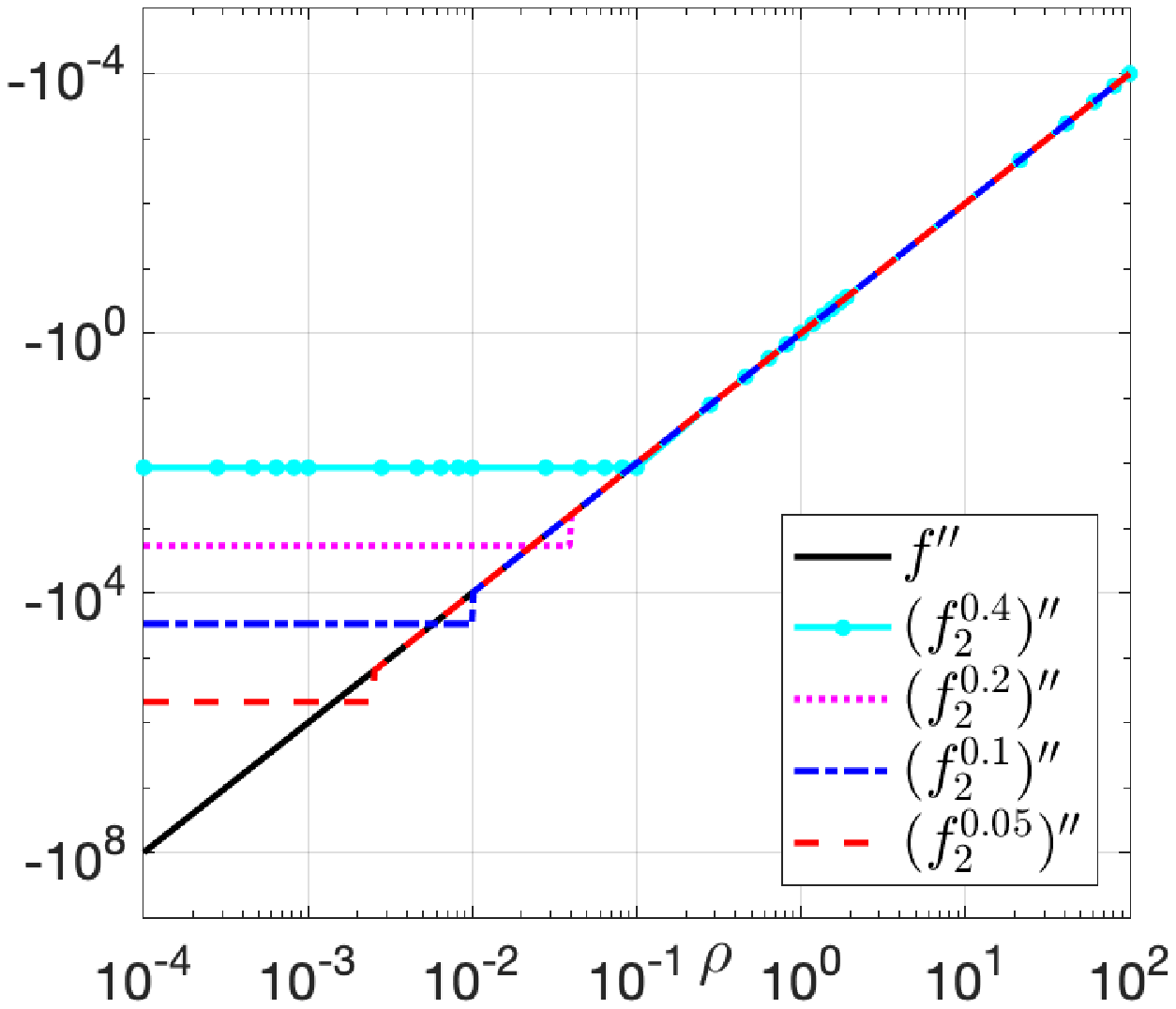}
\quad
\includegraphics[width=6cm,height=4cm]{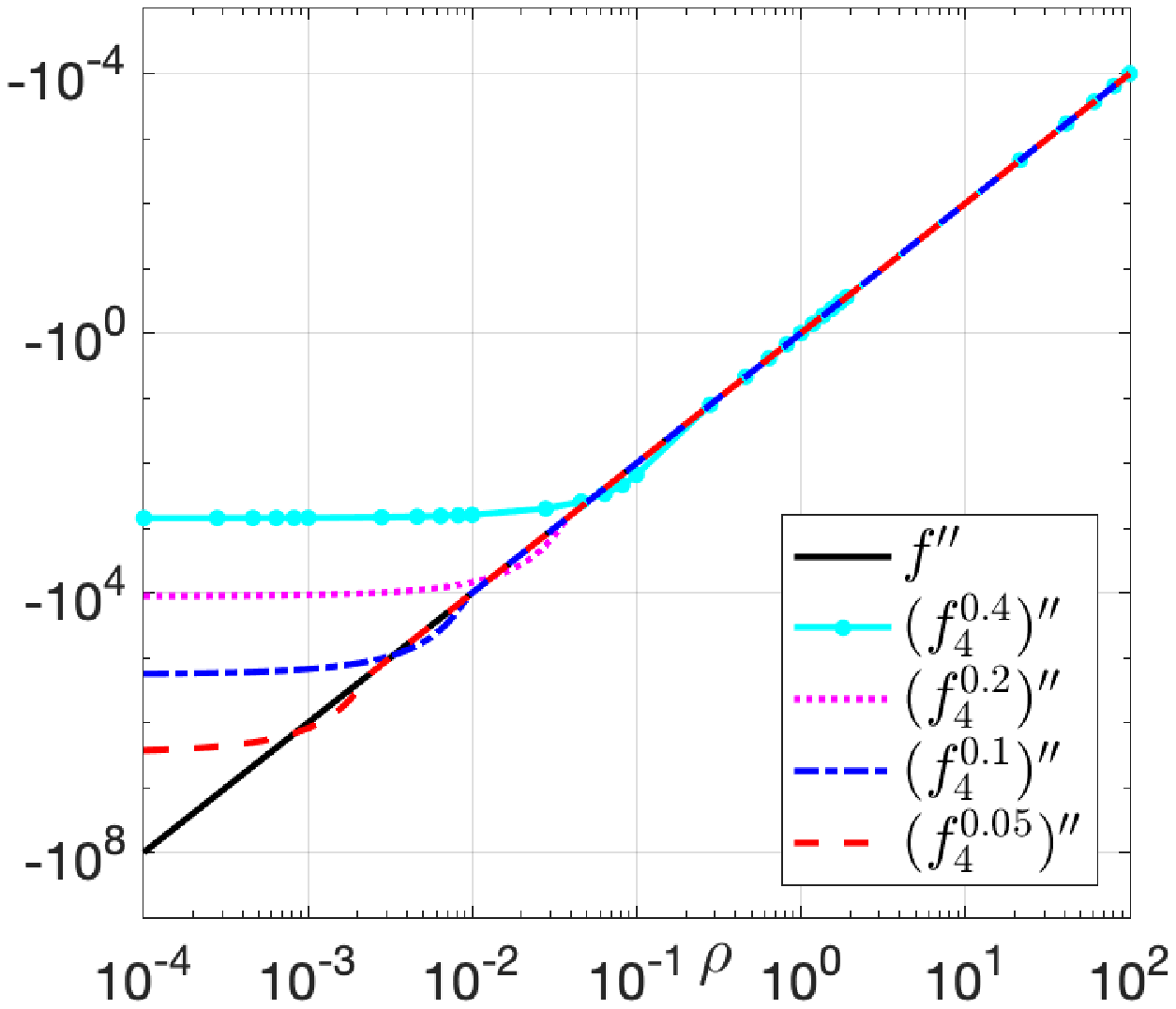}\\[1em]
\includegraphics[width=6cm,height=4cm]{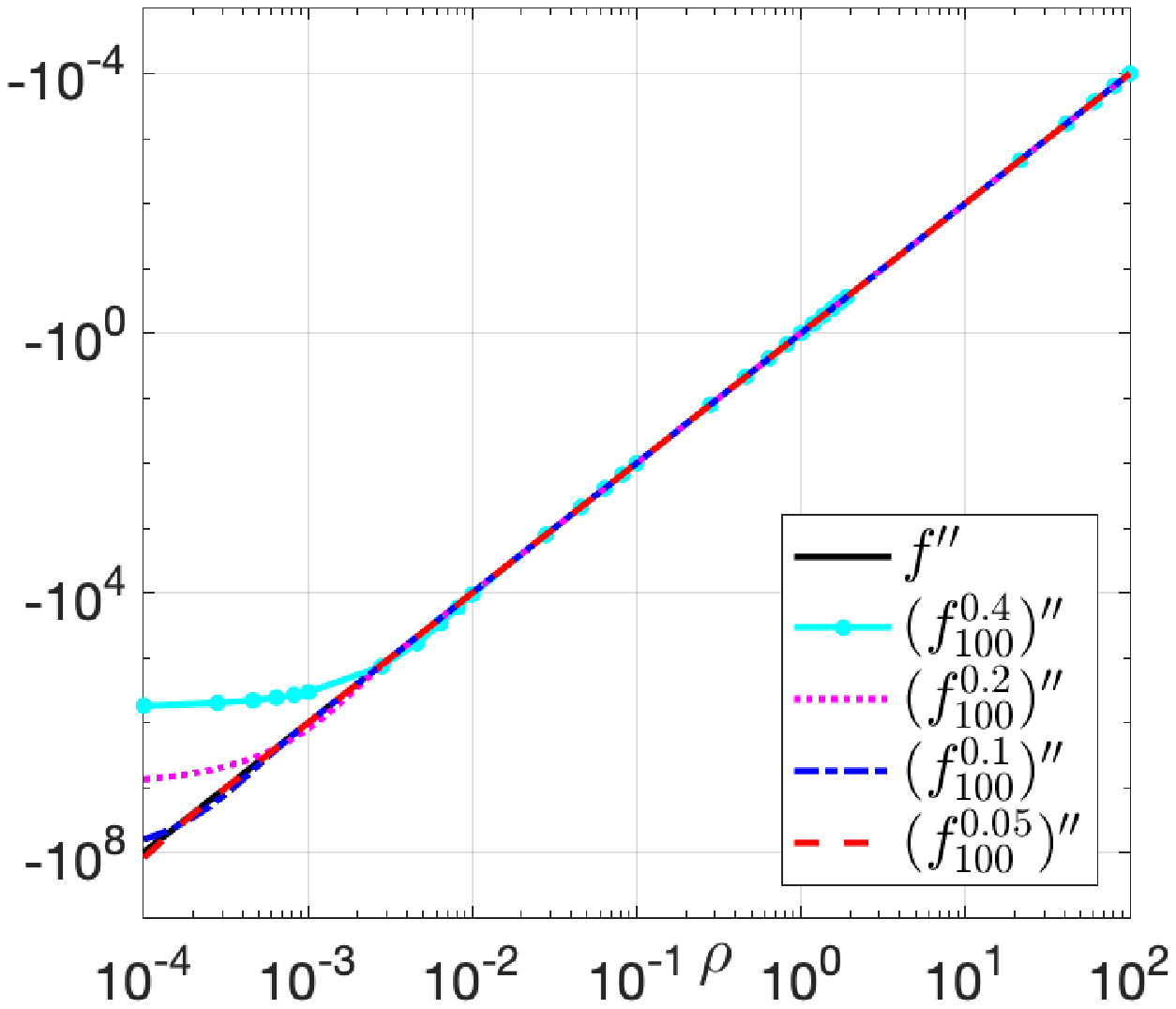}
\quad
\includegraphics[width=6cm,height=4cm]{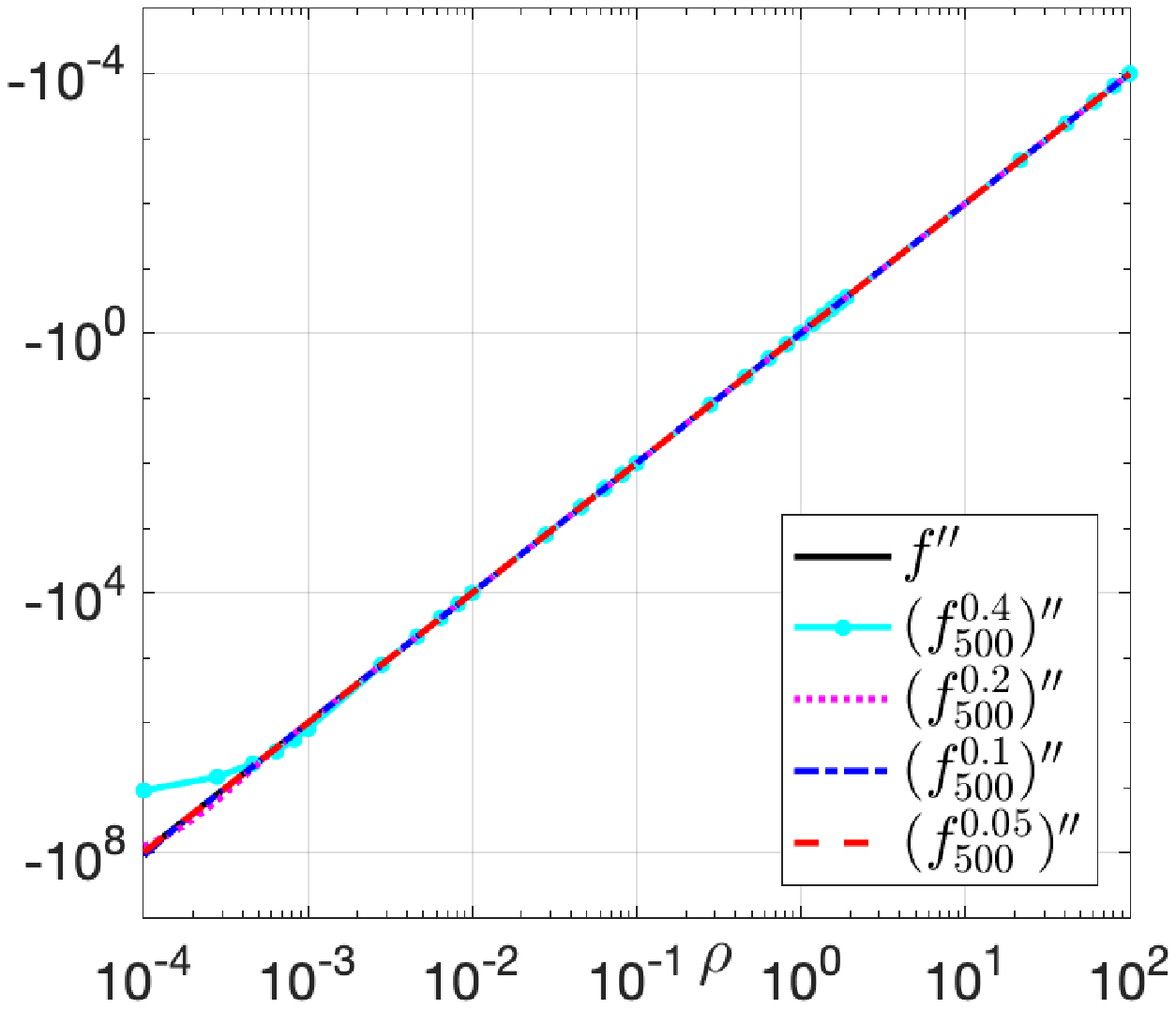}
\end{center}
 \caption{Comparison of different regularizations for  $f''(\rho)=-1/\rho^2$.}
\label{fpp}
\end{figure}

\begin{figure}[htbp!]
\begin{center}
\includegraphics[width=6cm,height=4cm]{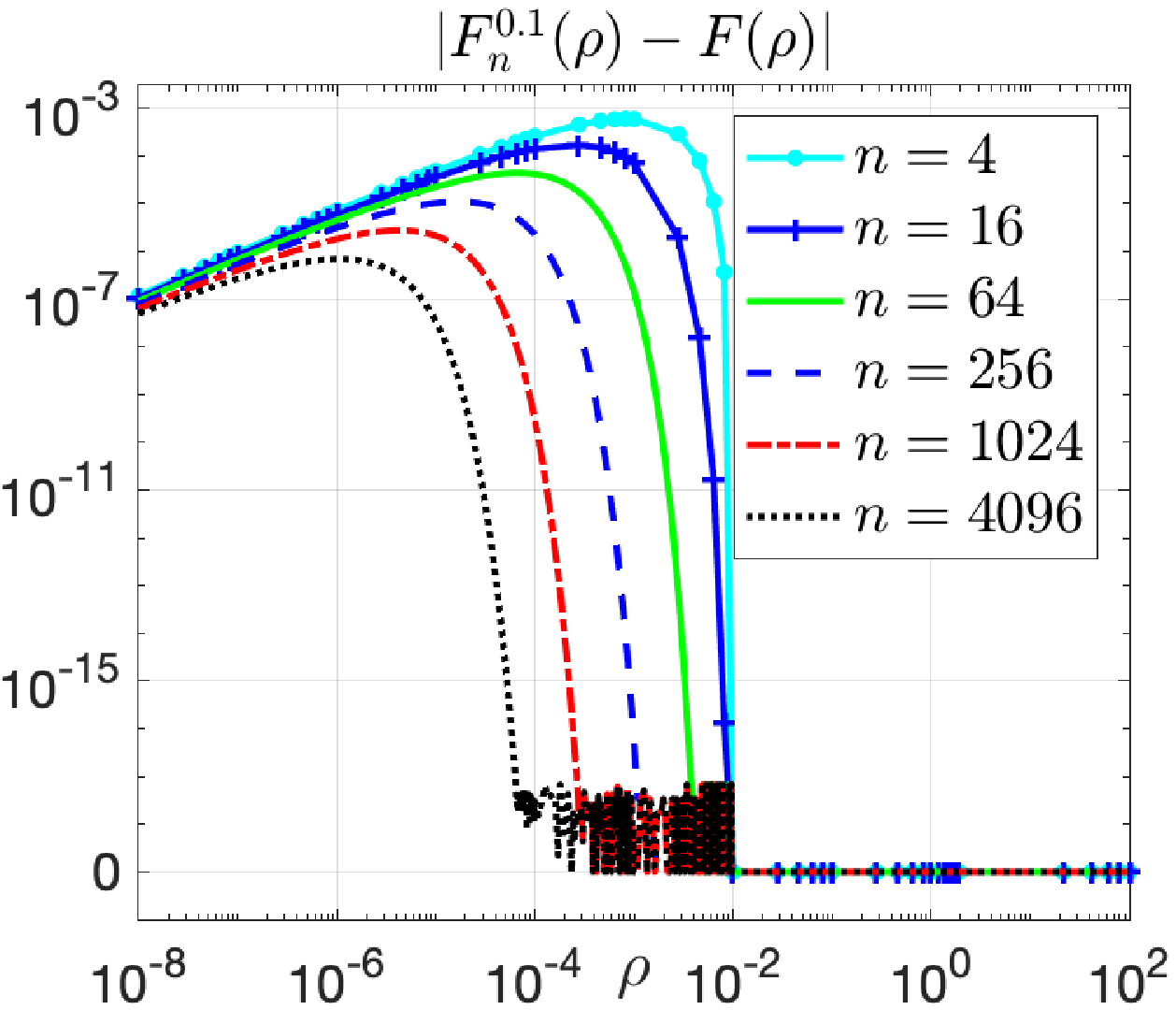}
\quad
\includegraphics[width=6cm,height=4cm]{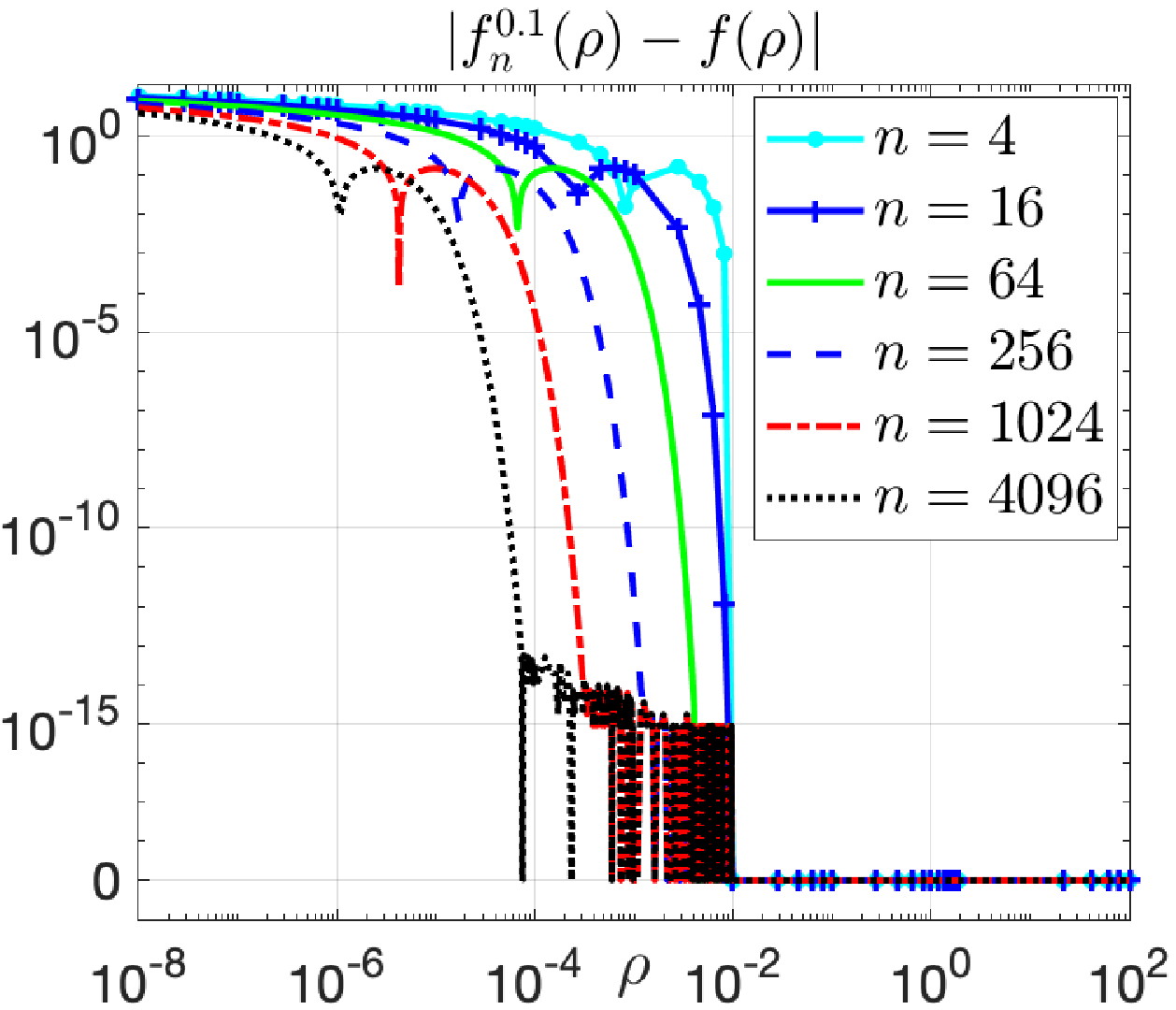}\\[1em]
\includegraphics[width=6cm,height=4cm]{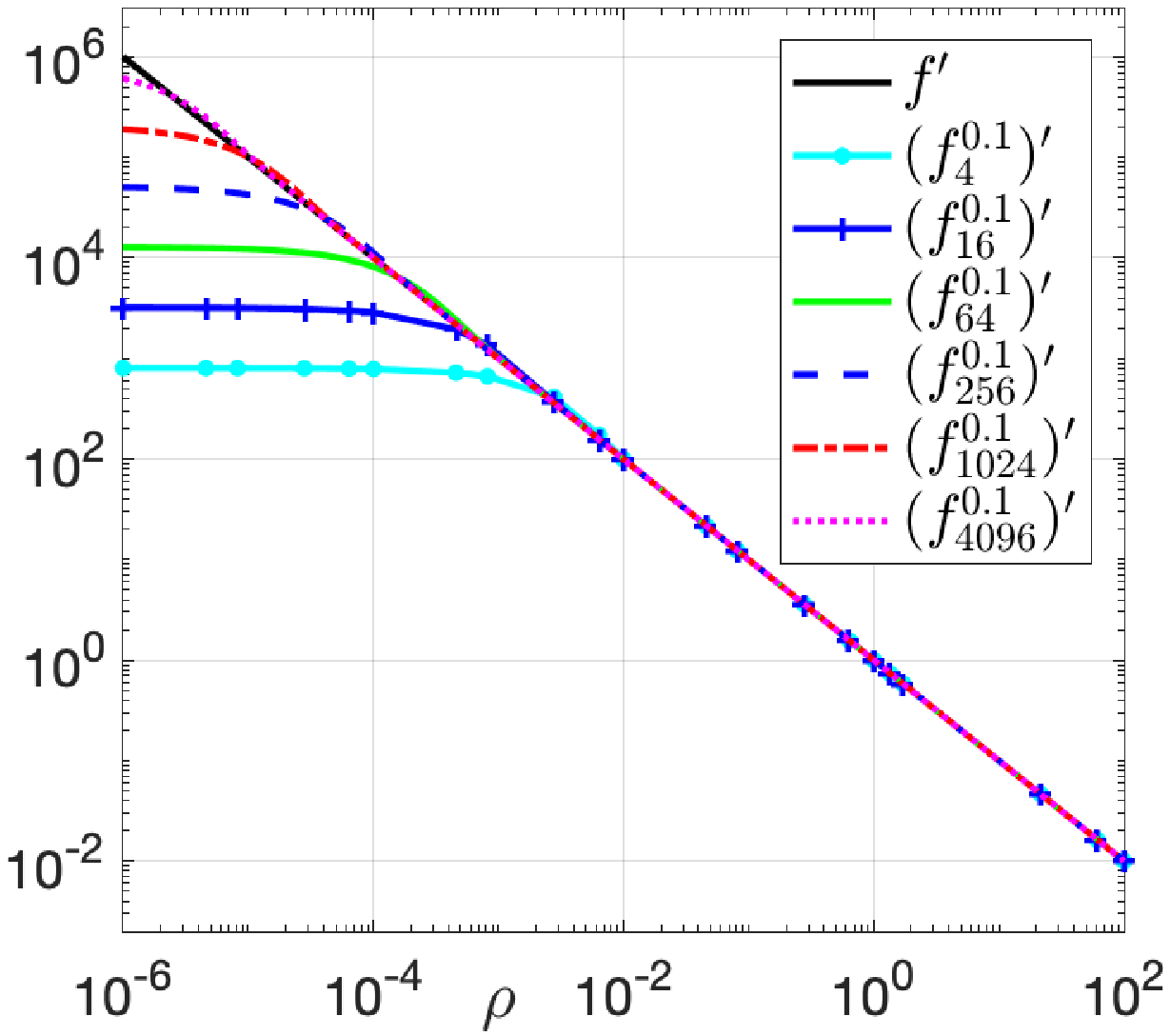}
\quad
\includegraphics[width=6cm,height=4cm]{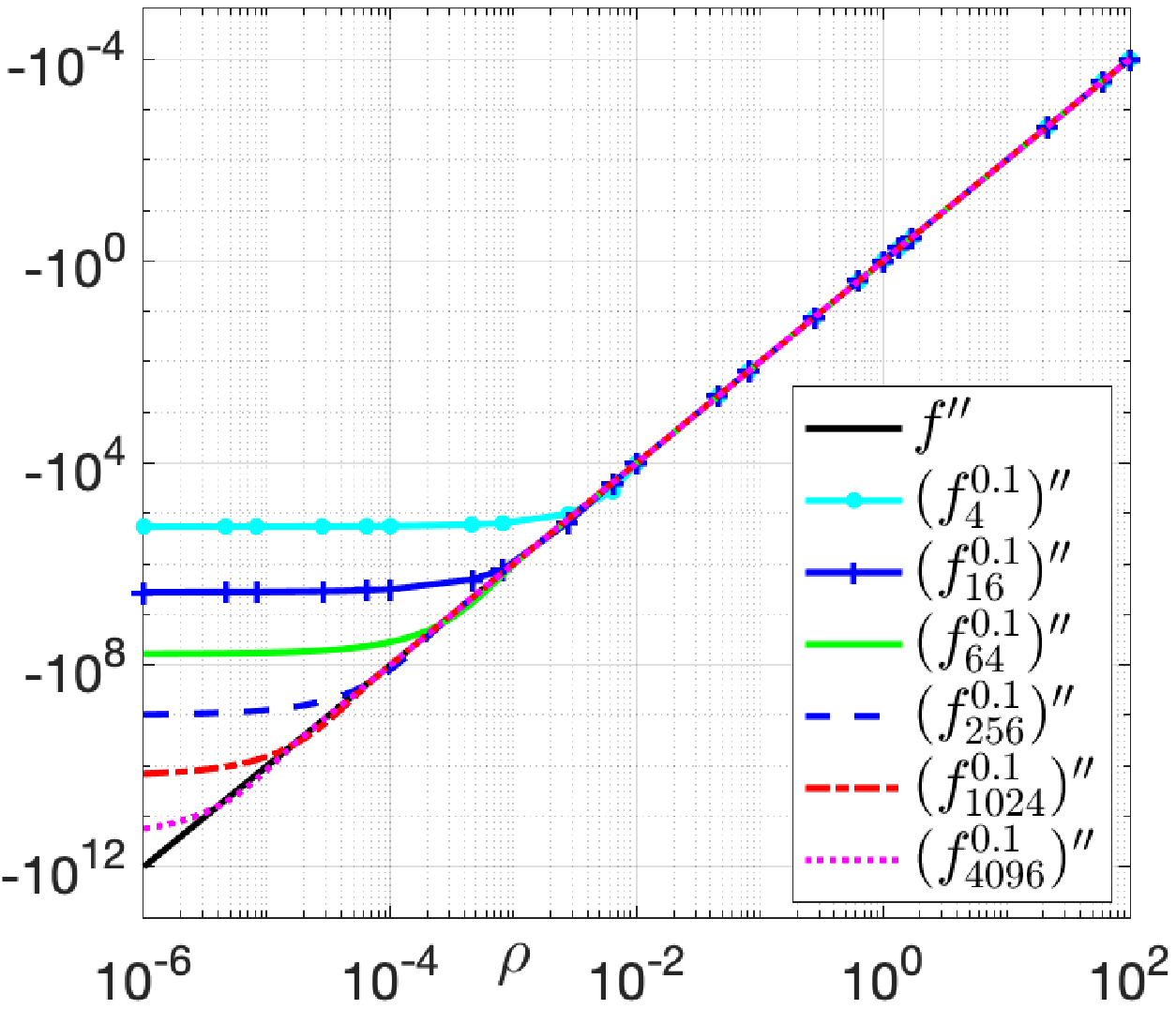}
\end{center}
 \caption{Comparison of regularizations $g_n^{0.1}$ ($g=F, f, f', f''$)  with different order $n$.}
\label{fig:Revision_EnergyReguL_Conv_WRT_Order_N}
\end{figure}

\section{Local energy regularization (LER) for the LogNLS}
%arithmic Schr\"odingerequation}
\label{sec:regulLSE}
In this section, we consider the regularized energy
\be
\label{RegL_Energ}
E_n^\ep(u):=\int_\Omega\left[|\nabla u|^2+\lambda F_n^\ep(|u|^2)\right]d\bx,
\ee
where $F_n^\ep$ is defined in \eqref{Fn}.
The Hamiltonian flow of the regularized energy $i\p_t u=\fl{\delta E_n^\ep(u)}{\delta \overline{u}}$ yields the following energy regularized logarithmic Schr\"odinger equation (ERLogSE)
with a regularizing parameter $0<\ep\ll1 $,
\be\label{ERLSE}
\left\{
\begin{aligned}
&i\p_t u^\ep(\bx,t)=-\Delta u^\ep(\bx,t)+\lambda\,
u^\ep(\bx,t)\,f_n^\ep(|u^\ep(\bx,t)|^2),\quad \bx\in \Omega, \quad
t>0,\\
&u^\ep(\bx,0)=u_0(\bx),\quad \bx\in \overline{\Omega}.
\end{aligned}
\right.
\ee
We recall that $f_n^\ep$ is defined by \eqref{fep}.

\subsection{The Cauchy problem}
To investigate the well-posedness of the problem \eqref{ERLSE}, we first introduce some appropriate spaces. For $\alpha>0$ and $\Omega=\R^d$, denote by $L^2_\alpha$ the weighted $L^2$ space
\[L^2_\alpha:=\{v\in L^2(\mathbb{R}^d), \quad \bx\longmapsto \langle
  \bx \rangle^\alpha v(\bx)\in L^2(\mathbb{R}^d)\},\]
where $\langle \bx \rangle :=\sqrt{1+|\bx|^2}$, with norm
$\|v\|_{L^2_\alpha}:=\|\langle \bx\rangle^\alpha
v(\bx)\|_{L^2(\mathbb{R}^d)}$.
Regarding the Cauchy problem \eqref{ERLSE}, we have similar results as
for the regularization \eqref{RLSE0} in \cite{bao2019error},
but not quite the same. For the convenience of the reader, we recall the main arguments.

\bigskip

\begin{theorem}\label{theo:cauchy}
  Let $\lambda\in \R$, $u_0\in H^1(\Omega)$, and $0<\ep\le 1$. \\
$(1).$ For \eqref{ERLSE} posed on $\Omega=\R^d$ or  a bounded domain
$\Omega$ with homogeneous Dirichlet or periodic boundary condition, there exists a unique, global weak solution $u^\ep\in L^\infty_{\rm
    loc}(\R; H^1(\Omega))$ to \eqref{ERLSE} (with $H^1_0(\Omega)$
  instead of $H^1(\Omega)$ in the Dirichlet case). Furthermore, for
  any given $T>0$, there exists a positive constant $C(\lambda, T)$
  (independent of $n$) such that
    \be\label{unib1}
    \|u^\ep\|_{L^\infty([0, T]; H^1(\Omega))}\le C(\lambda, T)\|u_0\|_{H^1(\Omega)}, \quad \forall \ep>0.
    \ee
 $(2).$  For \eqref{ERLSE} posed on a bounded domain $\Omega$ with
 homogeneous Dirichlet or periodic boundary condition, if in addition $u_0\in H^2(\Omega)$, then $u^\ep \in
L^\infty_{\rm loc}(\R;H^2(\Omega))$ and there
exists a positive constant $C(n, \lambda, T)$ such that
    \be\label{unib2}
    \|u^\ep\|_{L^\infty([0, T]; H^2(\Omega))}\le C(n, \lambda, T)\|u_0\|_{H^2(\Omega)}, \quad \forall \ep>0.
    \ee
$(3).$ For \eqref{ERLSE} on $\Omega=\R^d$, suppose moreover
$u_0\in L^2_\alpha$, for some $0<\alpha\le 1$.
\begin{itemize}
\item There exists a unique, global weak solution $u^\ep\in L^\infty_{\rm
    loc}(\R; H^1(\R^d)\cap L^2_\alpha)$ to \eqref{ERLSE}, and
    \be\label{unib3}
    \begin{split}
   & \|u^\ep\|_{L^\infty([0, T]; H^1)}\le C(n, \lambda, T)\|u_0\|_{H^1}, \\
    &\|u^\ep\|_{L^\infty([0, T]; L^2_\alpha)}\le C(n, \lambda, T, \|u_0\|_{H^1}) \|u_0\|_{L^2_\alpha},
    \quad \forall \ep>0.
    \end{split}
    \ee
\item If in addition $u_0\in H^2(\R^d)$, then $u^\ep \in
L^\infty_{\rm loc}(\R;H^2(\R^d))$, and
\be\label{unib4}
\|u^\ep\|_{L^\infty([0, T]; H^2)}\le C(n, \lambda, T, \|u_0\|_{H^2}, \|u_0\|_{L^2_\alpha}),\quad \forall \ep>0.\ee
\item If  $u_0\in H^2(\R^d)\cap
  L_2^2$, then  $u^\ep \in L^\infty_{\rm loc}(\R;H^2(\R^d)\cap L_2^2)$.
\end{itemize}
\end{theorem}

\begin{proof}
  (1).
For fixed $\ep>0$, the nonlinearity in \eqref{ERLSE} is locally
Lipschitz continuous, and grows more slowly than any power of
$|u^\ep|$. Standard Cauchy theory for
nonlinear Schr\"odinger equations implies that there exists a unique
solution $u^\ep\in L^\infty_{\rm loc}(\R;H^1(\Omega))$ to
\eqref{ERLSE} (respectively, $u^\ep\in L^\infty_{\rm
  loc}(\R;H^1_0(\Omega))$ in the Dirichlet case); see
e.g. \cite[Corollary~3.3.11 and Theorem~3.4.1]{CazCourant}. In
addition, the $L^2$-norm of $u^\ep$ is independent of time,
\[ \|u^\ep(t)\|_{L^2(\Omega)}^2=\|u_0\|_{L^2(\Omega)}^2,\quad \forall t\in\R.\]
  For $j\in
\{1,\dots,d\}$, differentiate \eqref{ERLSE} with respect to $x_j$:
\begin{equation*}
  \(i\partial_t +\Delta\)\partial_j u^\ep = \lambda \partial_j u^\ep f_n^\ep(|u^\ep|^2)
  + 2\lambda u^\ep (f_n^\ep)'(|u^\ep|^2){\mathrm{Re}}\(\overline{u^\ep}\partial_j u^\ep\).
\end{equation*}
Multiply the above equation by $\partial_j \overline{u^\ep}$,
integrate on $\Omega$, and take the imaginary part: \eqref{fd} implies
\begin{equation*}
  \frac{1}{2}\frac{d}{dt}\|\partial_j u^\ep \|_{L^2(\Omega)}^2 \le
  6|\lambda| \|\partial_j u^\ep \|_{L^2(\Omega)}^2,
\end{equation*}
hence \eqref{unib1}, by Gronwall lemma.

(2). The propagation of the $H^2$ regularity is standard, since
$f_n^\ep$ is smooth, so we focus on \eqref{unib2}. We now
differentiate \eqref{ERLSE} with respect to time:  we get the same
estimate as above, with $\partial_j$ replaced by $\partial_t$, and so
\begin{equation*}
  \|\partial_t u^\ep (t)\|_{L^2(\Omega)}^2\le  \|\partial_t u^\ep
  (0)\|_{L^2(\Omega)}^2e^{12|\lambda\rvert \lvert t|}.
\end{equation*}
In view of \eqref{ERLSE},
\begin{equation*}
 i \partial_t u^\ep_{\mid t=0} = -\Delta u_0 +\lambda
  u_0f_n^\ep(|u_0|^2).
\end{equation*}
For $0<\delta<1$, we have
\begin{equation*}
  \sqrt\rho |f_n^\ep(\rho)|\le C(\delta)\(\rho^{1/2-\delta/2} + \rho^{1/2+\delta/2}\),
\end{equation*}
for some $C(\delta)$ independent of $\ep$ and $n$, so for $\delta>0$
sufficiently small, Sobolev embedding entails
\begin{equation*}
  \|\partial_t u^\ep (0)\|_{L^2(\Omega)}\le \|u_0\|_{H^2(\Omega)}+
  C(\delta)\(\|u_0\|^{1-\delta}_{L^{2-2\delta}(\Omega)} +
  \|u_0\|_{H^1(\Omega)}^{1+\delta}\).
\end{equation*}
Since $\Omega$ is bounded, H\"older inequality yields
\begin{equation*}
  \|u_0\|_{L^{2-2\delta}(\Omega)} \le
  \|u_0\|_{L^2(\Omega}|\Omega|^{\delta/(2-2\delta)}.
\end{equation*}
Thus, the first term in \eqref{ERLSE} is controlled in $L^2$. Using
the same estimates as above, we control the last
term in \eqref{ERLSE} (thanks to \eqref{unib1}), and we infer an $L^2$-estimate for $\Delta
u^\ep$, hence \eqref{unib2}.

(3). In the case $\Omega=\R^d$, we multiply \eqref{ERLSE} by $\langle
  \bx \rangle^\alpha$, and the same energy estimate as before now
  yields
  \begin{align*}
    \frac d{dt} \|u^\ep\|_{L^2_\alpha}^2 = 4\alpha\,\mathrm{Im}\int_{\R^d}  \frac{\bx
  \cdot \nabla  u^{\ep}  }{  \langle \bx \rangle^{2-2\alpha}} \,
\overline {u^\ep}(t)  \, d\bx & \lesssim
\|\<\bx\>^{2\alpha-1}u^\ep\|_{L^2(\R^d)}\|\nabla
                                u^\ep\|_{L^2(\R^d)}\\
    &\lesssim \|\<\bx\>^{\alpha}u^\ep\|_{L^2(\R^d)}\|\nabla
u^\ep\|_{L^2(\R^d)},
\end{align*}
where the last inequality follows from the assumption $\alpha\le 1$,
hence \eqref{unib3}. To prove \eqref{unib4},
  we resume the same approach as to get \eqref{unib2}, with the
difference that the H\"older estimate must be replaced by some other
estimate (see e.g. \cite{CaGa18}): for $\delta>0$ sufficiently small,
 \begin{equation*}
  \int_{{\mathbb R}^d} |u|^{2-2\delta} \lesssim \|u\|_{L^2(\R^d)}^{2-2\delta-d \delta/\alpha}
  \left\lVert \lvert \bx\rvert^\alpha u\right\rVert_{L^2(\R^d)}^{d
      \delta/\alpha} .
\end{equation*}
The $L^2_2$ estimate follows easily, see
e.g. \cite{bao2019error} for details.
\end{proof}
\color{black}
\subsection{Convergence of the regularized model}
\label{sec:cvmodel}
In this subsection, we show an approximation property of the
regularized model \eqref{ERLSE} to \eqref{LSE}.

\begin{lemma}\label{thmcon}
Suppose the equation \eqref{ERLSE} is set on $\Omega$, where $\Omega=\mathbb{R}^d$,
or $\Omega\subset \mathbb{R}^d$ is a bounded domain with homogeneous
Dirichlet or periodic boundary condition. We have
the general estimate:
\be\label{gene}
\fl{d}{dt}\|u^\ep(t)-u(t)\|_{L^2}^2\le |\lambda|\left(4\|u^\ep(t)-u(t)\|_{L^2}^2+
6\ep\|u^\ep(t)-u(t)\|_{L^1}\right).
\ee
\end{lemma}
\begin{proof}
Subtracting \eqref{LSE} from \eqref{ERLSE}, we see that the error function $e^\ep:=u^\ep-u$ satisfies
\[
i\p_t e^\ep+\Delta e^\ep=\lambda \left[u^\ep\ln(|u^\ep|^2)-u\ln(|u|^2)\right]+\lambda u^\ep\left[
f^\ep_n(|u^\ep|^2)-\ln (|u^\ep|^2)\right]\chi_{\{|u^\ep|<\ep\}}.
\]
Multiplying the above error equation by $\overline{e^\ep(t)}$, integrating
in space and taking imaginary parts, we can get by using
Lemma~\ref{pre}, \eqref{eq:Taylor}  and \eqref{eq:Q'} that
\begin{align*}
&\quad\fl{1}{2}\fl{d}{dt}\|e^\ep(t)\|_{L^2}^2\\
&=2\lambda\, \mathrm{Im} \int_{\Omega}
\left[u^\ep\ln(|u^\ep|)-u\ln(|u|)\right]\overline{e^\ep}(\bx,t)d\bx\\
 &\quad+\lambda\, \mathrm{Im} \int_{|u^\ep|<\ep}u^\ep\left[
f_n^\ep(|u^\ep|^2)-\ln (|u^\ep|^2)\right]\overline{e^\ep}(\bx,t)d\bx\\
&\le 2 |\lambda|\|e^\ep(t)\|_{L^2}^2+|\lambda|\Big|\int_{|u^\ep|<\ep}u^\ep\overline{e^\ep}
\left[Q_n^\ep(|u^\ep|^2)-\ln (|u^\ep|^2)+|u^\ep|^2(Q_n^\ep)'(|u^\ep|^2)\right]d\bx\Big|\\
&\le 2 |\lambda|\|e^\ep(t)\|_{L^2}^2+|\lambda|\,\Big|\int_{|u^\ep|<\ep}\overline{e^\ep}u^\ep
\Big[\int_{|u^\ep|^2}^{\ep^2} \fl{(s-|u^\ep|^2)^n}{s^{n+1}}ds-1+|u^\ep|^2(Q_n^\ep)'(|u^\ep|^2)\Big]d\bx\Big|\\
&= 2 |\lambda|\,\|e^\ep(t)\|_{L^2}^2+|\lambda|\,\Big|\int_{|u^\ep|<\ep}\overline{e^\ep}u^\ep
\Big[\int_{|u^\ep|^2}^{\ep^2} \fl{(s-|u^\ep|^2)^n}{s^{n+1}}ds-\left(1-\fl{|u^\ep|^2}{\ep^2}\right)^n\Big]d\bx\Big|\\
&\le 2|\lambda|\,\|e^\ep(t)\|_{L^2}^2+\ep|\lambda|\|e^\ep\|_{L^1}+
|\lambda|\,\Big|\int_0^{\ep^2} s^{-n-1}\int_{|u^\ep|^2<s}\overline{e^\ep}u^\ep
(s-|u^\ep|^2)^{n}d\bx ds\Big|\\
&\le 2|\lambda|\,\|e^\ep(t)\|_{L^2}^2+3\ep|\lambda|\|e^\ep\|_{L^1}.
\end{align*}
This yields the result.
\end{proof}

Invoking the same arguments as in \cite{bao2019error}, based on
the previous error estimate, and interpolation between $L^2$ and $H^2$,
we get the following error estimate.
\begin{proposition}\label{prop1}
If $\Omega$ has finite measure and $u_0\in H^2(\Omega)$, then for any $T>0$,
\[
\|u^\ep-u\|_{L^\infty([0,T]; L^2(\Omega))}\le C_1\ep,\quad
\|u^\ep-u\|_{L^\infty([0,T]; H^1(\Omega))}\le C_2\ep^{1/2},
\]
where $C_1$ depends on $|\lambda|$, $T$, $|\Omega|$, and $C_2$ depends in addition on $\|u_0\|_{H^2(\Omega)}$.
If $\Omega=\R^d$, $1\le d\le 3$ and $u_0\in
H^2(\R^d)\cap L^2_2$, then for any $T>0$, we have
\[\|u^\ep-u\|_{L^\infty([0,T]; L^2(\mathbb{R}^d))}\le D_1\ep^{\fl{4}{4+d}},\quad
\|u^\ep-u\|_{L^\infty([0,T]; H^1(\mathbb{R}^d))}\le D_2\ep^{\fl{2}{4+d}},\]
where $D_1$ and $D_2$ depend on $d$, $|\lambda|$, $T$, $\|u_0\|_{L^2_2}$ and $\|u_0\|_{H^2(\mathbb{R}^d)}$.
\end{proposition}
\begin{proof}
The proof is the same as that in \cite{bao2019error}. We just list the outline for the readers'  convenience.
When $\Omega$ is bounded, the convergence in $L^2$ follows from Gronwall's inequality by applying \eqref{gene} and the estimate
$\|v\|_{L^1}\le |\Omega|^{1/2}\|v\|_{L^2}$. The estimate in
$H^1$ follows form the Gagliardo-Nirenberg inequality $\|v\|_{H^1}\le C\|v\|^{1/2}_{L^2}\|v\|_{H^2}^{1/2}$ and the property \eqref{unib2}.
For $\Omega=\mathbb{R}^d$, the  convergence in ${L^2}$ can be established by Gronwall's inequality and the estimate (cf. \cite{bao2019error})
\[\|v\|_{L^1}\le C_d\|v\|_{L^2}^{1-d/4}\|v\|_{L^2_2}^{d/4}
\le C_d \left(\ep^{-1}\|v\|_{L^2}^2+\ep^{\frac{4-d}{4+d}}\|v\|_{L^2_2}^{\frac{2d}{4+d}}\right),\]
which is derived by the Cauchy-Schwarz inequality and Young's inequality. The convergence in $H^1$ can similarly derived by the Gagliardo-Nirenberg inequality.
\end{proof}

\subsection{Convergence of the energy}
\label{sec:cvenergy}
By construction,  the energy is conserved, i.e.,
\begin{equation}\label{engre}
E^\ep_n(u^\ep)=\int_{\Omega}\big[|\nabla u^\ep(\bx,t)|^2+\lambda
F_n^\ep(|u^\ep(\bx,t)|^2)\big]d\bx=E^\ep_n(u_0).
\end{equation}
For the convergence of the energy, we have the following estimate.
\begin{proposition}\label{energyc}
For $u_0\in H^1(\Omega)\cap L^\alpha(\Omega)$ with $\alpha\in (0,2)$, the energy $E_n^\ep(u_0)$ converges to $E(u_0)$ with
$$|E_n^\ep(u_0)-E(u_0)|\le |\lambda|\,\|u_0\|_{L^\alpha}^\alpha \fl{\ep^{2-\alpha}}{1-\alpha/2}.$$
In addition, for bounded $\Omega$, we have
\[|E_n^\ep(u_0)-E(u_0)|\le |\lambda|\, |\Omega|\,\ep^2.\]
\end{proposition}
\begin{proof} It can be deduced from the definition \eqref{engre} and
  \eqref{eq:Taylor} that
\begin{align*}
\left|E_n^\ep(u_0)-E(u_0)\right|&=|\lambda|\left|\int_{\Omega} [F(|u_0(\bx)|^2)-F_n^\ep(|u_0(\bx)|^2)]d\bx\right|\\
&=|\lambda|\left|\int_{|u_0(\bx)|<\ep} |u_0(\bx)|^2[Q(|u_0(\bx)|^2)-Q_n^\ep(|u_0(\bx)|^2)]d\bx\right|\\
&=|\lambda|\int_{|u_0(\bx)|<\ep} |u_0(\bx)|^2\int_{|u_0(\bx)|^2}^{\ep^2} s^{-n-1}(s-|u_0(\bx)|^2)^ndsd\bx\\
&=|\lambda|\int_0^{\ep^2} s^{-n-1}\int_{|u_0(\bx)|^2<s} |u_0(\bx)|^2(s-|u_0(\bx)|^2)^nd\bx ds.
\end{align*}
If $\Omega$ is bounded, we immediately get
\[\left|E_n^\ep(u_0)-E(u_0)\right|\le |\lambda|\, |\Omega|\, \ep^2.\]
For unbounded $\Omega$, one gets
\[\left|E_n^\ep(u_0)-E(u_0)\right|\le |\lambda| \int_0^{\ep^2} s^{-n-1}s^{n+1-\alpha/2} \|u_0\|_{L^\alpha}^\alpha ds=|\lambda|\,\|u_0\|_{L^\alpha}^\alpha \fl{\ep^{2-\alpha}}{1-\alpha/2},\]
which completes the proof.
\end{proof}

\begin{remark}
Recall that it was shown in \cite{bao2019error} that for the regularized model \eqref{RLSE0} with the energy density \eqref{1str}, the energy
\be
\label{Energ1}
\widetilde{E}^\ep(u_0)=\|\nabla u_0\|_{L^2}^2+\lambda\int_{\Omega}
\widetilde{F}^\ep(|u_0|^2)d\bx
\ee
converges to $E(u_0)$ with an error $O(\ep)$. For the regularization \eqref{RLSE1} with the energy density \eqref{1str} and the regularized energy
\be\label{Energ2}
\widehat{E}^\ep(u_0)=\|\nabla u_0\|_{L^2}^2+\lambda\int_{\Omega}
\widehat{F}^\ep(|u_0|^2)d\bx,
\ee
we have
\begin{align*}
\left|\widehat{E}^\ep(u_0)-E(u_0)\right|&=|\lambda|\,\left|\int_{\Omega} [F(|u_0(\bx)|^2)-\widehat{F}^\ep(|u_0(\bx)|^2)]d\bx\right|\\
&\hspace{-7mm}=|\lambda|\,\left|\int_\Omega \left[(\ep^2+|u_0|^2)\ln(\ep^2+|u_0|^2)-\ep^2\ln(\ep^2)-
|u_0|^2\ln(|u_0|^2) \right]d\bx\right|\\
&\hspace{-7mm}\le |\lambda|\,\ep^2\int_\Omega \ln\left(1+\frac{|u_0|^2}{\ep^2}\right)d\bx+|\lambda|
\int_\Omega|u_0|^2\ln\left(1+\frac{\ep^2}{|u_0|^2}\right)d\bx\\
&\hspace{-7mm}\le |\lambda|\,\ep^{2-\alpha}C(\alpha)\int_\Omega|u_0|^\alpha d\bx\\
&\hspace{-7mm}=|\lambda|\,\ep^{2-\alpha}C(\alpha)\|u_0\|_{L^\alpha}^\alpha,
\end{align*}
where we have used the inequality $\ln(1+x)\le C(\beta) x^\beta$ for $\beta\in (0,1]$ and $x\ge 0$. Hence for $u_0\in H^1(\Omega)\cap L^\alpha(\Omega)$ with $\alpha\in (0,2)$, we infer
\[\left|\widehat{E}^\ep(u_0)-E(u_0)\right|\le |\lambda|\,\ep^{2-\alpha}C(\alpha)\|u_0\|_{L^\alpha}^\alpha,\]
that is, the same convergence rate as $E_n^\ep$. Thus the newly
proposed local energy regularization $F_n^\ep$ is more accurate than
$\widetilde{F}^\ep$, and than $\widehat{F}^\ep$ in the case of bounded
domains, from the viewpoint of energy.
\end{remark}

\section{Regularized Lie-Trotter splitting methods}
\label{sec:lie}
In this section, we investigate approximation properties of the Lie-Trotter
splitting methods \cite{mclachlan, descombes, besse2002order} for
solving the regularized model \eqref{ERLSE} in one dimension (1D). Extensions to
higher dimensions are straightforward. To simplify notations, we set
$\lambda=1$.

\subsection{A time-splitting for \eqref{ERLSE}}
The operator splitting methods are based on a decomposition of the flow of \eqref{ERLSE}:
\[\p_t u^\ep=A(u^\ep)+B(u^\ep),\]
where
\[A(v)=i\Delta v,\quad B(v)=-i v f_n^\ep(|v|^2),\]
and the solution of the sub-equations
\be\label{lp}
\left\{
\begin{aligned}
&\p_t v(x,t)=A(v(x,t)),\quad x\in\Omega,\quad t>0,\\
&v(x,0)=v_0(x),
\end{aligned}\right.
\ee
\be\label{nlp}
\left\{
\begin{aligned}
&\p_t \og(x,t)=B(\og(x,t)),\quad x\in\Omega,\quad t>0,\\
&\og(x,0)=\og_0(x),
\end{aligned}\right.
\ee
where $\Omega=\mathbb{R}$ or $\Omega\subset \mathbb{R}$ is a bounded domain with homogeneous Dirichlet or periodic boundary condition on the boundary.
Denote the flow of \eqref{lp} and \eqref{nlp} by
\be\label{ABs}
v(\cdot,t)=\Phi_A^t(v_0)=e^{it\Delta}v_0,\quad
\og(\cdot,t)=\Phi_B^t(\og_0)=\og_0e^{-itf_n^\ep(|\og_0|^2)},\quad t\ge0.
\ee
As is well known, the flow $\Phi_A^t$ satisfies the isometry relation
\be\label{Ap}
\|\Phi_A^t(v_0)\|_{H^s}=\|v_0\|_{H^s},\quad \forall s\in
\mathbb{R},\quad \forall t\ge 0.
\ee
Regarding the flow $\Phi_B^t$, we have the following properties.
\begin{lemma}
Assume $\tau>0$ and $\og_0\in H^1(\Omega)$, then
\be\label{Bp}
\|\Phi_B^\tau(\og_0)\|_{L^2}=\|\og_0\|_{L^2},\quad
\|\Phi_B^\tau(\og_0)\|_{H^1}\le (1+6\tau)\,\|\og_0\|_{H^1}.
\ee
For $v$, $w\in L^2(\Omega)$,
\be\label{phibl}
\|\Phi_B^\tau(v)-\Phi_B^\tau(w)\|_{L^2}\le (1+4n\tau)\,\|v-w\|_{L^2}.\ee
\end{lemma}
\emph{Proof.} By direct calculation, we get
\[\p_x\Phi_B^\tau(\omega_0)=e^{-i\tau f_n^\ep(|\omega_0|)^2}\left[\p_x
\omega_0-i\tau(f_n^\ep)'(|\omega_0|^2)(\og_0^2\p_x
\overline{\omega_0}+|\omega_0|^2\p_x \omega_0)\right],\]
which immediately gives \eqref{Bp} by recalling \eqref{fd}.
We claim that for any $x\in\Omega$,
\[|\Phi_B^\tau(v)(x)-\Phi_B^\tau(w)(x)|\le (1+4n\tau)\,|v(x)-w(x)|.\]
Assuming, for example, $|v(x)|\le|w(x)|$, by inserting a term
$v(x)e^{-i\tau f_n^\ep(|w(x)|)^2}$, we can get
\begin{align*}
&\quad|\Phi_B^\tau(v)(x)-\Phi_B^\tau(w)(x)|\\
&=\left|v(x)e^{-i\tau f_n^\ep(|v(x)|)^2}-
w(x)e^{-i\tau f_n^\ep(|w(x)|)^2}\right|\\
&=\Big|v(x)-w(x)+v(x)\Big(e^{i\tau [f_n^\ep(|w(x)|^2)-f_n^\ep(|v(x)|^2)]}-1\Big)\Big|\\
&\le|v(x)-w(x)|+2|v(x)|\,\Big|\sin
\Big(\fl{\tau}{2}\left[f_n^\ep(|w(x)|^2)-f_n^\ep(|v(x)|^2)\right]\Big)\Big|\\
&\le |v(x)-w(x)|+\tau|v(x)|\,|f_n^\ep(|w(x)|^2)-f_n^\ep(|v(x)|^2)|\\
&\le(1+4n\tau)\,|v(x)-w(x)|,
\end{align*}
where we have used the estimate \eqref{fl}.
When $|v(x)|\ge |w(x)|$, the same inequality can be obtained by
exchanging $v$ and $w$ in the above computation. Thus the proof for \eqref{phibl} is complete.
\hfill $\square$ \bigskip

\subsection{Error estimates for $\Phi^\tau=\Phi_A^\tau\Phi_B^\tau$}
We consider the Lie-Trotter splitting
\be\label{LT}
u^{\ep, k+1}=\Phi^\tau(u^{\ep, k})=\Phi_A^\tau(\Phi_B^\tau(u^{\ep, k})),\quad k\ge 0;\quad u^{\ep,0}=u_0, \quad \tau>0.
\ee
For $u_0\in H^1(\Omega)$, it follows from \eqref{Ap} and \eqref{Bp} that
\be\label{unp}
\begin{split}
&\|u^{\ep,k}\|_{L^2}=\|u^{\ep,k-1}\|_{L^2}\equiv \|u^{\ep,0}\|_{L^2}= \|u_0\|_{L^2} ,\\
&\|u^{\ep,k}\|_{H^1}\le (1+6\tau)\,\| u^{\ep,k-1}\|_{H^1}\le e^{6k\tau}\| u_0\|_{H^1}, \quad k\ge0.
\end{split}
\ee
\begin{theorem}\label{thmlt}
Let $T>0$ and $\tau_0>0$ be given constants. Assume that the solution of \eqref{ERLSE} satisfies
$u^\ep\in  L^\infty([0,T]; H^1(\Omega))$ and the time step $\tau\le \tau_0$. Then there exists
$0<\ep_0<1$ depending on $n$, $\tau_0$ and $M:=\|u^\ep\|_{L^\infty([0,T]; H^1(\Omega))}$ such that
when $\ep\le \ep_0$ and $t_k:=k\tau\le T$, we have
\be\label{li1}
\|u^{\ep,k}-u^\ep(t_k)\|_{L^2}\le C\left(n, \tau_0, T, M \right)\ln(\ep^{-1})\tau^{1/2}.\ee
\end{theorem}
\begin{proof} Denote the exact flow of \eqref{ERLSE} by
$u^\ep(t)=\Psi^t(u_0)$.  First, we establish the local error for $v\in H^1(\Omega)$:
\be\label{local1}
\|\Psi^\tau(v)-\Phi^\tau(v)\|_{L^2}\le C(n, \tau_0) \|v\|_{H^1}\ln(\ep^{-1})\tau^{3/2}, \quad \tau\le \tau_0,
\ee
when $\ep$ is sufficiently small.
Note that definitions imply
\begin{align*}
&i\p_t\Psi^t(v)+\Delta \Psi^t (v)=\Psi^t (v)f_n^\ep(|\Psi^t (v)|^2),\\
&i\p_t\Phi^t(v)+\Delta\Phi^t(v)=\Phi_A^t\left(\Phi_B^t(v)f_n^\ep(|\Phi_B^t(v)|^2)\right).
\end{align*}
Denoting $\mathcal E^t(v)=\Psi^t(v)-\Phi^t(v)$, we have
\be\label{erq}
i\p_t\mathcal E^t(v)+\Delta \mathcal E^t(v)=\Psi^t (v)f_n^\ep(|\Psi^t (v)|^2)-\Phi_A^t\left(\Phi_B^t(v)f_n^\ep(|\Phi_B^t(v)|^2)\right).
\ee
Multiplying \eqref{erq} by $\overline{\mathcal E^t (v)}$, integrating in space
and taking the imaginary part, we get
\begin{align*}
\fl{1}{2}\fl{d}{dt}\|\mathcal E^t(v)\|_{L^2}^2&=\mathrm{Im}\left(\Psi^t (v)f_n^\ep(|\Psi^t(v)|^2)-\Phi_A^t\left(\Phi_B^t(v)f_n^\ep(|\Phi_B^t(v)|^2)\right), \mathcal E^t(v)\right)\\
&=\mathrm{Im}\left(\Psi^t (v)f_n^\ep(|\Psi^t (v)|^2)-\Phi^t (v)f_n^\ep(|\Phi^t (v)|^2), \mathcal E^t(v)\right)\\
&\quad+\mathrm{Im}\left(\Phi^t (v)f_n^\ep(|\Phi^t(v)|^2)-\Phi_A^t\left(\Phi_B^t(v)f_n^\ep(|\Phi_B^t(v)|^2)\right), \mathcal E^t(v)\right)\\
&\le 4n\|\mathcal E^t(v)\|_{L^2}^2\\
&\quad+\left\|\Phi^t (v)f_n^\ep(|\Phi^t(v)|^2)-\Phi_A^t\left(\Phi_B^t(v)f_n^\ep(|\Phi_B^t(v)|^2)\right)\right\|_{L^2}
\|\mathcal  E^t(v)\|_{L^2},
\end{align*}
where we have used \eqref{gl} and the scalar product is the standard one in $L^2$: $(u, w)=\int_\Omega u(x)\overline{w(x)}dx$.
This implies
\be\label{tmp1}
\fl{d}{dt}\|\mathcal E^t(v)\|_{L^2}\le 4n\|\mathcal E^t(v)\|_{L^2}+J_1+J_2,
\ee
where
\begin{align*}
J_1&=\|\Phi^t(v)f_n^\ep(|\Phi^t(v)|^2)-\Phi_B^t(v) f_n^\ep(|\Phi_B^t(v)|^2)\|_{L^2},\\
J_2&=\|\Phi_B^t(v)f_n^\ep(|\Phi_B^t(v)|^2)-
\Phi_A^t\left(\Phi_B^t(v)f_n^\ep(|\Phi_B^t(v)|^2)\right)\|_{L^2}.
\end{align*}

To estimate $J_1$ in \eqref{tmp1}, first we try to find the bound of $\|\Phi^t(v)\|_{L^\infty}, \|\Phi_B^t(v)\|_{L^\infty}$.
It follows from \eqref{Ap} and \eqref{Bp} that
\[\|\Phi^t(v)\|_{H^1}=\|\Phi_B^t(v)\|_{H^1}\le (1+6t)\|v\|_{H^1}\le (1+6t_0)\|v\|_{H^1},\quad t\le t_0.\]
Hence by Sobolev embedding, we have
\be\label{phib}
\|\Phi^t(v)\|_{L^\infty}\le c(1+6t_0)\|v\|_{H^1},\quad
\|\Phi_B^t(v)\|_{L^\infty}\le c(1+6t_0)\|v\|_{H^1},\ee
where $c$ is the constant in the Sobolev inequality $\|\omega\|_{L^\infty}\le c\|\omega\|_{H^1}$. Next we claim that for $y$, $z$ satisfying $|y|, |z|\le D$, it can be established that
\be\label{vp}
|yf_n^\ep(|y|^2)-zf_n^\ep(|z|^2)|\le 4\ln(\ep^{-1})|y-z|,\ee
when $\ep$ is sufficiently small.
It follows from \eqref{fb} that $|f_n^\ep(|y|^2)|\le 2 +\ln(n\ep^{-2})$, when $|y|\le D$ and $\ep\le\sqrt{n}/D$.
Assuming, for example, $0<|z|\le|y|$, and applying \eqref{fl}, we get
\begin{align*}
|yf_n^\ep(|y|^2)-zf_n^\ep(|z|^2)|&=|(y-z)f_n^\ep(|y|^2)|
+|z||f_n^\ep(|y|^2)-f_n^\ep(|z|^2)|\\
&\le (2+\ln(n\ep^{-2}))|y-z|+|z|\fl{4n|y-z|}{|z|}\\
&\le 2 (3n+\ln(\ep^{-1}))|y-z|\\
&\le 4\ln(\ep^{-1})|y-z|,
\end{align*}
when $\ep\le \widetilde{\ep}:=\min\{\sqrt{n}/D, e^{-3n}\}$.
The case when $y=0$ or $z=0$ can be handled similarly.
Recalling \eqref{phib}, taking $D=c(1+6t_0)\|v\|_{H^1}$, we obtain, when $\ep\le \ep_1:=\min\{\frac{\sqrt{n}}{c(1+6t_0)\|v\|_{H^1}}, e^{-3n}\}$,
\begin{align}
J_1&\le 4\ln(\ep^{-1})\|\Phi^t(v)-\Phi_B^t(v)\|_{L^2}\nn\\
&\le 4\ln(\ep^{-1})\sqrt{2t}\|\Phi_B^t(v)\|_{H^1}\nn\\
&\le 6 \ln(\ep^{-1})\sqrt{t}\|v\|_{H^1},\label{tmp3}
\end{align}
where we have used the estimate
\be\label{ld}
\left\|\og-\Phi_A^t(\og)\right\|_{L^2}\le \sqrt{2t}\,\|\og\|_{H^1},
\ee
as in \cite{bao2018}, instead of the estimate from \cite{besse2002order},
\[\left\|\og-\Phi_A^t(\og)\right\|_{L^2}\le 2t\|\og\|_{H^2},\]
which in our case yields an extra $1/\ep$ factor in the error estimate.

To estimate $J_2$, we first claim that
\be\label{c1}
\|\Phi_B^t(v)f_n^\ep(|\Phi_B^t(v)|^2)\|_{H^1}\le 6\ln(\ep^{-1})(1+3t_0)\|v\|_{H^1},
\ee
when $\ep\le\ep_1$ and $t\le t_0$. Recalling that
$$\Phi_B^t(v)f_n^\ep(|\Phi_B^t(v)|^2)=vf_n^\ep(|v|^2)e^{-itf_n^\ep(|v|^2)},$$
and $|f_n^\ep(|v|^2)|\le 3\ln(\ep^{-1})$, when $\ep\le \ep_1$,
this implies
\[\|(\Phi_B^t(v))f_n^\ep(|\Phi_B^t(v)|^2)\|_{L^2}\le 3 \ln(\ep^{-1})\|v||_{L^2}.\]
Noticing that
\begin{align*}
\p_x [\Phi_B^t(v)f_n^\ep(|\Phi_B^t(v)|^2)]&=e^{-itf_n^\ep(|v|^2)}\left[
v_x f_n^\ep(|v|^2)\right.\\
&\qquad\qquad\qquad\left.+(1-itf_n^\ep(|v|^2))(f_n^\ep)'(|v|^2)(v^2\overline{v_x}
+|v|^2v_x)\right],
\end{align*}
which together with \eqref{fd} yields
\[|\p_x [\Phi_B^t(v)f_n^\ep(|\Phi_B^t(v)|^2)]|\le \left[6+3\ln(\ep^{-1})(1+6t_0)\right]|v_x|\le 6\ln(\ep^{-1})(1+3t_0)|v_x|,\]
which immediately gives \eqref{c1}. Applying \eqref{ld} again entails
\be\label{tmp2}
\begin{aligned}
J_2\le \sqrt{2t}\,
\|(\Phi_B^t(v))f_n^\ep(|\Phi_B^t(v)|^2)\|_{H^1}\le 9\ln(\ep^{-1})(1+3t_0)\sqrt{t}\|v\|_{H^1},
\end{aligned}
\ee
for $\ep\le \ep_1$ and $t\le t_0$. Combining \eqref{tmp1}, \eqref{tmp3} and \eqref{tmp2}, we get
\[\fl{d}{dt}\|\mathcal E^t(v)\|_{L^2}\le 4n\|\mathcal E^t (v)\|_{L^2}+15(1+2t_0)\ln(\ep^{-1}))\sqrt{t}\|v\|_{H^1}.\]
Invoking Gronwall's inequality, we have
\begin{align*}
\|\mathcal E^\tau (v)\|_{L^2}&\le e^{4n\tau}\left[\|\mathcal E^0 (v)\|_{L^2}+15(1+2\tau_0)\ln(\ep^{-1})\|v\|_{H^1}\int_0^\tau \sqrt{s}ds\right]\\
&\le 30(1+2\tau_0)e^{4n\tau}\|v\|_{H^1}\ln(\ep^{-1})\tau^{3/2}\\
&\le C(n, \tau_0)\|v\|_{H^1}\ln(\ep^{-1})\tau^{3/2},
\end{align*}
when $\tau\le \tau_0$ and $\ep\le \ep_0:=\min\{\frac{\sqrt{n}}{c(1+6\tau_0)M}, e^{-3n}\}$ depending on $\tau_0$, $n$ and
$M=\|u^\ep\|_{L^\infty([0, T]; H^1)}$, which completes the proof for \eqref{local1}.

Next we infer the stability analysis for the operator $\Phi^t$:
\be\label{stab}
\|\Phi^\tau(v)-\Phi^\tau(w)\|_{L^2}\le (1+4n\tau)\|v-w\|_{L^2},\quad
\mathrm{for}\quad v, w\in L^2(\Omega).
\ee
Noticing that $\Phi_A^\tau$ is a linear isometry on $H^s(\Omega)$, \eqref{phibl} gives \eqref{stab} directly. Thus the error \eqref{li1} can be established by combining the local error \eqref{local1}, the stability property \eqref{stab} and a standard argument \cite{besse2002order, bao2018}:
\begin{align*}
&\hspace{-4mm}\|u^{\ep, k}-u^\ep(t_k)\|_{L^2}=\|\Phi^\tau(u^{\ep, k-1})-\Psi^\tau(u^\ep(t_{k-1})\|_{L^2}\\
&\le \|\Phi^\tau(u^{\ep, k-1})-\Phi^\tau(u^\ep(t_{k-1}))\|_{L^2}+\|\Phi^\tau(u^\ep(t_{ k-1}))-\Psi^\tau(u^\ep(t_{k-1}))\|_{L^2}\\
&\le (1+4n\tau)\|u^{\ep, k-1}-u^\ep(t_{k-1})\|_{L^2}+C(n, \tau_0)\ln(\ep^{-1})\tau^{3/2}\|u^\ep(t_{k-1})\|_{H^1}\\
&\le (1+4n\tau)\|u^{\ep, k-1}-u^\ep(t_{k-1})\|_{L^2}+MC(n, \tau_0)\ln(\ep^{-1})\tau^{3/2}\\
&\le (1+4n\tau)^2\|u^{\ep, k-2}-u^\ep(t_{k-2})\|_{L^2}+MC(n, \tau_0)\ln(\ep^{-1})\tau^{3/2}\left[1+(1+4n\tau)\right]\\
&\le \ldots\\
&\le (1+4n\tau)^k\|u^{\ep, 0}-u_0\|_{L^2}+MC(n, \tau_0)\ln(\ep^{-1})\tau^{3/2}\sum\limits_{j=0}^{k-1}(1+4n\tau)^j\\
&\le C(n, \tau_0, T, M)\ln(\ep^{-1})\tau^{1/2},
\end{align*}
which completes the proof.
\end{proof}

\begin{remark}\label{rem41}
As established in Theorem \ref{theo:cauchy}, for an arbitrarily large fixed $T>0$, we have $u^\ep\in L^\infty([0, T]; H^1(\Omega))$  as soon as $u_0\in H^1(\Omega)$ when $\Omega $ is bounded. More specifically,
\[
M=\|u^\ep\|_{L^\infty([0,T]; H^j)} \le C\left(n, \lambda, T, \|u_0\|_{H^j}\right), \quad j=1, 2,
\]
for a constant $C$ independent of $\ep$. When
$\Omega=\R^d$, we require in addition $u_0\in L^2_\alpha$ for some $0<\alpha\le 1$ and $C$ depends additionally on $\|u_0\|_{L^2_\alpha}$. Hence the constant in \eqref{li1} as well as \eqref{li2} in Theorem \ref{thmlt2} is independent of $\ep$.
\end{remark}

\begin{remark}
By applying similar arguments as in \cite{bao2018}, for $d=2, 3$, the error estimate \eqref{li1} can be established under a more restrictive condition $u^\ep\in L^\infty([0,T];H^2(\Omega))$,
in which case $\ep_0$ depends on $n$ and
$\|u^\ep\|_{L^\infty(0,T; H^2(\Omega))}$, and $\|\Phi_B^t(v)\|_{H^2}$ has to be further investigated due to the Sobolev inequality $H^2(\Omega)\hookrightarrow L^\infty(\Omega)$. For details, we refer to \cite{bao2018}.
\end{remark}

\subsection{Error estimates for $\Phi^\tau=\Phi_B^\tau\Phi_A^\tau$}
We consider another Lie-Trotter splitting
\be\label{LT1}
u^{\ep, k+1}=\Phi^\tau(u^{\ep, k})=\Phi_B^\tau(\Phi_A^\tau(u^{\ep, k})),\quad k\ge 0;\quad u^{\ep,0}=u_0, \quad \tau\in (0,\tau_0].
\ee
In the same fashion as above, we have
\be\label{unp1}
\|u^{\ep,k}\|_{L^2}=\|u_0\|_{L^2} ,\quad
\|u^{\ep,k}\|_{H^1}\le e^{6k\tau}\| u_0\|_{H^1}, \quad k\ge0.
\ee
\begin{theorem}\label{thmlt2}
Let $T>0$. Assume that the solution of \eqref{ERLSE} satisfies
$u^\ep\in  L^\infty([0,T];H^2(\Omega))$. Then there exists
$\ep_0>0$ depending on $n$, $\tau_0$ and $M=\|u^\ep\|_{L^\infty([0,T]; H^1(\Omega))}$ such that
when $\ep\le \ep_0$ and $k\tau\le T$, we have
\be\label{li2}
\|u^{\ep,k}-u^\ep(t_k)\|_{L^2}\le C\left(n, \tau_0, T,
  \|u^\ep\|_{L^\infty([0,T];H^2(\Omega))}\right)\frac{\tau}{\ep},\ee
where $C(\cdot,\cdot, \cdot, \cdot)$ is independent of $\ep$.
\end{theorem}
\begin{proof}  First, we prove the local error estimate: for $v_0\in H^1(\Omega)$,
\be\label{local2}
\|\Psi^\tau(v_0)-\Phi^\tau(v_0)\|_{L^2}\le C(n, \|v_0\|_{H^2})\fl{\tau^{2}}{\ep} , \quad \ep\le\widetilde{\ep}_0,
\ee
where $\Phi^\tau=\Phi_B^\tau\Phi^\tau_A$, $\Psi^\tau(v_0)$ is the exact flow of \eqref{ERLSE} with initial data $v_0$ and $C(\cdot, \alpha)$ is increasing with respect to $\alpha$ and $\widetilde{\ep}_0$ depends on $n$ and $\|v_0\|_{H^1}$.
We start from the Duhamel formula for $v(t)=\Psi^t(v_0)$:
\be\label{duh}
\Psi^t(v_0)=e^{it\Delta} v_0+\int_0^t e^{i(t-s)\Delta} B(v(s))ds.\ee
Recall
\be\label{bex}
B(v(s))=B(e^{is\Delta}v_0)+\int_0^s dB(e^{i(s-y)\Delta}v(y))[e^{i(s-y)\Delta}B(v(y))]dy,\ee
which is the variation-of-constants formula \[B(g(s))-B(g(0))=\int_0^s dB(g(y))[g'(y)]dy,\quad g(y)=e^{i(s-y)\Delta}v(y).\]
Here $dB(\cdot)[\cdot]$ is the G\^{a}teaux derivative:
\begin{align}
dB(w_1)[w_2]&=\lim\limits_{\delta\rightarrow 0}\frac{B(w_1+\delta w_2)-B(w_1)}{\delta}\nn\\
&=-i w_2 f_n^\ep(|w_1|^2)-iw_1 (f_n^\ep)'(|w_1|^2)[w_1 \overline{w_2}+\overline{w_1}w_2].\label{dBdef}
\end{align}
Plugging \eqref{bex} into \eqref{duh} with $t=\tau$, we get
\[\Psi^\tau(v_0)=e^{i\tau \Delta}v_0+\int_0^\tau e^{i(\tau-s)\Delta}B(e^{is\Delta}v_0)ds+e_1,\]
where
\[e_1=\int_0^\tau\int_0^s e^{i(\tau-s)\Delta}dB(e^{i(s-y)\Delta}v(y))[e^{i(s-y)\Delta}B(v(y))]dyds.\]
On the other hand, for the Lie splitting $\Phi^\tau(v_0)=\Phi_B^\tau\Phi_A^\tau(v_0)$,
applying the first-order Taylor expansion
\[\Phi_B^\tau(w)=w+\tau B(w)+\tau^2\int_0^1(1-s) dB(\Phi_B^{s\tau}(w))[B(\Phi_B^{s\tau}(w))]ds,\]
for $w=\Phi_A^\tau(v_0)=e^{i\tau \Delta}v_0$, we get
\[\Phi^\tau(v_0)=\Phi_B^\tau\Phi_A^\tau(v_0)=e^{i\tau\Delta}v_0+\tau B(e^{i\tau\Delta}v_0)+e_2,\]
with
\[e_2=\tau^2\int_0^1(1-s)dB(\Phi_B^{s\tau}(e^{i\tau\Delta}v_0))
[B(\Phi_B^{s\tau}(e^{i\tau\Delta}v_0))]ds.\]
Thus
\[\Psi^\tau(v_0)-\Phi^\tau(v_0)=e_1-e_2+e_3,\]
where
\[e_3=\int_0^\tau e^{i(\tau-s)\Delta}B(e^{is\Delta}v_0)ds-\tau B(e^{i\tau\Delta}v_0).\]
Noticing that $e_3$ is the quadrature error of the rectangle rule approximating the integral on $[0,\tau]$ of the function $g(s)=e^{i(\tau-s)\Delta}B(e^{is\Delta}v_0)$, this implies
\[e_3=-\tau^2\int_0^1 \theta g'(\theta \tau)d\theta,\]
where $g'(s)=-e^{i(\tau-s)\Delta}[A, B](e^{is\Delta}v_0)$, with
\begin{align*}
[A, B](w)&=dA(w)[Bw]-dB(w)[Aw]=i\Delta(Bw)-dB(w)[Aw]\\
%&=w_{xx}f_n^\ep(|w|^2)+2(f_n^\ep)'(|w|^2)(w_x^2\overline{w}+w|w_x|^2)+(f_n^\ep)'(|w|^2)
%(w_{xx}|w|^2+w^2\overline{w_{xx}}+2w|w_x|^2)\\
%&+w(f_n^\ep)''(|w|^2)(w_x\overline{w_x}+w\overline{w_x})^2-w_{xx}f_n^\ep(|w|^2)+2w (f_n^\ep)'(|w|^2)(w\overline{w_{xx}}-\overline{w}w_{xx})\\
&=(f_n^\ep)'(|w|^2)(2w_x^2\overline{w}+4w|w_x|^2+3w^2\overline{w_{xx}}-|w|^2w_{xx})\\
&\quad+w(f_n^\ep)''(|w|^2)(w_x\overline{w}+w\overline{w_x})^2,
\end{align*}
by recalling \eqref{dBdef} and
\be\label{dAdef}
dA(w_1)[w_2]=\lim\limits_{\delta\rightarrow 0}\frac{A(w_1+\delta w_2)-A(w_1)}{\delta}=i\Delta w_2.\ee
Applying \eqref{fd}, we get
\[\left|[A, B](w)\right|\le \fl{12n+6n^2}{\ep}|w_x|^2+12|w_{xx}|,\]
which implies
\begin{align*}
\|[A, B](w)\|_{L^2}&\le \fl{12n+6n^2}{\ep}\|w_x\|_{L^4}^2+12\|w_{xx}\|_{L^2}\\
&\le \fl{12n+6n^2}{\ep}\|w_x\|_{L^\infty}\|w_x\|_{L^2}+12\|w_{xx}\|_{L^2}\\
&\le 12\|w\|_{H^2}+\frac{12cn^2}{\ep}\|w\|_{H^2}^2,
\end{align*}
where we have used $n\ge 2$ and the Sobolev embedding $\|w\|_{L^\infty}\le c\|w\|_{H^1}$ for $d=1$. This yields that for any $s\in [0,1]$,
\[\|g'(s)\|_{L^2}=\|[A, B](e^{is\Delta} v_0)\|_{L^2}
\le 12\|v_0\|_{H^2}(1+cn^2\|v_0\|_{H^2}/\ep),\]
which immediately gives
\be\label{e3}
\|e_3\|_{L^2}\le \tau^2\int_0^1\|g'(\theta \tau)\|_{L^2}d\theta
\le 12\|v_0\|_{H^2}(1+cn^2\|v_0\|_{H^2}/\ep)\tau^2.\ee

Next we estimate $e_1$ and $e_2$. In view of \eqref{fd}, we have
\[\|dB(w_1)[w_2]\|_{L^2}\le (8+\ln(n\ep^{-2}))\|w_2\|_{L^2}, \]
when $\ep\le \widetilde{\ep}:=\sqrt{n}/\|w_1\|_{L^\infty}$.
Thus one gets
\begin{align*}
\|dB(e^{i(s-y)\Delta} v(y))[e^{i(s-y)\Delta} B(v(y))]\|_{L^2}&\le (8+\ln(n\ep^{-2}))\|e^{i(s-y)\Delta}B(v(y))\|_{L^2}\\
&=(8+\ln(n\ep^{-2}))\|B(v(y))\|_{L^2},
\end{align*}
when $\ep\le \ep_1=\sqrt{n}/\|e^{i(s-y)\Delta} v(y)\|_{L^\infty}$.
By Sobolev embedding,
\be\label{qq1}
\|e^{i(s-y)\Delta}v(y)\|_{L^\infty}\le c\|e^{i(s-y)\Delta}v(y)\|_{H^1}
=c\|\Psi^y(v_0)\|_{H^1},\ee
thus when $\ep\le \ep_2:=\fl{\sqrt{n}/c}{\max\limits_{y\in[0,\tau]}\|\Psi^y(v_0)\|_{H^1}}$, we have
\begin{align}
\|e_1\|_{L^2}&\le \int_0^\tau\int_0^s
\|dB(e^{i(s-y)\Delta} v(y))[e^{i(s-y)\Delta} B(v(y))]\|_{L^2}dyds\nn\\
&\le (8+\ln(n\ep^{-2}))\int_0^\tau\int_0^s
\|B(v(y))\|_{L^2}dyds\nn\\
&\le (8+\ln(n\ep^{-2}))\tau^2 \max\limits_{0\le y\le\tau}\|v(y)f_n^\ep(|v(y)|^2)\|_{L^2}\nn\\
&\le (8+\ln(n\ep^{-2}))^2\tau^2 \max\limits_{0\le y\le\tau}\|v(y)\|_{L^2}\nn\\
&= (8+\ln(n\ep^{-2}))^2\|v_0\|_{L^2}\tau^2.\label{e1}
\end{align}
Similarly, by recalling
\[\|\Phi_B^{s\tau}(e^{i\tau\Delta}v_0)\|_{L^\infty}=\|e^{i\tau\Delta}v_0\|_{L^\infty}
\le c\|v_0\|_{H^1},\]
when $\ep\le \ep_3:=\sqrt{n}/(c\|v_0\|_{H^1})$,
\begin{align}
\|e_2\|_{L^2}&\le (8+\ln(n\ep^{-2}))\tau^2\int_0^1\|B(\Phi_B^{s\tau}(e^{i\tau\Delta}v_0))\|_{L^2}ds\nn\\
&\le (8+\ln(n\ep^{-2}))^2\tau^2 \int_0^1\|\Phi_B^{s\tau}(e^{i\tau\Delta}v_0)\|_{L^2}ds\nn\\
&=(8+\ln(n\ep^{-2}))^2 \|v_0\|_{L^2}\tau^2.\label{e2}
\end{align}
Combining \eqref{e3}, \eqref{e1} and \eqref{e2}, when $\ep\le \widetilde{\ep}_0=\min\{\ep_2, \ep_3\}=\ep_2$, we have
\begin{align*}
\|\Psi^\tau(v_0)-\Phi^\tau(v_0)\|_{L^2}&\le \tau^2\|v_0\|_{H^2}\big[c_1
+c_2\ln(n\ep^{-2})+c_3(\ln(n\ep^{-2}))^2+\frac{12cn^2}{\ep}\|v_0\|_{H^2}\big]\\
&\le \tau^2\|v_0\|_{H^2}\big[\frac{c_1}{\ep}
+\frac{C_2n^{1/2}}{\ep}+\frac{12cn^2}{\ep}\|v_0\|_{H^2}\big]\\
&\le C(n, \|v_0\|_{H^2})\frac{\tau^2}{\ep},
\end{align*}
where we have employed the inequalities $\ln(x)\le Cx^{1/2}$ and $\ln(x)\le Cx^{1/4}$ for $x\in [1, \infty)$. Hence \eqref{local2} is established.

Similarly the stability can be yielded by \eqref{phibl}:
\be\label{stab1}
\|\Phi^\tau(v)-\Phi^\tau(w)\|_{L^2}\le (1+4n\tau)\|\Phi_A^\tau(v-w)\|_{L^2}=
(1+4n\tau)\|v-w\|_{L^2},
\ee
for $v, w\in L^2(\Omega)$. Denote
$\ep_0=\frac{\sqrt{n}/c}{\|u\|_{L^\infty([0, T]; H^1)}}$, then by applying similar arguments in the proof of Theorem \ref{thmlt},  we can get the error estimate \eqref{li2}.
\end{proof}

\bigskip

\begin{remark}
For $d=2, 3$, the error estimate \eqref{li2} can be established with $\ep_0$ depending on $n$, $\tau_0$ and
$\|u^\ep\|_{L^\infty([0,T]; H^2(\Omega))}$ by noticing that $H^2(\Omega)\hookrightarrow L^\infty(\Omega)$ and $H^2(\Omega)\hookrightarrow W^{1,4}(\Omega)$ for $d=2, 3$.
\end{remark}

\bigskip

\begin{remark}[Strang splitting]
\label{rem:ST_Error}
When considering a Strang splitting,
    \be
     \label{ST}
u^{\ep,k+1}= \Phi_B^{\tau/2}\left(\Phi_A^\tau
    \left(\Phi_B^{\tau/2}(u^{\ep,k})\right)\right),\ \
\mathrm{or}\ \ u^{\ep,k+1}= \Phi_A^{\tau/2}\left(\Phi_B^\tau
    \left(\Phi_A^{\tau/2}(u^{\ep,k})\right)\right),
  \ee
by applying similar but more intricate arguments as above, we can prove
the error bound
\[\|u^{\ep,k}-u^\ep(t_k)\|_{L^2}\le C\left(n, \tau_0, T,
  \|u^\ep\|_{L^\infty([0,T];H^4(\Omega))}\right)\,\frac{\tau^2}{\ep^3},\]
under the assumption that $u^\ep\in  L^\infty([0,T];H^4(\Omega))$.
\end{remark}

\begin{remark}
In view of Theorem~\ref{theo:cauchy}, Theorems~\ref{thmlt} and \ref{thmlt2} rely on a regularity that we
  know is available. On the other hand, the regularity assumed in the above
  remark on Strang splitting is unclear in general, in the sense that
  we don't know how to bound $u^\ep$ in $ L^\infty([0,T];H^4(\Omega))$.
  \end{remark}

\section{Numerical results}\label{sec:num}
In this section, we first test the convergence rate of the local energy regularized model
\eqref{ERLSE} and compare it with the  other two \eqref{RLSE0} and \eqref{RLSE1}.
We then test the order of accuracy of the regularized Lie-Trotter splitting (LTSP) schemes \eqref{LT} and  \eqref{LT1} and
Strang splitting (STSP) scheme \eqref{ST}.  To simplify the
presentation, we unify the regularized models  \eqref{RLSE0},
\eqref{RLSE1}  and \eqref{ERLSE}  as follows:
\be
\label{RLSE_Unified}
\left\{
\begin{aligned}
&i\p_t u^\ep(\bx,t)+\Delta u^\ep(\bx,t)=\lambda
u^\ep(\bx,t)f_{\rm reg}^\ep(|u^\ep(\bx,t)|^2),\quad \bx\in \Omega, \quad
t>0,\\
&u^\ep(\bx,0)=u_0(\bx),\quad \bx\in \overline{\Omega}.
\end{aligned}
\right.
\ee
%And the corresponding regularized energy  $E_n^\ep(u^\ep)$, $\widehat{E}^\ep(u^\ep)$ and $\widetilde{E}^\ep(u^\ep)$ are unified as $E_{\rm reg}^{\ep}(u^\ep)$.
With the regularized nonlinearity $f_{\rm reg}^\ep(\rho)$ being chosen as  $\widetilde{f}^\ep$,
$\widehat{f}^\ep$ and $f_n^\ep$,   \eqref{RLSE_Unified}
 corresponds to the regularized models \eqref{RLSE0}, \eqref{RLSE1}  and \eqref{ERLSE}, respectively.
In practical computation, we impose periodic boundary condition on $\Og$ and employ the  standard
Fourier pseudo-spectral method \cite{bao2002,bao2003,bao2018}  for  spatial discretization.
The details are omitted here for brevity.

%Combining the Lie-Trotter splitting and the  Strang splitting  method for time discretization, we obtain the Lie-Trotter
%splitting and the  Strang splitting pseudo-spectral methods.

Hereafter, unless specified, we consider the following  Gaussian initial data in $d$-dimension ($d=1,2$), i.e., $u_0(\bx)$ is chosen as
\be
u_0(\bx)=b_d\, e^{i\bx\cdot\bm{v} +\frac{\lambda}{2}|\bx|^2}, \qquad \bx\in {\mathbb R}^d.
\ee
In this case, the LogSE \eqref{LSE}  admits
the moving Gausson solution
\be
\label{Gausson}
u(\bx,t)=b_d\, e^{i(\bx\cdot\bm{v}-(a_d+|\bm{v}|^2)t)+\fl{\lambda}{2}|\bx-2\bm{v}t|^2},\qquad \bx\in {\mathbb R}^d, \quad t\ge0,
\ee
with $a_d=-\lambda\, (d-\ln|b_d|^2).$
In this paper, we let $\lambda=-1$, $b_d=1/\sqrt[4]{-\lambda\pi}$ and choose   $\Omega=[-16, 16]^d$. Moreover, we fix $v=1$ and $\bm{v}=(1, 1)^T$ as well as take the  mesh size  as $h=1/64$  and $h_x=h_y=1/16$ for $d=1$ and $2$, respectively. To quantify the numerical errors, we define the following error  functions:
\be
\label{Neror}
\begin{split}
&\breve{e}^{\ep}_\rho(t_k):=\rho(\cdot,t_k)-\rho^{\ep}(\cdot,t_k)=|u(\cdot, t_k)|^2-|u^\ep(\cdot, t_k)|^2, \\
&\breve{e}^{\ep}(t_k):=u(\cdot,t_k)-u^{\ep}(\cdot,t_k), \qquad  \Breve{\Breve{e}}^{\ep}(t_k):=u(\cdot, t_k)-u^{\ep,k}, \\
& e^{\ep}(t_k):=u^{\ep}(\cdot, t_k)-u^{\ep,k}, \qquad\;\;\;\; e_{E}^{\ep}:=|E(u_0)-E_{\rm reg}^{\ep}(u_0)|.
\end{split}
\ee
Here, $u$ and $u^\ep$ are the exact solutions of the LogSE \eqref{LSE} and RLogSE \eqref{RLSE_Unified}, respectively, while
$u^{\ep, k}$ is the numerical solution of the RLogSE \eqref{RLSE_Unified} obtained by LTSP \eqref{LT} (or \eqref{LT1}) or STSP \eqref{ST}.
The ``exact'' solution $u^{\ep}$  is obtained numerically by  STSP \eqref{ST} with a very small time step, e.g., $\tau=10^{-5}$.  %and a very fine mesh size, e.g., $h=1/64$ for $d=1$ and $h_x=h_y=1/32$ for $d=2$.
The energy is obtained by the trapezoidal rule for approximating the integrals in the energy  \eqref{conserv}, \eqref{RegL_Energ}, \eqref{Energ1} and \eqref{Energ2}.

\subsection{Convergence rate of the regularized model}
 Here, we consider the error between the solutions of the RLogSE \eqref{RLSE_Unified}  and the LogSE \eqref{LSE}.
 For various regularized models (i.e., different choices of regularized nonlinearity $f_{\rm reg}^\ep$ in equation \eqref{RLSE_Unified}), 
 Fig. \ref{fig:ModeL_Conv_Rate}  shows $\|\breve{e}^{\ep}(t)\|_{H^1}$
 % $\|\breve{e}^{\ep}(3)\|$,  $\|\breve{e}^{\ep}(3)\|_{\infty}$
 and $\|\breve{e}_\rho^{\ep}(t)\|_1$ at $t=3$ and  $t=2$, respectively, for $d=1$ and $2$,
while Fig. \ref{fig:ModeL_Conv_Rate_Total_Energy} depicts $e_{E}^{\ep}$ versus $\ep$.
The results are similar when $\breve{e}^{\ep}(t)$ is measured by $L^2$- or $L^\infty$-norm.

\begin{figure}[htbp!]
\begin{center}
\includegraphics[width=6cm,height=4cm]{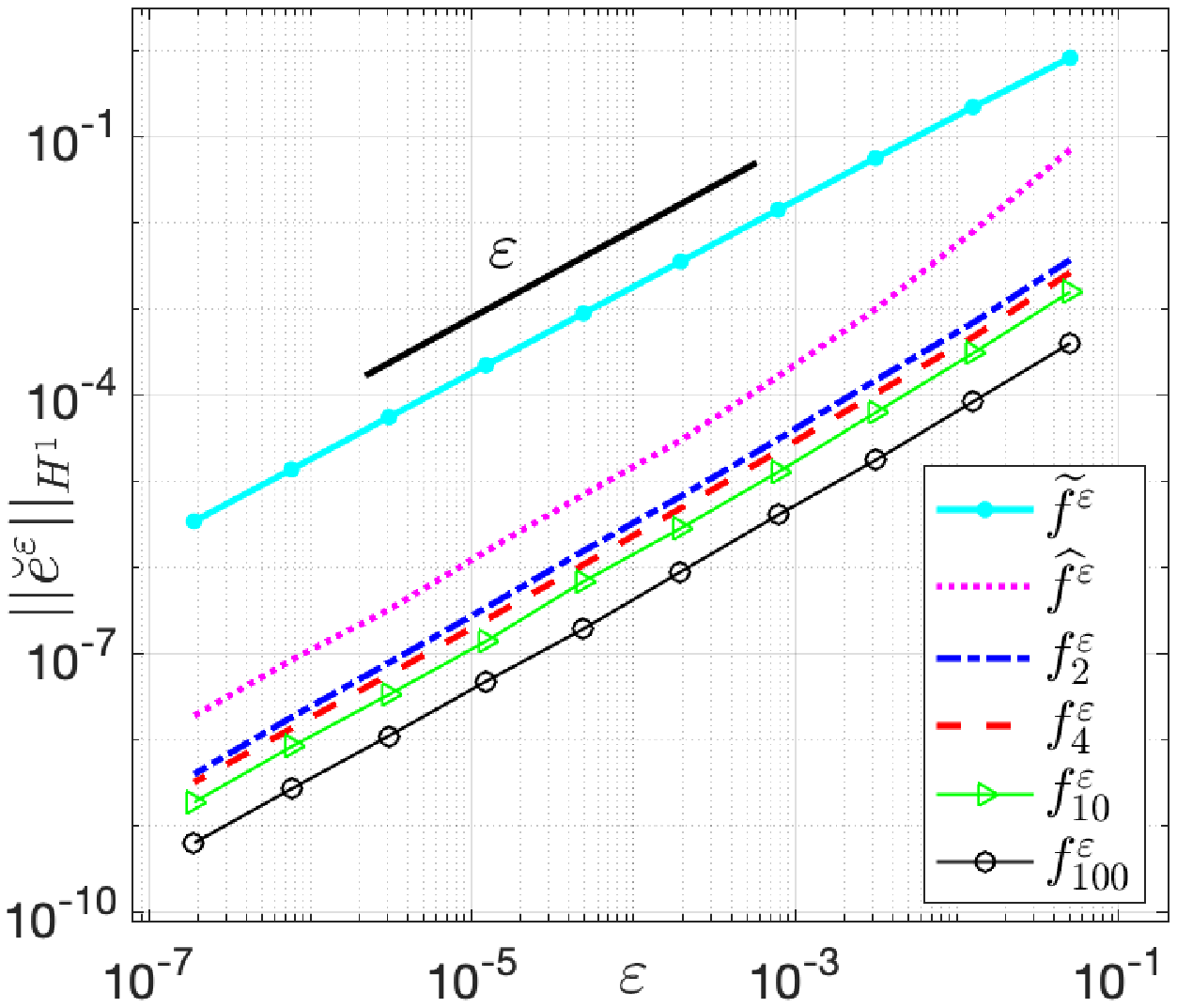}
\quad
%\includegraphics[width=6cm,height=4cm]{./figs/ModeL_Conv/ModeL_Converg_Rate_ALL_ReguL_FunN_L2_norm_At_t3.eps}\\[1em]
%\includegraphics[width=6cm,height=4cm]{./figs/ModeL_Conv/ModeL_Converg_Rate_ALL_ReguL_FunN_Max_norm_At_t3.eps}
%\quad
\includegraphics[width=6cm,height=4cm]{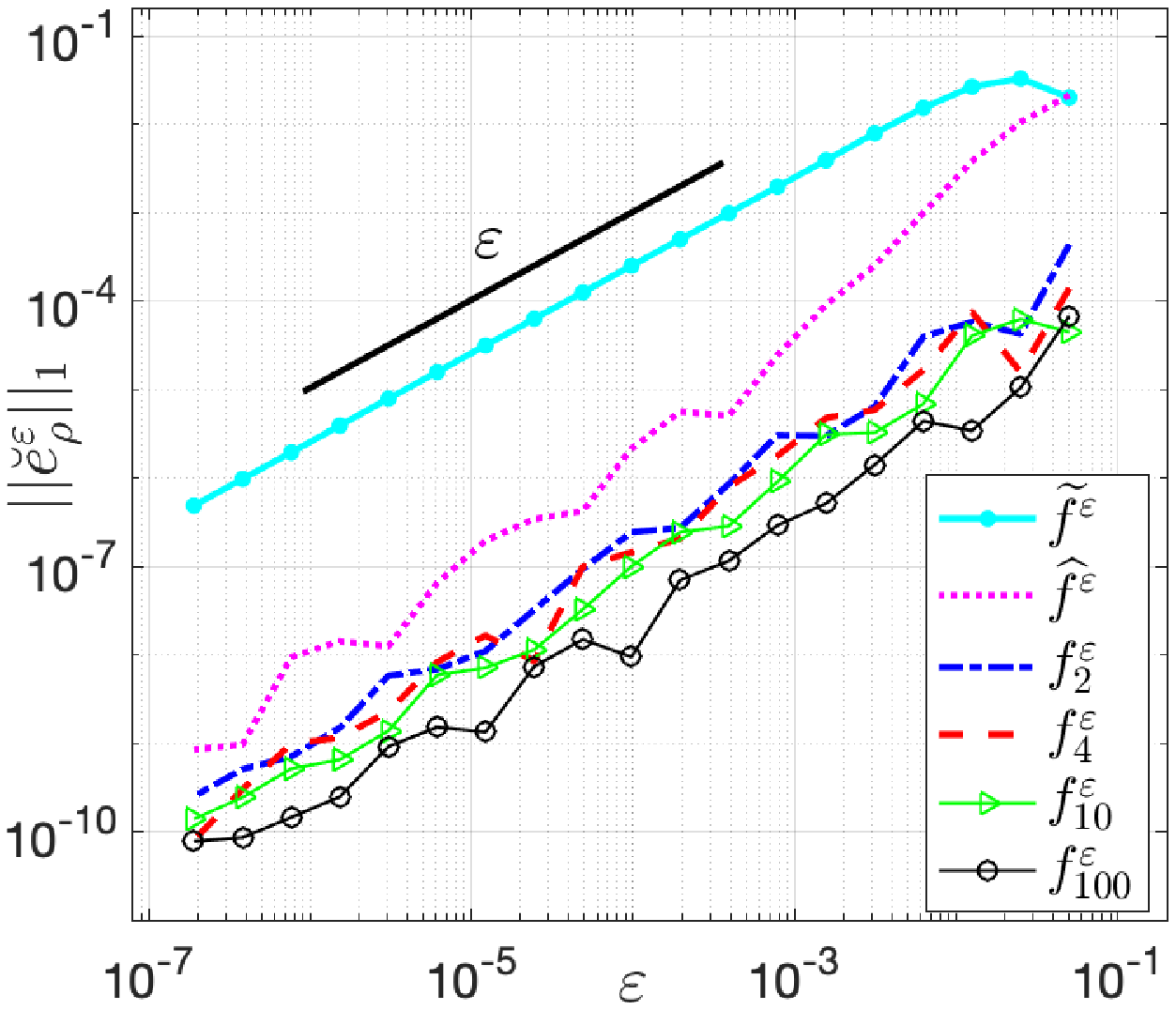}\\[1em]
\includegraphics[width=6cm,height=4cm]{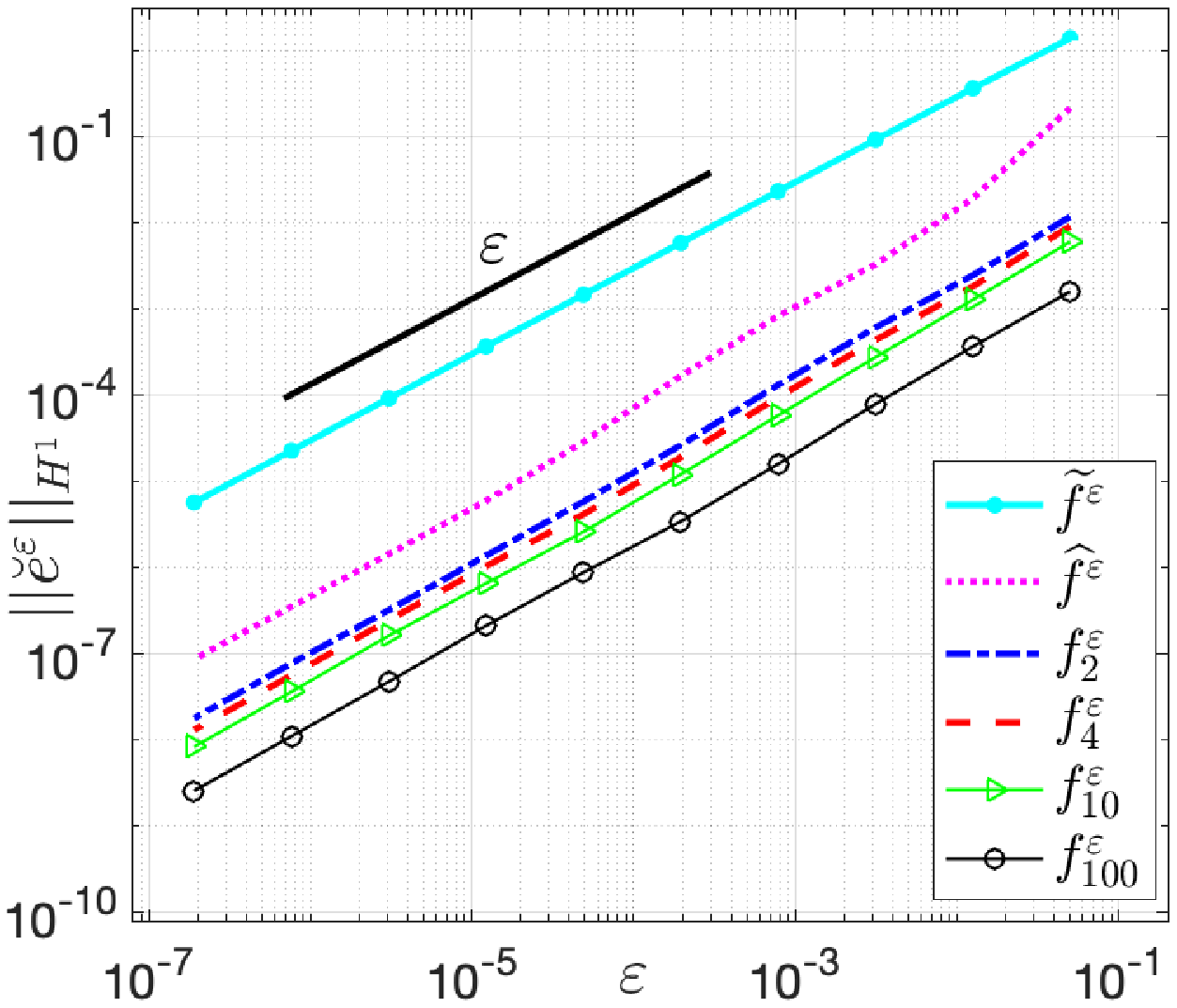}
\quad
\includegraphics[width=6cm,height=4cm]{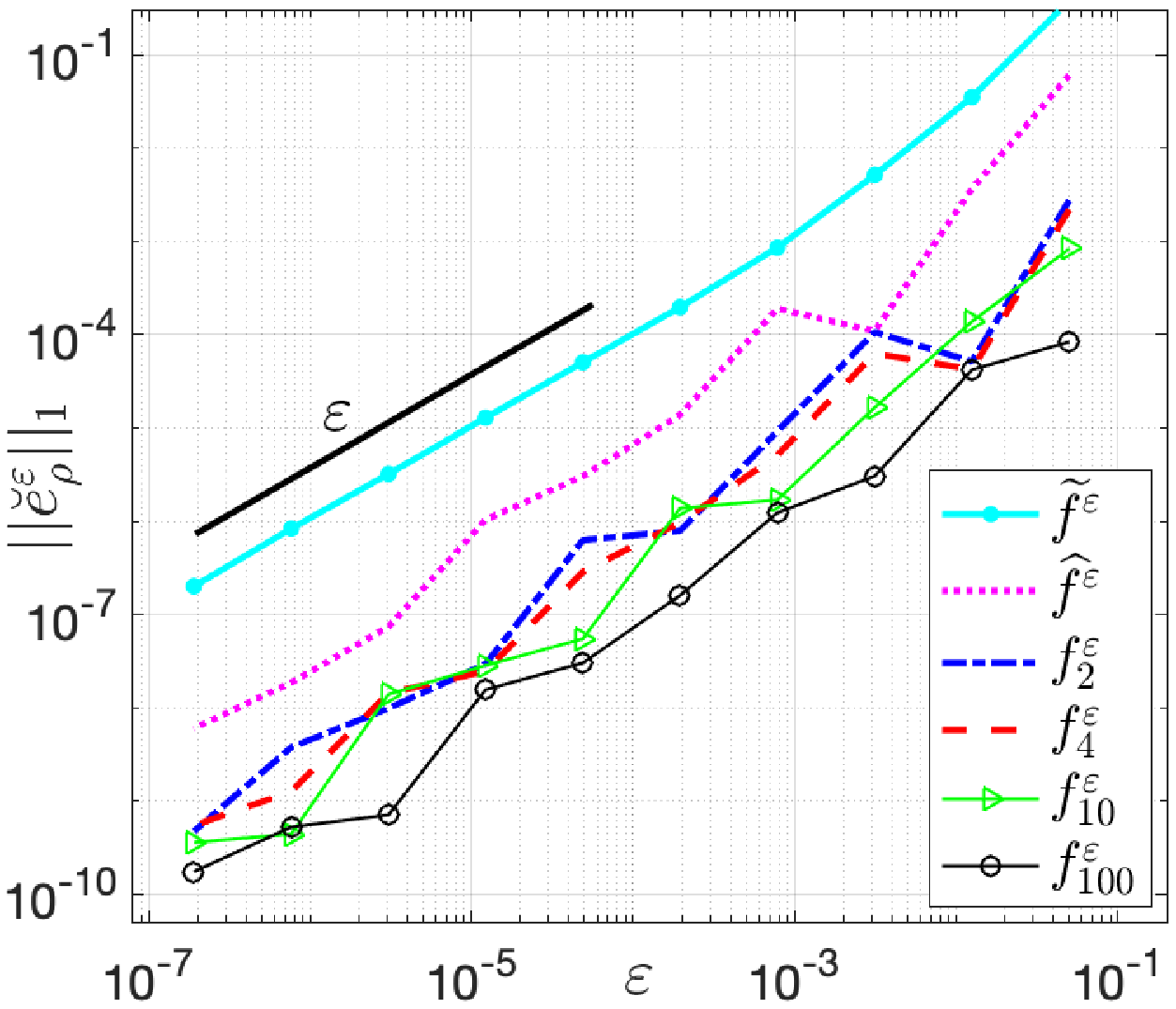}
\end{center}
 \caption{Convergence of the RLogSE \eqref{RLSE_Unified} with various regularized  nonlinearities $f_{\rm reg}^\ep$ to
 the LogSE \eqref{LSE}, i.e., the error $\|\breve{e}^\varepsilon(t)\|_{H^1}$ and $\|\breve{e}^\varepsilon_\rho(t)\|_1$ versus the regularization parameter $\ep$ at $t=3$ for $d=1$ (upper)  and $t=2$ for $d=2$ (lower). }
\label{fig:ModeL_Conv_Rate}
\end{figure}

From these figures and additional
similar numerical results not shown here for brevity, we could clearly see:
(i) The solution of the RLogSE \eqref{RLSE_Unified}  converges linearly to that of the LogSE \eqref{LSE} in terms of $\ep$ for all the three types of regularized models.
Moreover, the regularized energy $\widetilde{E}^\ep$ converges linearly to the original energy $E$ in terms of
$\ep$, while $\widehat{E}^\ep$ \& $E_n^\ep$ (for any $n\ge 2$) converges quadratically. These results confirm the theoretical results from  Section \ref{sec:cvmodel} \& \ref{sec:cvenergy}.
(ii) In $L^1$-norm, the density $\rho^\ep$ of the solution of the RLogSE with regularized nonlinearity $\widetilde{f}^\ep$  converges linearly to that of the LogSE \eqref{LSE} in terms of $\ep$, while the convergence rate is not clear for those of RLogSE with other
regularized nonlinearities.  Generally, for fixed $\ep$, the errors of the densities measured in $L^1$-norms are smaller than  those of wave functions (measured in $L^2$, $H^1$ or $L^\infty$-norm).
(iii) For any fixed $\ep>0$, the proposed local energy regularization (i.e., $f_{\rm reg}^\ep=f_n^\ep$) outperforms the  other two
(i.e., $f_{\rm reg}^\ep=\widehat{f}^\ep$ and $f_{\rm reg}^\ep=\widetilde{f}^\ep$) in the sense that its corresponding errors in wave function and total energy
are smaller.  The larger the order (i.e., $n$) of the  energy-regularization is chosen, the smaller the difference between the solutions of the ERLogSE \eqref{ERLSE}
 and LogSE is obtained.

\begin{figure}[h!]
\begin{center}
\includegraphics[width=6cm,height=5cm]{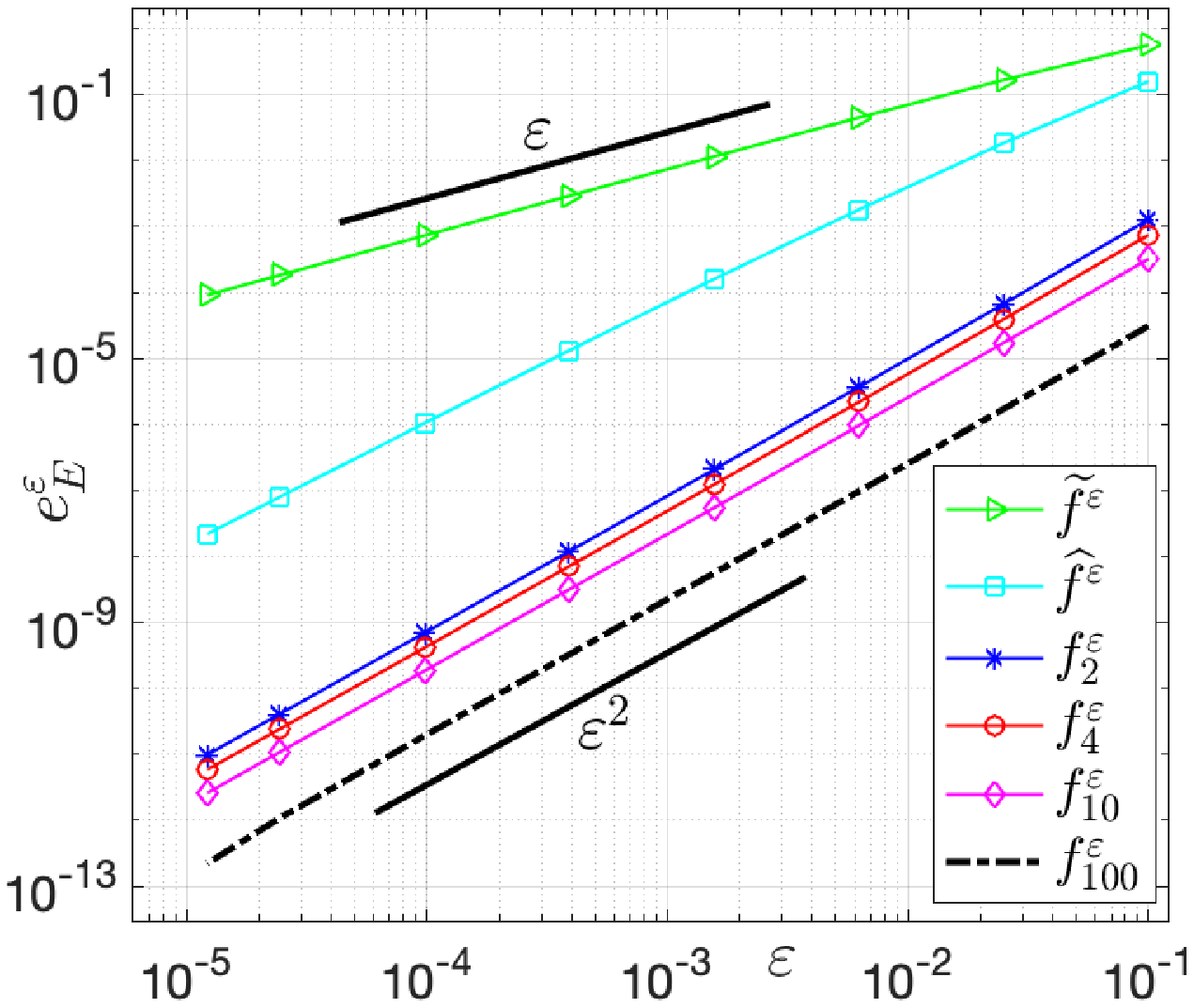}\quad
\includegraphics[width=6cm,height=5cm]{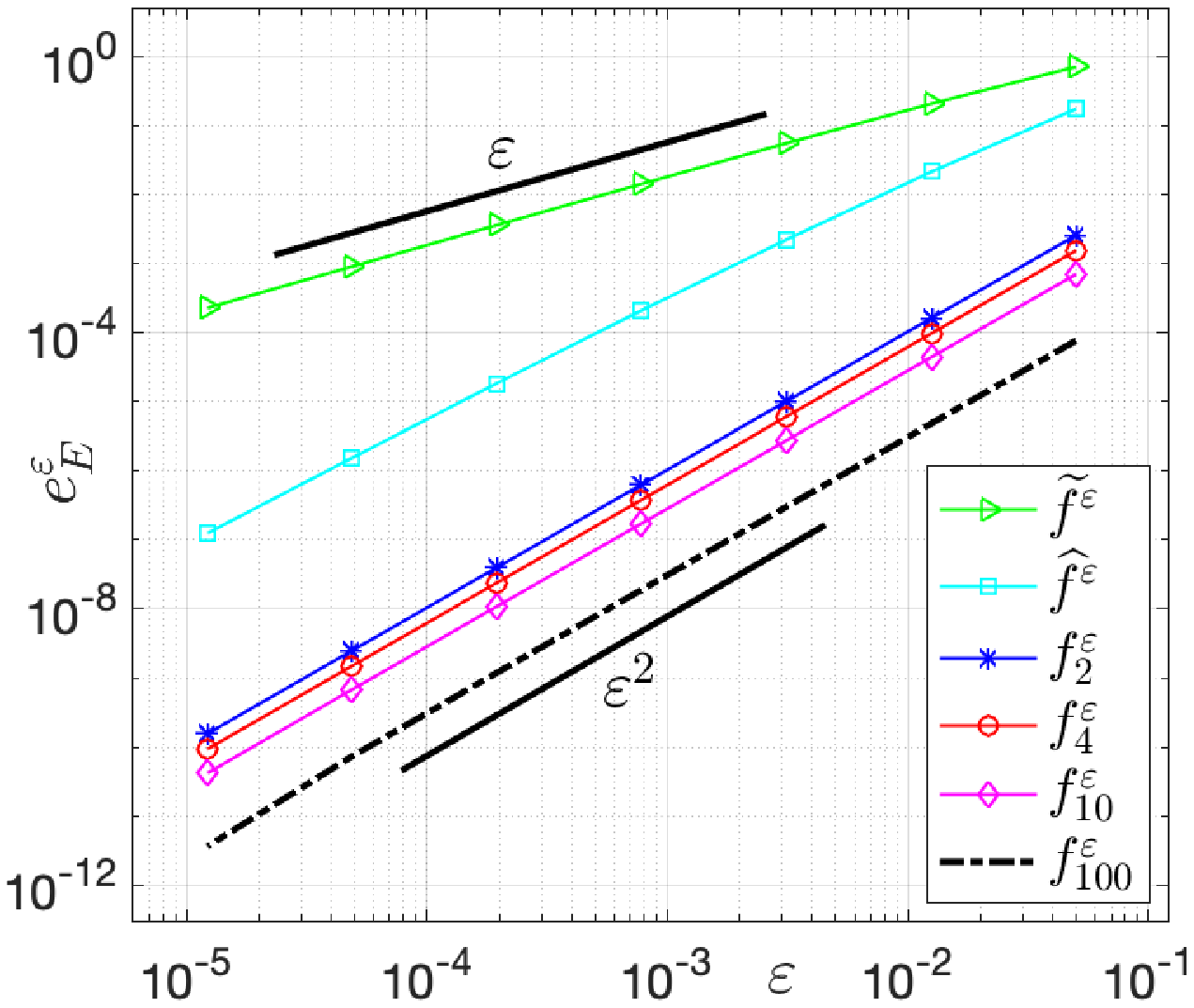}
\end{center}
 \caption{Convergence of the RLogSE \eqref{RLSE_Unified} with various regularized nonlinearities $f_{\rm reg}^\ep$
    to the LogSE \eqref{LSE}: the energy error
 $e_{E}^\ep(t)$ \eqref{Neror} at $t=3$ for $d=1$ (left)  and $t=2$ for $d=2$ (right).}
\label{fig:ModeL_Conv_Rate_Total_Energy}
\end{figure}

\subsection{Convergence rate of the time-splitting spectral method}
Here, we investigate the model RLogSE \eqref{RLSE_Unified} with  $f_{\rm reg}^\ep=f_n^\ep$, i.e., the ERLogSE \eqref{ERLSE}.
We will test the convergence rate of type-1 LTSP \eqref{LT} \& type-2 LTSP \eqref{LT1} and the STSP \eqref{ST} to the ERLogSE \eqref{ERLSE}
%with different regularization order  $n$
or the LogSE \eqref{LSE} in terms of the time step $\tau$ for fixed $\ep\in(0,1)$.
Fig. \ref{fig:Order_Accuracy_LT_ST} shows the errors $\|e^\varepsilon(3)\|_{H^1}$
versus time step $\tau$ for  $f_2^\ep$ \& $f_4^\ep$. In addition, Table \ref{tab:conv_STSP_Energy_ReguL_f2}
displays $\|\breve{\breve{e}}^{\ep}(3)\|$ versus  $\ep$ \& $\tau$ for  $f_2^\ep$.

\begin{figure}[h!]
\begin{center}
\includegraphics[width=6cm,height=5cm]{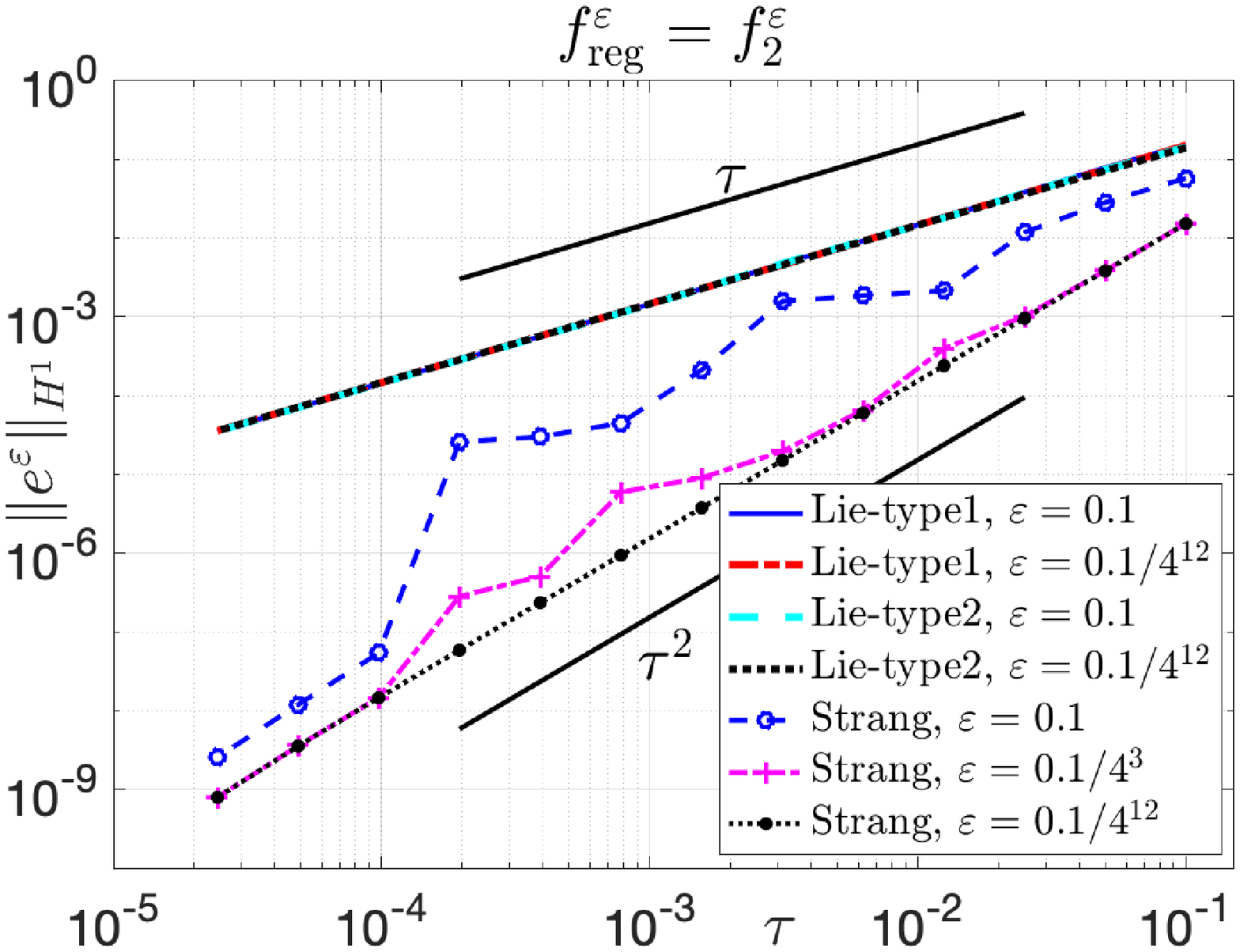}
\quad
\includegraphics[width=6cm,height=5cm]{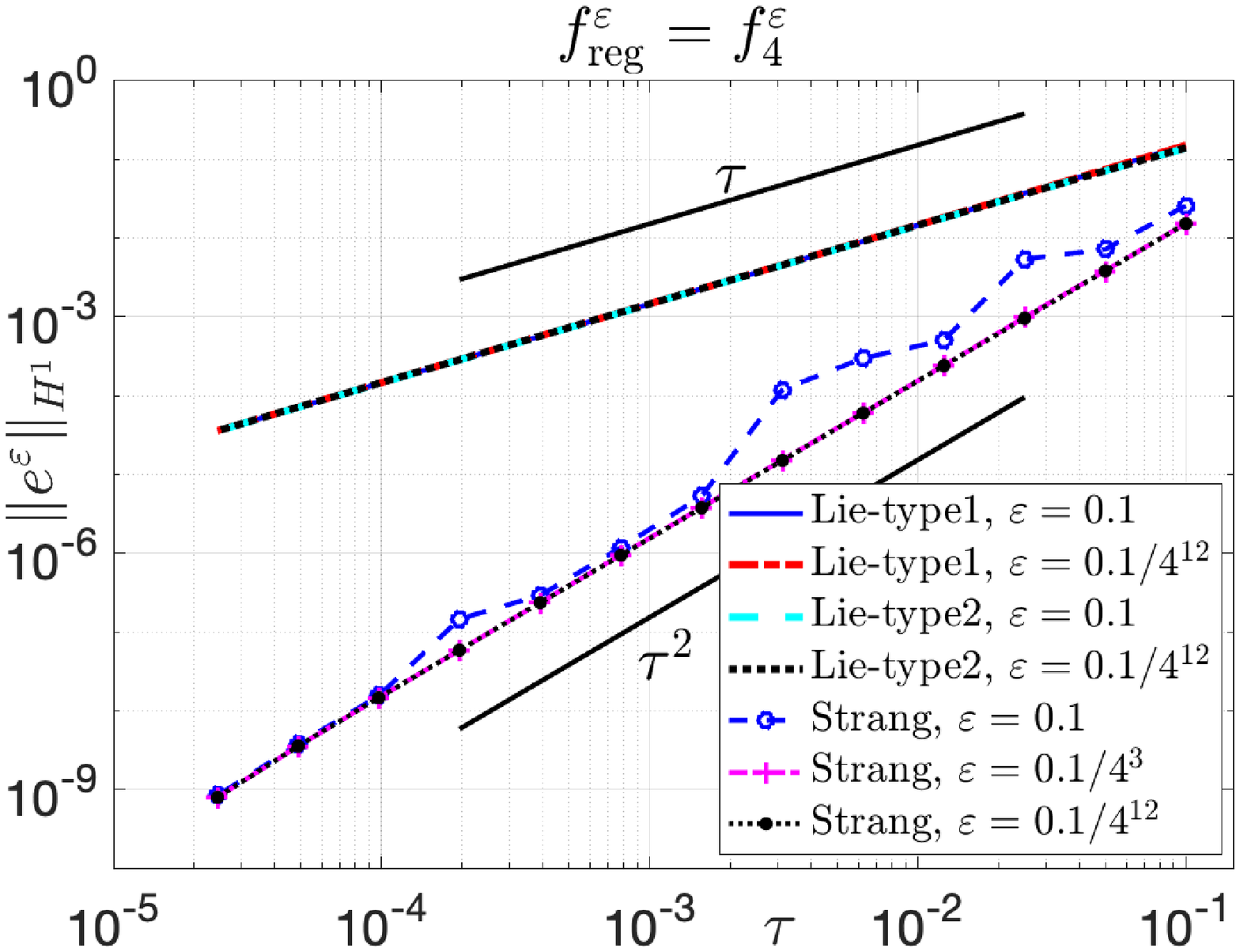}
\end{center}
 \caption{Convergence of the type-1 LTSP \eqref{LT} \&  type-2  LTSP \eqref{LT1}  as well as the STSP \eqref{ST} to the ERLogSE \eqref{ERLSE} with  regularized nonlinearity $f_2^\ep$ (left)
and $f_4^\ep$ (right), i.e., errors $\|e^\varepsilon(3)\|_{H_1}$  versus $\tau$ for various $\ep$.}
\label{fig:Order_Accuracy_LT_ST}
\end{figure}

 From Fig. \ref{fig:Order_Accuracy_LT_ST}, Table \ref{tab:conv_STSP_Energy_ReguL_f2} and additional
 similar results not shown here for brevity, we can observe that:
 (i) In $H^1$  norm, for any fixed $\ep\in(0, 1)$ and $n\ge 2$, the LTSP scheme converges linearly while the STSP scheme converges quadratically when $\ep<\ep_0$ for some $\ep_0>0$. (ii)  For  any  $f_n^\ep$ with $n\ge2$, the STSP converges quadratically to the LogSE \eqref{LSE}  only when $\ep$ is sufficiently small, i.e., $\ep\lesssim \tau^2$ (cf. each row in the lower triangle below the diagonal in bold letter in Table \ref{tab:conv_STSP_Energy_ReguL_f2}).
(iii) When  $\tau$  is sufficiently small, i.e., $\tau^2\lesssim \ep$,  the ERLogSE  \eqref{ERLSE}  converges linearly  at $O(\ep)$ to
 the LogSE \eqref{LSE} (cf.  each column in the upper triangle above the diagonal in bold letter in Table
  \ref{tab:conv_STSP_Energy_ReguL_f2}).
 (iv) The numerical results are similar for other $f_n^\ep$ with $n\ge3$ and when the errors are measured  in  $L^\infty$- and $L^2$-norm,
  which confirm the theoretical conclusion in Theorem \ref{thmlt2} and Remark \ref{rem:ST_Error}.

\begin{table}[htbp!]
\footnotesize
\tabcolsep 0pt
\caption{Convergence of the STSP \eqref{ST} (via solving the ERLogSE \eqref{ERLSE} with $f_2^\ep$) to the LogSE \eqref{LSE},
i.e., $\|\breve{\breve{e}}^{\ep}(3)\|$ for different $\ep$ and $\tau$. }
\label{tab:conv_STSP_Energy_ReguL_f2}
\begin{center}\vspace{-0.5em}
\def\temptablewidth{1\textwidth}
{\rule{\temptablewidth}{1pt}}
\begin{tabularx}{\temptablewidth}{@{\extracolsep{\fill}}p{1.38cm}|cccccccccc}
 & $\tau=0.1$  & $\tau/2$    & $\tau/2^2$   & $\tau/2^3$  & $\tau/2^4$  & $\tau/2^5$ &  $\tau/2^6$ &  $\tau/2^7$ &  $\tau/2^8$ &  $\tau/2^{9}$  \\[0.3em]
\hline
$\ep$=0.025 &7.98E-3 & \bf{2.13E-3 }& 8.86E-4 & 7.28E-4 & 7.14E-4 & 7.12E-4 & 7.12E-4 & 7.12E-4 & 7.12E-4 & 7.12E-4 \\ [0.25em]
rate 	       &    --   & \bf{1.91 }& 1.27 & 0.28 & 0.03 & 0.00 & 0.00 & 0.00 & 0.00 & 0.00     \\  [0.25em]
\hline
$\ep/4$ &      7.77E-3 &  1.96E-3 & \bf{5.02E-4 }& 1.67E-4 & 1.12E-4 & 1.08E-4 & 1.08E-4 & 1.08E-4 & 1.08E-4 & 1.08E-4   \\  [0.25em]
rate	   &   --    &1.99 & \bf{1.97} & 1.59 & 0.57 & 0.06 & 0.01 & 0.00 & 0.00 & 0.00            \\[0.25em]
\hline
$\ep/4^2$ &  7.76E-3 & 1.95E-3 & 4.88E-4 & \bf{1.25E-4 }& 3.81E-5 & 2.40E-5 & 2.28E-5 & 2.27E-5 & 2.27E-5 & 2.27E-5 	  \\  [0.25em]
rate	   &  --		    &   2.00    & 2.00        & \bf{1.97} & 1.71 & 0.67 & 0.07 & 0.01 & 0.00 & 0.00              \\[0.25em]
\hline
$\ep/4^3$ &   7.76E-3  & 1.95E-3 & 4.87E-4 & 1.22E-4 & \bf{3.08E-5 }& 8.95E-6 & 5.09E-6 & 4.74E-6 & 4.72E-6 & 4.71E-6 \\  [0.25em]
rate	   &     --  &   2.00 & 2.00 & 2.00 & \bf{1.98} & 1.78 & 0.82 & 0.10 & 0.01 & 0.00            \\[0.25em]
\hline
$\ep/4^4$ &7.76E-3 & 1.95E-3 & 4.87E-4 & 1.22E-4 & 3.04E-5 & \bf{7.66E-6} & 2.092E-6 & 9.93E-7 & 8.80E-7 & 8.72E-7\\  [0.25em]
rate	   & -- &2.00 &  2.00 & 2.00 & 2.00 & \bf{1.99} & 1.87 & 1.08 & 0.18 & 0.01 \\[0.25em]
\hline
$\ep/4^5$ &7.76E-3  & 1.95E-3   & 4.87E-4 & 1.22E-4 & 3.04E-5 & 7.61E-6 & \bf{1.92E-6} & 5.26E-7 & 2.54E-7 & 2.27E-7\\  [0.25em]
rate	   & -- &2.00 & 2.00 & 2.00 & 2.00 & 2.00 & \bf{1.99} & 1.87 & 1.05 & 0.16 \\[0.25em]
\hline
$\ep/4^6$ & 7.76E-3  & 1.95E-3   & 4.87E-4 & 1.22E-4 & 3.04E-5 & 7.61E-6 & 1.90E-6 & \bf{4.78E-7}  &   1.27E-7   &  5.36E-8  \\  [0.25em]
rate	   & -- &2.00 & 2.00 & 2.00 & 2.00 & 2.00 & 2.00 &\bf{ 1.99} & 1.91 & 1.25  \\[0.25em]
\hline
$\ep/4^7$ &  7.76E-3  & 1.95E-3   & 4.87E-4 & 1.22E-4 & 3.04E-5 & 7.61E-6 & 1.90E-6 & 4.76E-7  &   \bf{1.19E-7}   &  3.13E-8   \\  [0.25em]
rate	   & -- &2.00 & 2.00 & 2.00 & 2.00 & 2.00 & 2.00 & 2.00 & \bf{2.00 }& 1.93    \\[0.25em]
\hline
$\ep/4^8$ &  7.76E-3  & 1.95E-3   & 4.87E-4 & 1.22E-4 & 3.04E-5 & 7.61E-6 & 1.90E-6 & 4.76E-7  &   1.19E-7   &  \bf{2.98E-8}   \\  [0.25em]
rate	   &  -- &2.00 & 2.00 & 2.00 & 2.00 & 2.00 & 2.00 & 2.00 & 2.00 & \bf{2.00 }   \\[0.25em]
\end{tabularx}
{\rule{\temptablewidth}{1pt}}
\end{center}
\end{table}

\subsection{Application for  interaction of 2D Gaussons}
In this section, we apply the STSP method to investigate
the interaction of Gaussons in dimension 2. To this end, we fix  $n=4$, $\varepsilon=10^{-12}$,  $\tau=0.001$, $h_x=h_y=1/16$,  $\Omega=[-16, 16]^2$ for {\bf Case 1} \& {\bf Case 2} while $\Omega=[-48, 48]^2$ for {\bf Case 3}.
The initial data is chosen as
\be
u_0(\bx)=b_1 e^{i\bx\cdot\bm{v}_1 +\frac{\lambda}{2}|\bx-\bx_1^0|^2}
+b_2 e^{i\bx\cdot\bm{v}_2 +\frac{\lambda}{2}|\bx-\bx_2^0|^2},
\ee
where $b_j$, $\bm{v}_j$ and $\bx_j^0$ ($j=1,2$)  are real   constant vectors, i.e., the initial data is the sum of two Gaussons \eqref{Gausson} with velocity $\bm{v}_j$ and initial location  $\bx_j^0$. Here, we consider the following   cases:
\begin{itemize}
\item[(i)] $b_1=b_2=\fl{1}{\sqrt[4]{\pi}}$,\quad   $\bm{v}_1=\bm{v_2}=(0, 0)^T$,\quad  $\bx_1^0=-\bx_2^0=(-2, 0)^T$;
\smallskip
\item[(ii)] $b_1= 1.5\,b_2 =\fl{1}{\sqrt[4]{\pi}}$,\quad   $\bm{v}_1=(-0.15, 0)^T$,\quad $\bm{v_2}=\bx_1^0=(0, 0)^T$,\quad $\bx_2^0=(5, 0)^T$;
\smallskip
\item[(iii)] $b_1=b_2 =\fl{1}{\sqrt[4]{\pi}}$,\quad   $\bm{v}_1=(0, 0)^T$,\quad $\bm{v_2}=(0, 0.85)^T$,\quad $\bx_1^0=-\bx_2^0=(-2, 0)^T$.
\end{itemize}
Fig. \ref{fig:Rev_2D_Gau_Inter_Case1_2} shows the contour  plots of $|u^\varepsilon(x,y,t)|^2$ at different time as well as the evolution of $\sqrt{|u^\varepsilon(x,0,t)|}$ for {\bf Case} (i) \& (ii). While Fig. \ref{fig:Rev_2D_Gau_Inter_Case3}  illustrates that for {\bf Case} (iii).
 From these figures we clearly see that: (1) Even for two static
 Gaussons, if they stay close enough, they will contact and undergo
 attractive interactions. They will collide and stick together shortly
 then separate again. The Gaussons will swing like a pendulum and
 small solitary waves are emitted outward during the interaction
 (cf. Fig.~\ref{fig:Rev_2D_Gau_Inter_Case1_2} top). This dynamics phenomena is similar to that in 1D case \cite{bao2018}. (2) For Case (ii), the two Gaussons also undergo  attractive interactions. The slowly moving Gausson  will drag  its  nearby static Gausson to move in the same direction (cf. Fig.~\ref{fig:Rev_2D_Gau_Inter_Case1_2} bottom), which is also similar to that in 1D case \cite{bao2018}.  (3) For two Gaussons (one static and the other moving) staying close enough, if the moving Gausson move perpendicular to the line connecting the two Gaussons, the static Gausson will be dragged to move and the direction of the moving Gausson will be altered.  The two Gaussons will rotate with each other and gradually drift away, which is similar to the dynamics of a vortex pair in the cubic Schr\"odinger equation \cite{Bao2014}.

\begin{figure}[h!]
\begin{center}
\includegraphics[width=3cm,height=2.5cm]{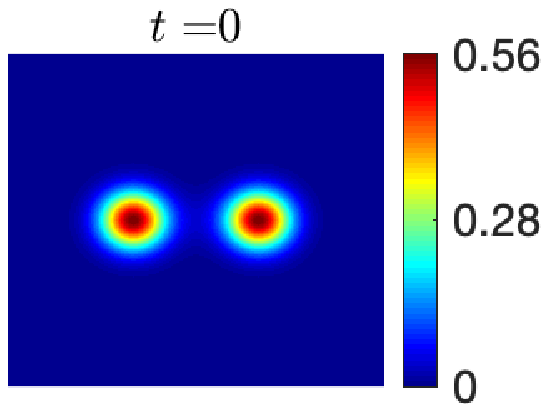}
\includegraphics[width=3cm,height=2.5cm]{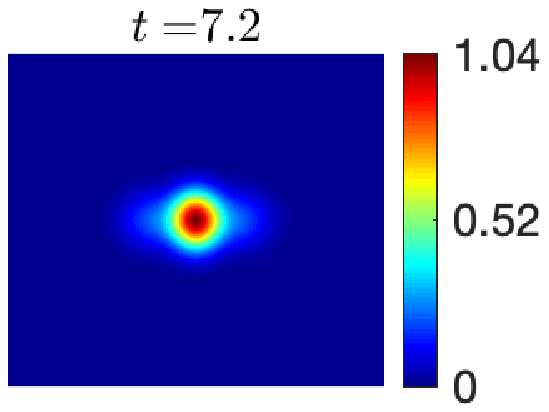}
\includegraphics[width=3cm,height=2.5cm]{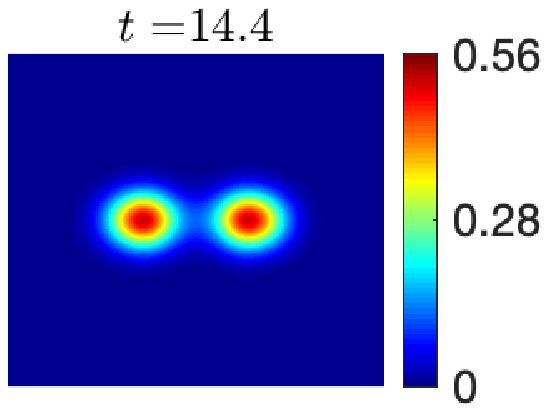}
\includegraphics[width=3cm,height=2.5cm]{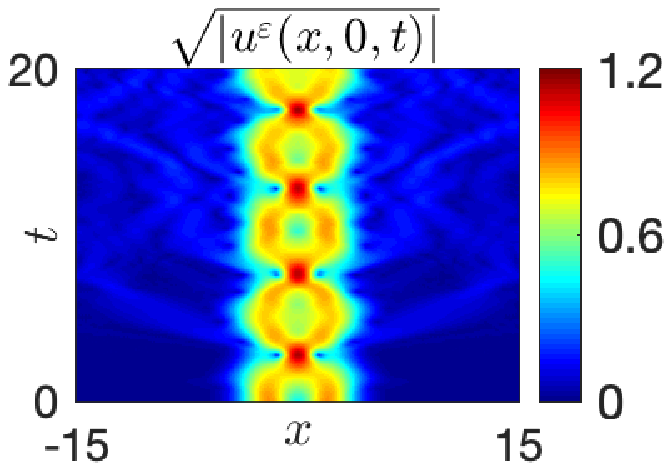}\\[1em]
\includegraphics[width=3cm,height=2.5cm]{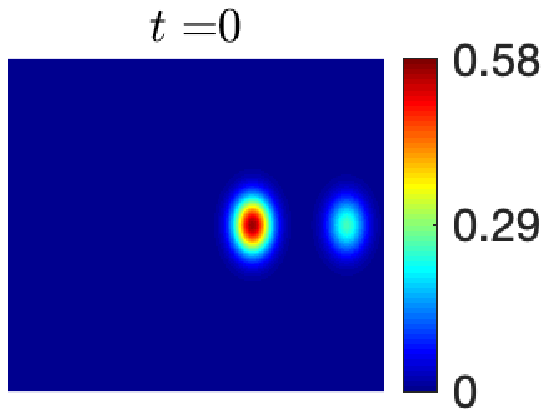}
\includegraphics[width=3cm,height=2.5cm]{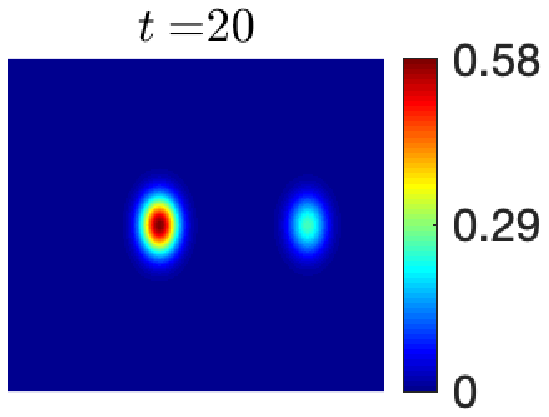}
\includegraphics[width=3cm,height=2.5cm]{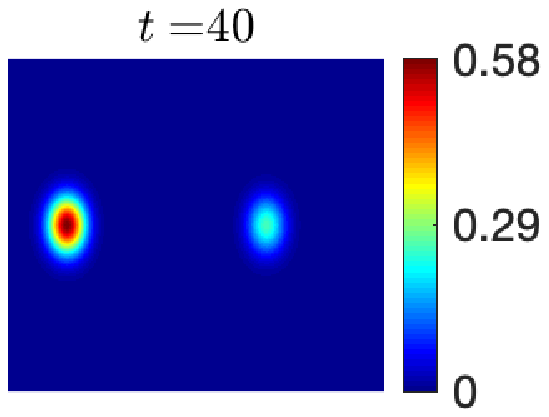}
\includegraphics[width=3cm,height=2.5cm]{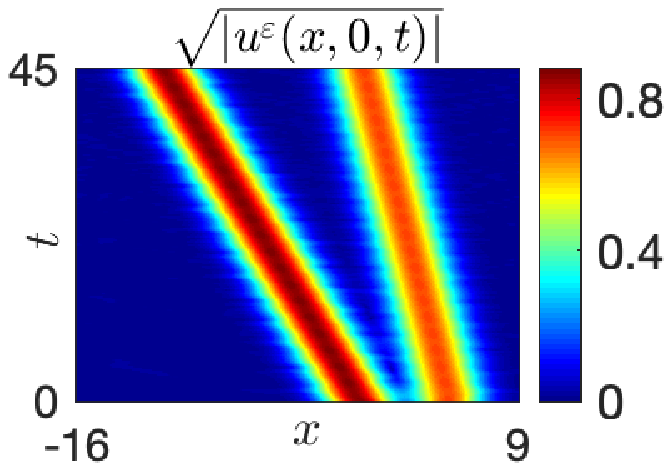}\\
\end{center}
 \caption{Plots of $|u^\varepsilon(x,y,t)|^2$ at different times (first three column) and contour plot of  $|u^\varepsilon(x,0,t)|^2$ (last column)
for {\bf Case} (i) (Upper) in region $[-6, 6]^2$  and {\bf Case} (ii) (Lower) in  region $[-13, 7]\times[-6, 6]$.}
 \label{fig:Rev_2D_Gau_Inter_Case1_2}
\end{figure}

\begin{figure}[h!]
\begin{center}
\includegraphics[width=3cm,height=2.5cm]{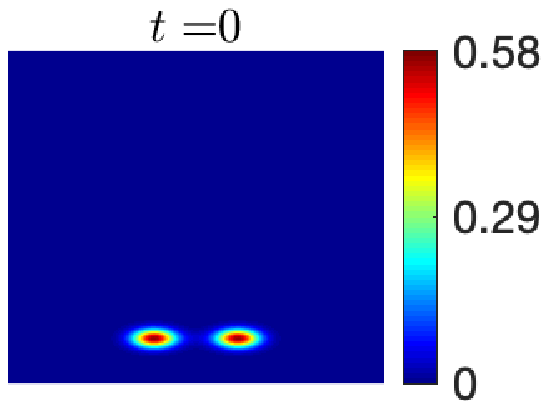}
\includegraphics[width=3cm,height=2.5cm]{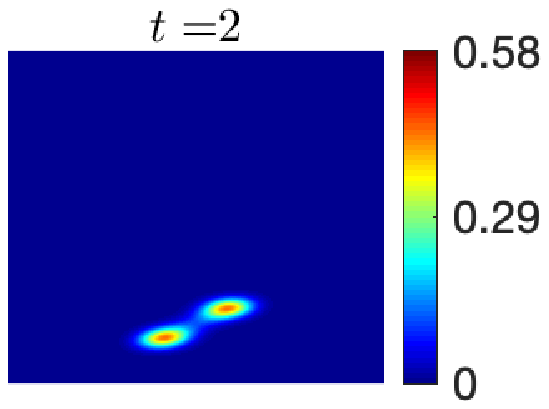}
\includegraphics[width=3cm,height=2.5cm]{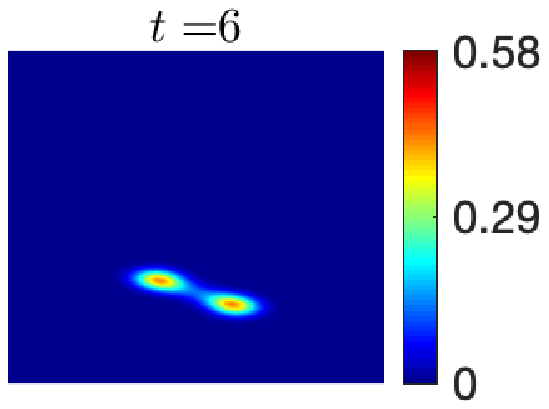}
\includegraphics[width=3cm,height=2.5cm]{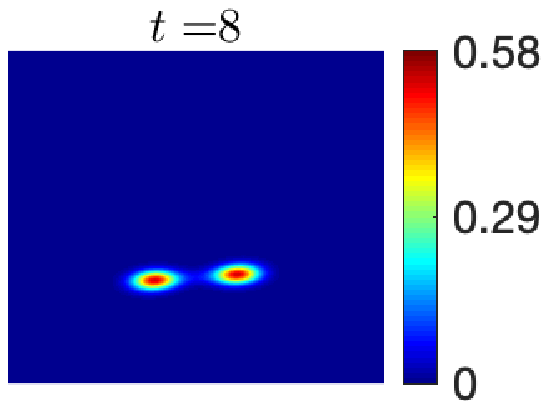}\\[1em]
\includegraphics[width=3cm,height=2.5cm]{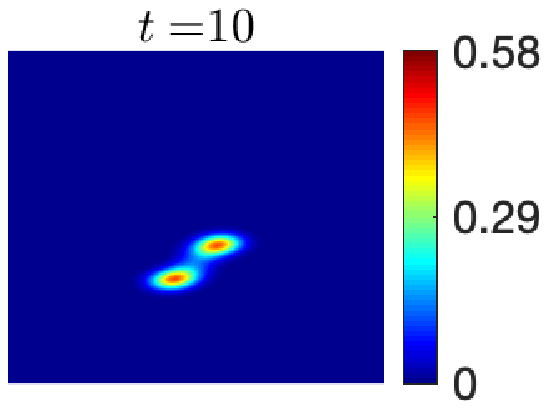}
\includegraphics[width=3cm,height=2.5cm]{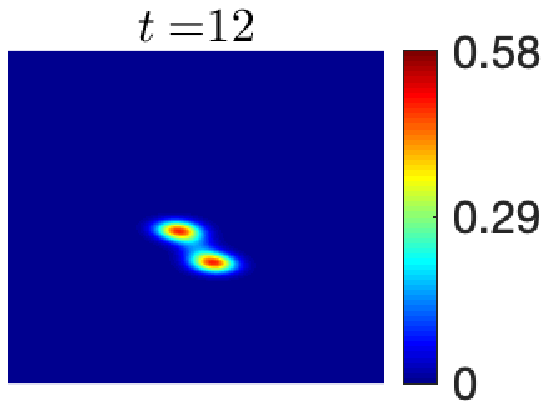}
\includegraphics[width=3cm,height=2.5cm]{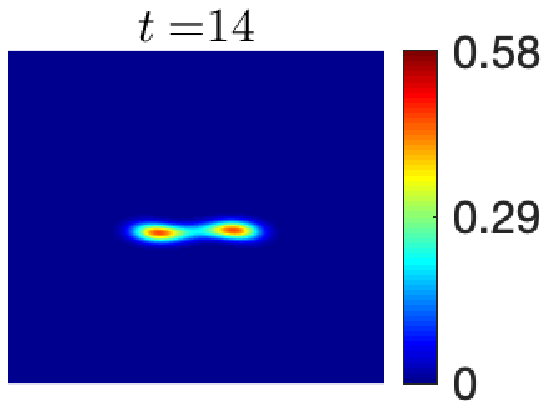}
\includegraphics[width=3cm,height=2.5cm]{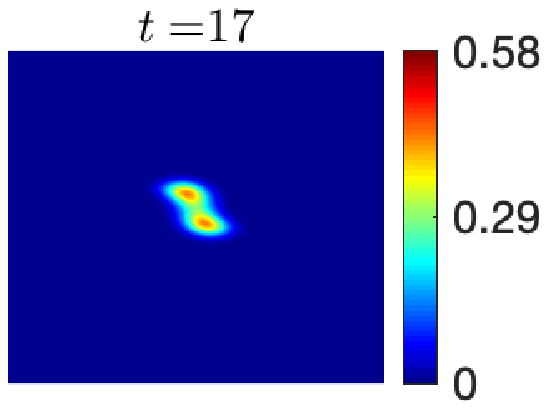}\\[1em]
\includegraphics[width=3cm,height=2.5cm]{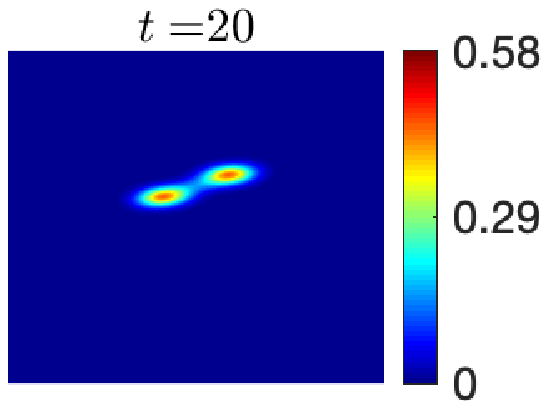}
\includegraphics[width=3cm,height=2.5cm]{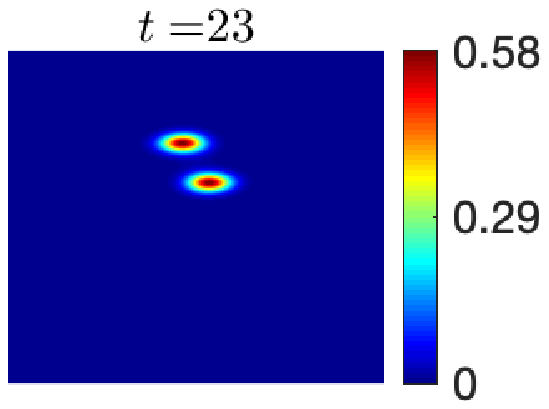}
\includegraphics[width=3cm,height=2.5cm]{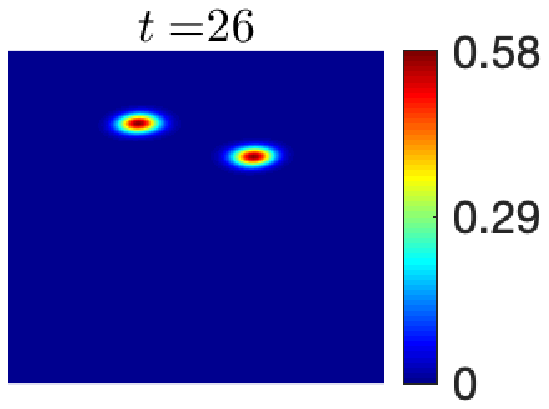}
\includegraphics[width=3cm,height=2.5cm]{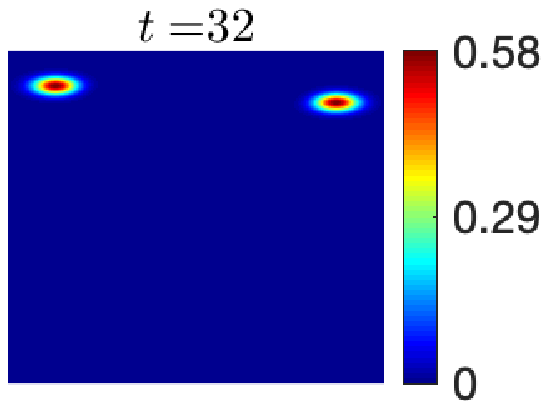}\\[1em]
\end{center}
 \caption{Plots of $|u^\varepsilon(x,y,t)|^2$ at different times for {\bf Case} (iii) in region $[-9, 9]\times [-5, 32]$.}
  \label{fig:Rev_2D_Gau_Inter_Case3}
\end{figure}

\section{Conclusion}
We proposed  a new systematic local energy regularization (LER) approach  to overcome the singularity of the nonlinearity in the logarithmic Schr\"{o}dinger equation (LogSE).
With a small regularized parameter $0<\ep\ll1$, in contrast to the existing ones that directly regularize the logarithmic  nonlinearity,
we regularized locally the interaction energy
density in the energy functional of the LogSE. The Hamiltonian flow of the new regularized energy
then yields an energy regularized logarithmic Schr\"{o}dinger equation
(ERLogSE). Linear and quadratic convergence in
terms of $\ep$ was established between
the solutions, and between the conserved total energy  of  ERLogSE and LogSE, respectively.
Then we presented and analyzed  time-splitting schemes to solve the  ERLogSE. The classical first order  of convergence  was obtained  both theoretically and numerically for the Lie-Trotter splitting scheme.
Numerical results suggest that the error bounds of splitting schemes to the LogSE clearly depend on the time step $\tau$
and mesh size $h$ as well as the small regularized parameter $\ep$.  Our numerical results confirm the error bounds and
indicate that the ERLogSE model outperforms the other existing ones in accuracy.

\section*{Acknowledgment}
This work was partially supported by the Ministry
of Education of Singapore grant R-146-000-296-112
(MOE2019-T2-1-063) (W. Bao), Rennes M\'etropole through its AIS
program (R. Carles), the Alexander von Humboldt
Foundation (C. Su), the Institutional Research Fund from Sichuan University (No. 2020SCUNL110) and the National Natural Science Foundation of China (No. 11971335) (Q. Tang).

%=============================================================================
%                                                                          REFERENCES
% =============================================================================
%\section*{References}

\end{document}